\documentclass[11pt,a4paper]{amsart}

\def\baselinestretch{1.2}
\usepackage{graphicx,amssymb}
\usepackage[top=40mm,bottom=35mm,left=30mm,right=30mm]{geometry}
\usepackage[frame]{xy}
\xyoption{arc}
\xyoption{graph}
\xyoption{matrix}

\theoremstyle{plain}
\newtheorem{theorem}{Theorem}[section]
\newtheorem{lemma}[theorem]{Lemma}
\newtheorem{corollary}[theorem]{Corollary}
\newtheorem{proposition}[theorem]{Proposition}

\theoremstyle{definition}
\newtheorem{definition}[theorem]{Definition}
\newtheorem{example}[theorem]{Example}
\newtheorem{remark}[theorem]{Remark}

\newtheorem{conjecture}[theorem]{Conjecture}

\makeatletter
\newif\if@my@Hidebox \@my@Hideboxtrue
\newbox\@my@Hidebox
\def\myHidebox{\setbox\@my@Hidebox\vbox\bgroup}
\def\endmyHidebox{\egroup
   \if@my@Hidebox
      \unvbox\@my@Hidebox\par
   \else\par
   \fi
}
\def\showmyHidebox{\@my@Hideboxtrue}
\def\hidemyHidebox{\@my@Hideboxfalse}
\def\showmyproof{\@my@Hideboxtrue}
\def\hidemyproof{\@my@Hideboxfalse}

\makeatother

\hidemyHidebox  
\showmyHidebox  

\def\mynewpage{\vskip 20mm}
\let\mynewpage\newpage
\let\mynewpage\relax

\let\le\leqslant
\let\ge\geqslant

\def\R{\mathbb R}

\def\Z{\mathbb Z}

\def\sb{sb}         
\let\splitBr\otimes
\def\PG{\operatorname{PG}}

\begin{document}

\title{Petal grid diagrams of torus knots}

\author{Eon-Kyung Lee and Sang-Jin Lee}
\address{Department of Mathematics and Statistics, Sejong University, Seoul, Korea}
\email{eonkyung@sejong.ac.kr}

\address{Department of Mathematics, Konkuk University, Seoul, Korea}
\email{sangjin@konkuk.ac.kr}
\date{\today}

\begin{abstract}
A petal diagram of a knot is a projection
with a single multi-crossing such that there are no nested loops.
The petal number $p(K)$ of a knot $K$ is the minimum number of loops
among all petal diagrams of $K$.

Let $T_{n,s}$ denote the $(n,s)$-torus knot
for relatively prime integers $2\le n<s$.
Recently, Kim, No and Yoo proved 
that $p(T_{n,s})\le 2s-2\left\lfloor \frac sn\right\rfloor+1$
whenever $s\equiv \pm 1\bmod n$.
They conjectured that the inequality holds without the assumption $s\equiv \pm 1\bmod n$.
They also showed that $p(T_{n,s})=2s-1$ whenever $2\le n<s<2n$
and $n\equiv 1\bmod s-n$.
Their proofs construct petal grid diagrams for those torus knots.

In this paper, we prove the conjecture that
$p(T_{n,s})\le 2s-2\left\lfloor \frac sn\right\rfloor+1$ holds for any $2\le n<s$.
We also show that $p(T_{n,s})=2s-1$ holds for any $2\le n<s<2n$.
Our proofs construct petal grid diagrams for any torus knots.

\medskip\noindent
{\em Keywords\/}: petal number, petal grid diagram, torus knot, braid group  \\
{\em MSC2020\/}: 57K10, 20F36
\end{abstract}

\maketitle

\section{Introduction}

A knot in $\R^3$ is usually represented by a knot diagram,
which is a picture of a projection of a knot onto a plane
such that only double points are allowed and
the double points must be genuine crossings
which transverse in the plane.
In~\cite{Ada12,Ada14}, Adams introduced the notion of $n$-crossings.
An \emph{$n$-crossing} is a point in the projection of a knot
where $n$ strands cross so that the crossing bisects each strand.
See Figure~\ref{F:ex1}(a).
We call an $n$-crossing simply a \emph{multi-crossing}
when we don't need to specify the number of strands.

A \emph{petal diagram} of a knot $K$ is a projection
of $K$ with a single multi-crossing such that there are no nested loops.
See Figure~\ref{F:ex1}(b).
Every knot admits a petal diagram~\cite{ACD+15}.
The \emph{petal number}, denoted $p(K)$, is the minimum number of loops
among all petal diagrams of $K$, or equivalently, the minimum number of
strands passing through the single multi-crossing.

Suppose we are given a petal diagram with $p$ loops.
We label the strands with $1,\ldots,p$ according to the vertical positions of the strands.
From one end of the strands we read the labels counter clockwise half way
around the crossing.
This sequence is called a \emph{petal permutation} of the projection.
The petal diagram in Figure~\ref{F:ex1}(b)
has a petal permutation $\pi=(3,5,2,4,1)$.

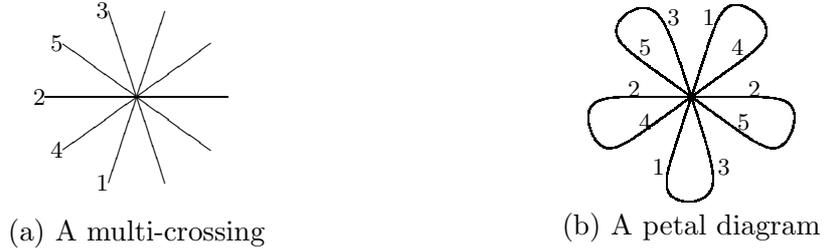
\begin{figure}[ht]\footnotesize
$$\begin{xy}/r1.2mm/:
(0,0)="or",
(3.09, 0)="h1",
(0, 9.51)="v1",
(8.09, 0)="h2",
(0, 5.88)="v2",
(0,0)-"h1"+"v1"="a1",
(0,0)-"h2"+"v2"="a2",
        (-10,0)="a3",
(0,0)-"h2"-"v2"="a4",
(0,0)-"h1"-"v1"="a5",
"a1"="a", "or"-"a"; "or"+"a" **@{-} *!R{3},
"a2"="a", "or"-"a"; "or"+"a" **@{-} *!R{5},
"a3"="a", "or"-"a"; "or"+"a" **@{-} *!R{2},
"a4"="a", "or"-"a"; "or"+"a" **@{-} *!R{4},
"a5"="a", "or"-"a"; "or"+"a" **@{-} *!R{1},
"or"+(0,-15) *{\normalsize\mbox{(a) A multi-crossing}},
\end{xy}
\qquad\qquad\qquad\qquad\qquad\qquad
\begin{xy}/r.75mm/:
(0,0)="or",
(3.09, 0)="h1",
(0, 9.51)="v1",
(8.09, 0)="h2",
(0, 5.88)="v2",
(0,0)-"h1"+"v1"="a1",
(0,0)-"h2"+"v2"="a2",
        (-10,0)="a3",
(0,0)-"h2"-"v2"="a4",
(0,0)-"h1"-"v1"="a5",
"or"; "or" **\crv{ 
"or"+"a1" & "or"+"a1"+"a1" &  "or"+"a2"+"a2" & "or"+"a2" &
"or"-"a2" & "or"-"a2"-"a2" & "or"-"a3"-"a3" & "or"-"a3" &
"or"+"a3" & "or"+"a3"+"a3" & "or"+"a4"+"a4" & "or"+"a4" &
"or"-"a4" & "or"-"a4"-"a4" & "or"-"a5"-"a5" & "or"-"a5" &
"or"+"a5" & "or"+"a5"+"a5" & "or"-"a1"-"a1" & "or"-"a1" },
"a1" *++!D{3},
"a2" *+!D{5},
"a3" *!D{2},
"a4" *!D{4},
"a5" *+!UR{1},
"or"-"a1" *+!UL{3},
"or"-"a2" *!DL{5},
"or"-"a3" *!DL{2},
"or"-"a4" *+!D{4},
"or"-"a5" *++!D{1},
"or"+(0,-23) *{\normalsize \mbox{(b) A petal diagram}},
\end{xy}
$$
\normalsize
\caption{A multi-crossing and a petal diagram}
\label{F:ex1}
\end{figure}

Notice that if $K$ is a nontrivial knot represented by a petal diagram with $p$ loops,
then $p$ is an odd integer.
(If $p$ is even, then $K$ is a link with $p/2$ components.
If $p=2$, then $K$ is the unknot.)
Since this paper is interested in torus knots,
we assume that $p$ is an odd integer in the sequel.

\subsection{Petal number of torus knots}
Let $T_{n,s}$ denote the torus knot of type $(n,s)$
for relatively prime integers $n$ and $s$ with $2\le n<s$.
There have been several works for the petal number of torus knots.

In~\cite[Theorem 4.15]{ACD+15}, Adams et al.\ showed that if $s\equiv \pm 1\bmod n$, then
$$p(T_{n,s})\le
\left\{\begin{array}{ll}
2s-1 &\quad\mbox{if $s\equiv 1\bmod n$},\\
2s+3 &\quad\mbox{if $s\equiv -1\bmod n$}.
\end{array}\right.$$

It was much improved by Kim, No and Yoo as follows.

\begin{theorem}[{\cite[Theorem 3]{KNY22}}]\label{thm:KNY-B}
Let $n$ and $s$ be relatively prime integers with $2\le n<s$.
If $s\equiv \pm 1\bmod n$, then
$$
p(T_{n,s})\le 2s-2\left\lfloor\frac sn\right\rfloor+1.
\leqno (\ast)
$$
\end{theorem}

They conjectured that
the condition ``$s\equiv \pm 1\bmod n$'' in the above theorem is not necessary.

\begin{conjecture}[\cite{KNY22}]\label{conj:KNY}
The inequality $(*)$ holds for any relatively prime integers $n$ and $s$ with $2\le n<s$.
\end{conjecture}

Using super bridge index,
they determined the petal number $p(T_{n,s})$ for some $n$ and $s$.

\begin{theorem}[{\cite[Theorem 1]{KNY22}}]\label{thm:KNY-A}
Let $n$ and $s$ be relatively prime integers with $2\le n<s$.
If $r\equiv 1\bmod s-n$, then
$$p(T_{n,s})=2s-1.$$
\end{theorem}

The above theorem was proved by Adams et al.~\cite[Corollary 4.16]{ACD+15}
for the case where $s=n+1$
and by Lee and Jin~\cite[Theorem 1.2]{LJ21} for the case where $s=n+2$.

\begin{remark}
The condition ``$n\equiv 1\bmod s-n$''
implies that $n=am+1$ and $s=(a+1)m+1$ for some $a,m\ge 1$,
and hence that $s<2n$.
Therefore Theorem~\ref{thm:KNY-A} concerns the case where $2\le n<s<2n$.
\end{remark}

\subsection{Our results}
In this paper, we prove Conjecture~\ref{conj:KNY}.

\medskip\noindent
\rm\textbf{Theorem A.} (Corollary~\ref{T:1main})\em  \ \
Let $n$ and $s$ be relatively prime integers with $2\le n<s$.
Then
$$p(T_{n,s})\le 2s-2\left\lfloor\frac sn\right\rfloor +1.$$
\normalfont
\smallskip

We also show that Theorem~\ref{thm:KNY-A} holds for all
$n$ and $s$ with $2\le n<s<2n$, without the
assumption ``$n\equiv 1\bmod s-n$''.

\medskip\noindent
\textbf{Theorem B.} (Corollary~\ref{T:2main})\em  \ \
Let $n$ and $s$ be relatively prime integers with $2\le n<s<2n$.
Then
$$p(T_{n,s})=2s-1.$$
\normalfont

\begin{remark}
After finishing the present paper, the authors found the preprint~\cite{Nie23} of Zipei Nie  
released a few days ago, where he also obtained Theorems A and B by using a different technique.
He considered \emph{star diagrams} of  braids which is an analogue of petal diagrams.

Our work was done independently by using \emph{petal grid diagrams},
and our result is slightly general than his.
He considered only torus knots.
Our technique works for a larger class of knots
(see Corollary~\ref{T:11} for an example).
\end{remark}

\subsection{Notations}
For $x\in\R$,
$\lceil x\rceil=\min\{n\in\Z: n\ge x\}$ and
$\lfloor x\rfloor=\max\{n\in\Z: n\le x\}$.

The permutation $\pi$ on $\{1,\ldots,n\}$
sending $i$ to $a_i$ is denoted by $\pi=(a_1,\ldots,a_n)$,
and called an \emph{$n$-permutation}.
Our convention is that permutations act from the left,
hence $(\pi_1\circ\pi_2)(i)=\pi_1(\pi_2(i))$
for $n$-permutations $\pi_1$ and $\pi_2$.
For a $(k+1)$-permutation $\pi_1=(a_1,\ldots, a_{k+1})$
and a $k$-permutation $\pi_2=(b_1,\ldots,b_{k})$,
$\pi_1\odot\pi_2$ denotes the $(2k+1)$-permutation defined by
$$\pi_1\odot\pi_2=(a_1,b_1,a_2,b_2,\ldots,a_k,b_k,a_{k+1}).$$
For example, $(1,2,3,4)\odot(5,6,7)=(1,5,2,6,3,7,4)$.

\mynewpage
\section{Petal grid diagrams}

A \emph{grid diagram} is a polygonal knot diagram in $\R^2$ consisting of
finitely many horizontal edges and the same number of vertical
edges such that vertical edges always cross over horizontal edges.
It is required that the end points of these edges, called \emph{nodes},
are points in $\{1,\ldots,p\}\times\{1,\ldots,p\}$,
where $p$ is the number of horizontal edges,
and that each line of the form $\R\times \{k\}$ or $\{k\}\times\R$
for $k\in\{1,\ldots,p\}$ contains exactly two nodes.
See Figure~\ref{F:ex2}(e).

It is known that every knot admits a grid diagram.
We use the convention that
the $x$- (resp.\ $y$-) coordinate increases
from 1 to $p$ as we move rightward (resp.\ upward).

\begin{definition}[petal grid diagram]
A \emph{petal grid diagram} of a knot is a grid diagram with $p=2n+1$
horizontal edges for some $n\ge 1$ such that the following hold
(see Figure~\ref{F:ex2}(e)).
\begin{itemize}
\item[(i)] There is exactly one vertical edge, called the \emph{inflection edge},
whose adjacent horizontal edges point opposite directions
(one points to the left and the other points to the right).
The adjacent horizontal edges have length $n$.
\item[(ii)] Every vertical edge except the inflection edge has
one adjacent horizontal edge with length $n$ and the other with length $n+1$.
\end{itemize}
\end{definition}

The following lemma is well known to experts, and it
was often used in literature~\cite{ACD+15,LJ21,KNY22}.

\begin{lemma}\label{l:gr}
A knot has a petal diagram with $p$ loops if and only if
it has a petal grid diagram with $p$ horizontal edges.
\end{lemma}

Let us explain how to obtain a petal grid diagram from a petal permutation.
See Figure~\ref{F:ex2}.

Let $p=2n+1$.
Suppose that we are given a petal diagram of a knot $K$
with petal permutation
$$\pi=(a_1,\ldots,a_{p}).$$
For notational convenience,
we also use the integers $a_{p+1},\ldots, a_{2p}$
defined by $a_{p+i}=a_i$ for $1\le i\le p$.

Let $B_1,\ldots,B_{2p}$ be the points in $[-1,1]\times[1,p]\subset\R^2$
defined by
$$
B_i=(-1,p+1-i),\qquad B_{p+i}=(1,i)
$$
for $1\le i\le p$.
We may assume that the petal diagram
consists of the line segments $B_iB_{p+i}$
and $B_{2i-1}B_{2i}$ for $i=1,\ldots,p$.
See Figure~\ref{F:ex2}(a).
We assume that the knot $K$ is oriented so that one travels
from $B_{2i-1}$ to $B_{2i}$ along the strand.

Let $B_1',\ldots,B_{2p}'$ be the points in the trapezoid
$\{(x,y): 1\le y\le p,~-y\le x\le y\}\subset\R^2$
defined by
$$
B_i'=(-p-1+i,p+1-i),\qquad B_{p+i}'=(i,i)
$$
for $1\le i\le p$.
See Figure~\ref{F:ex2}(b).
Let $\mathcal D'$ be the union of the line segments $B_i'B_{p+i}'$
and $B_{2i-1}'B_{2i}'$  for $i=1,\ldots,p$.
Then $\mathcal D'$ is a knot diagram of $K$.

Let $B_1'',\ldots,B_{2p}''$ be the points in
$\{(x,y): 1\le y\le p,~-y\le x\le y\}\times[1,p]\subset\R^3$
defined by
$$
B_i''=(-p-1+i,p+1-i,a_i),\qquad B_{p+i}''=(i,i,a_i)
$$
for $1\le i\le p$.
Let $\mathcal D''$ be the union of the line segments $B_i''B_{p+i}''$
and $B_{2i-1}''B_{2i}''$  for $1\le i\le p$.
See Figure~\ref{F:ex2}(c).
Then $\mathcal D''$ is isotopic to the knot $K$,
and the projection of $\mathcal D''$ onto the $xy$-plane
is the knot diagram $\mathcal D'$.


Let $B_1''',\ldots,B_{2p}'''$ be the points in $[-p,p]\times[1,p]\subset\R^2$
defined by
$$B_i'''=(-p-1+i,a_i),\qquad B_{i+p}'''=(i,a_i)$$
for $1\le i\le p$.
Let $\mathcal D'''$ be the union of the line segments $B_i'''B_{p+i}'''$
and $B_{2i-1}'''B_{2i}'''$  for $1\le i\le p$.
See Figure~\ref{F:ex2}(d).
Then $\mathcal D'''$ is the projection of
$\mathcal D''$ onto the $xz$-plane, hence it is a knot diagram of $K$.

Lastly, let $A_1,\ldots,A_{2p}$ be the points
in $[1,p]\times[1,p]\subset\R^2$ defined by
$$
A_i=\left( \left\lceil\textstyle \frac i2\right\rceil, a_i\right)
$$
for $1\le i\le 2p$.
Let $\mathcal D$ be the union of the line segments $A_iA_{p+i}$
and $A_{2i-1}A_{2i}$  for $1\le i\le p$.
See Figure~\ref{F:ex2}(e).
Then $\mathcal D$ is a petal grid diagram
of the knot $K$.

Consequently, we have the following lemma.

\begin{figure}\footnotesize
$$
\begin{array}{ccc}
\begin{xy}/r.5mm/:
(0,0)="or", (20,0)="hd", (0,15)="vd",
"or"-"hd"+"vd"+"vd"="x1"="aa",
"aa"-"vd"="aa"="x2" ,
"aa"-"vd"="aa"="x3",
"aa"-"vd"="aa"="x4",
"aa"-"vd"="aa"="x5",
"x5"+"hd"+"hd"="y1"="aa",
"aa"+"vd"="aa"="y2",
"aa"+"vd"="aa"="y3",
"aa"+"vd"="aa"="y4",
"aa"+"vd"="aa"="y5",
"x1" *+!R{B_1},
"x2" *+!R{B_2},
"x3" *+!R{B_3},
"x4" *+!R{B_4},
"x5" *+!R{B_5},
"y1" *+!L{B_6},
"y2" *+!L{B_7},
"y3" *+!L{B_8},
"y4" *+!L{B_9},
"y5" *+!L{B_{10}},
"x1"; "x2" **@{-} ?(.6) *@{>},
"x3"; "x4" **@{-} ?(.6) *@{>},
"x5"; "y1" **@{-} ?(.6) *@{>},
"y2"; "y3" **@{-} ?(.6) *@{>},
"y4"; "y5" **@{-} ?(.6) *@{>},
"x1"; "y1" **@{-} ?(.15) *+!D{3} ?(.85) *+!D{3},
"x2"; "y2" **@{-} ?(.15) *+!D{5} ?(.85) *+!D{5} ,
"x3"; "y3" **@{-} ?(.15) *+!D{2} ?(.85) *+!D{2},
"x4"; "y4" **@{-} ?(.15) *+!D{4} ?(.85) *+!D{4},
"x5"; "y5" **@{-} ?(.15) *+!D{1} ?(.85) *+!D{1},
\end{xy}
&\qquad\qquad&
\begin{xy}/r.5mm/:
(0,0)="or", (15,0)="hd", (-5,15)="vdL", (5,15)="vd",
"or"-"hd"+"vd"+"vd"="x1"="aa",
"aa"-"vdL"="aa"="x2" ,
"aa"-"vdL"="aa"="x3",
"aa"-"vdL"="aa"="x4",
"aa"-"vdL"="aa"="x5",
"x5"+"hd"+"hd"="y1"="aa",
"aa"+"vd"="aa"="y2",
"aa"+"vd"="aa"="y3",
"aa"+"vd"="aa"="y4",
"aa"+"vd"="aa"="y5",
"x1" *+!R{B_1'},
"x2" *+!R{B_2'},
"x3" *+!R{B_3'},
"x4" *+!R{B_4'},
"x5" *+!R{B_5'},
"y1" *+!L{B_6'},
"y2" *+!L{B_7'},
"y3" *+!L{B_8'},
"y4" *+!L{B_9'},
"y5" *+!L{B_{10}'},
"x1"; "x2" **@{-} ?(.6) *@{>},
"x3"; "x4" **@{-} ?(.6) *@{>},
"x5"; "y1" **@{-} ?(.6) *@{>},
"y2"; "y3" **@{-} ?(.6) *@{>},
"y4"; "y5" **@{-} ?(.6) *@{>},
"x1"="aa"; "aa"+(17.5,-21)="aa" **@{-} ?(.3) *++!L{3},
    "aa"+(5,-6)="aa"; "aa"+(6.5,-7.8)="aa" **@{-},
    "aa"+(2,-2.4)="aa"; "y1"**@{-},
"x2"; "y2" **@{-} ?(.1) *!UR{5},
"x3"="aa"; "aa"+(13,0)="aa" **@{-} ?(.3) *+!D{2},
    "aa"+(4,0)="aa"; "aa"+(2,0)="aa" **@{-},
    "aa"+(4,0)="aa"; "aa"+(4,0)="aa" **@{-},
    "aa"+(4,0)="aa"; "y3" **@{-},
"x4"="aa"; "aa"+(18.5,11.1)="aa" **@{-} ?(.2) *+!D{4},
    "aa"+(2.5,1.5)="aa"; "y4" **@{-},
"x5"="aa"; "aa"+(13.5,16.2)="aa" **@{-} ?(.3) *++!D{1},
    "aa"+(2.5,3)="aa"; "aa"+(3,3.6)="aa" **@{-},
    "aa"+(2.5,3)="aa"; "aa"+(2.5,3)="aa" **@{-},
    "aa"+(2,2.4)="aa"; "aa"+(2.5,3)="aa" **@{-},
    "aa"+(4,4.8)="aa"; "y5"**@{-},
\end{xy}\\
\mbox{\normalsize  (a) Petal diagram}\rule{0pt}{1.5em}
&&
\mbox{\normalsize  (b) $\mathcal D'$}
\end{array}
$$

$$
\begin{array}{ccccc}
\begin{xy}/r.4mm/:
(0,60)="or", (5,-5)="hd", (5,5)="hu", (0,-15)="vd", (25,0)="LRd",
 "hd"+"hd"="2hd",
"2hd"+"hd"="3hd",
"3hd"+"hd"="4hd",
"4hd"+"hd"="5hd",
 "hu"+"hu"="2hu",
"2hu"+"hu"="3hu",
"3hu"+"hu"="4hu",
"4hu"+"hu"="5hu",
 "vd"+"vd"="2vd",
"2vd"+"vd"="3vd",
"3vd"+"vd"="4vd",
"4vd"+"vd"="5vd",
"or"="L0", "L0"+"hd"+"vd"="L1", "L1"+"4hd"="L2",
"L2"+"LRd"="R2", "R2"+"4hu"="R1", "R2"-"hu"-"vd"="R0",
"L1"; "R1" **@{.},
"L2"; "R2" **@{.},
"L2"+"4vd"; "R2"+"4vd" **@{.},
"L0"+ "hd"+"3vd"="P1" *{\bullet},
"L0"+"2hd"+ "vd"="P2" *{\bullet},
"L0"+"3hd"+"4vd"="P3" *{\bullet},
"L0"+"4hd"+"2vd"="P4" *{\bullet},
"L0"+"5hd"+"5vd"="P5" *{\bullet},
"R0"+ "hu"+"3vd"="Q1" *{\bullet},
"R0"+"2hu"+ "vd"="Q2" *{\bullet},
"R0"+"3hu"+"4vd"="Q3" *{\bullet},
"R0"+"4hu"+"2vd"="Q4" *{\bullet},
"R0"+"5hu"+"5vd"="Q5" *{\bullet},
"P1"; "P2" **@{-} ?(.5) *@{>},
"P3"; "P4" **@{-} ?(.5) *@{>},
"P5"; "Q1" **@{-} ?(.5) *@{>},
"Q2"; "Q3" **@{-} ?(.5) *@{>},
"Q4"; "Q5" **@{-} ?(.5) *@{>},
"L1"="aa"; "aa"+"4hd" **@{.},
"aa"+"vd"="aa"; "aa"+"4hd" **@{.},
"aa"+"vd"="aa"; "aa"+"4hd" **@{.},
"aa"+"vd"="aa"; "aa"+"4hd" **@{.},
"aa"+"vd"="aa"; "aa"+"4hd" **@{.},
"R2"="aa"; "aa"+"4hu" **@{.},
"aa"+"vd"="aa"; "aa"+"4hu" **@{.},
"aa"+"vd"="aa"; "aa"+"4hu" **@{.},
"aa"+"vd"="aa"; "aa"+"4hu" **@{.},
"aa"+"vd"="aa"; "aa"+"4hu" **@{.},
"L1"="aa"; "aa"+"4vd" **@{.},
"aa"+"hd"="aa"; "aa"+"4vd" **@{.},
"aa"+"hd"="aa"; "aa"+"4vd" **@{.},
"aa"+"hd"="aa"; "aa"+"4vd" **@{.},
"aa"+"hd"="aa"; "aa"+"4vd" **@{.},
"R2"="aa"; "aa"+"4vd" **@{.},
"aa"+"hu"="aa"; "aa"+"4vd" **@{.},
"aa"+"hu"="aa"; "aa"+"4vd" **@{.},
"aa"+"hu"="aa"; "aa"+"4vd" **@{.},
"aa"+"hu"="aa"; "aa"+"4vd" **@{.},
"P1" *+!R{B_1''},
"P2" *+!D{B_2''},
"P3" *+!R{B_3''},
"P4" *+!D{B_4''},
"P5" *+!U{B_5''},
"Q1" *+!D{B_6''},
"Q2" *+!D{B_7''},
"Q3" *+!U{B_8''},
"Q4" *+!L{B_9''},
"Q5" *+!L{B_{10}''},
"P1"="aa"; "aa"+(11.7,-5.2)="aa" **@{-}, "aa"+(4.5,-2)="aa"; "Q1" **@{-},
"P2"; "Q2" **@{-},
"P3"="aa"; "aa"+(28,0)="aa" **@{-}, "aa"+(4,0)="aa"; "Q3" **@{-},
"P4"="aa"; "aa"+(33.3,7.4)="aa" **@{-}, "aa"+(4.5,1); "Q4" **@{-},
"P5"; "Q5" **@{-},
\end{xy}
&&
\begin{xy}/r.4mm/:
(0,60)="or", (6,0)="hd", (6,0)="hu", (0,-20)="vd", (12,0)="LRd",
(3,0)="hds",
 "hd"+"hd"="2hd",
"2hd"+"hd"="3hd",
"3hd"+"hd"="4hd",
"4hd"+"hd"="5hd",
 "hu"+"hu"="2hu",
"2hu"+"hu"="3hu",
"3hu"+"hu"="4hu",
"4hu"+"hu"="5hu",
 "vd"+"vd"="2vd",
"2vd"+"vd"="3vd",
"3vd"+"vd"="4vd",
"4vd"+"vd"="5vd",
"or"="L0", "L0"+"hd"+"vd"="L1", "L1"+"4hd"="L2",
"L2"+"LRd"="R2", "R2"+"4hu"="R1", "R2"-"hu"-"vd"="R0",
"L1"; "R1" **@{.},
"L2"; "R2" **@{.},
"L2"+"4vd"; "R2"+"4vd" **@{.},
"L0"+ "hd"+"3vd"="P1" *{\bullet},
"L0"+"2hd"+ "vd"="P2" *{\bullet},
"L0"+"3hd"+"4vd"="P3" *{\bullet},
"L0"+"4hd"+"2vd"="P4" *{\bullet},
"L0"+"5hd"+"5vd"="P5" *{\bullet},
"R0"+ "hu"+"3vd"="Q1" *{\bullet},
"R0"+"2hu"+ "vd"="Q2" *{\bullet},
"R0"+"3hu"+"4vd"="Q3" *{\bullet},
"R0"+"4hu"+"2vd"="Q4" *{\bullet},
"R0"+"5hu"+"5vd"="Q5" *{\bullet},
"P1"; "P2" **@{-} ?(.5) *@{>},
"P3"; "P4" **@{-} ?(.5) *@{>},
"P5"; "Q1" **@{-} ?(.5) *@{>},
"Q2"; "Q3" **@{-} ?(.5) *@{>},
"Q4"; "Q5" **@{-} ?(.5) *@{>},
"L1"="aa"; "aa"+"4hd" **@{.},
"aa"+"vd"="aa"; "aa"+"4hd" **@{.},
"aa"+"vd"="aa"; "aa"+"4hd" **@{.},
"aa"+"vd"="aa"; "aa"+"4hd" **@{.},
"aa"+"vd"="aa"; "aa"+"4hd" **@{.},
"R2"="aa"; "aa"+"4hu" **@{.},
"aa"+"vd"="aa"; "aa"+"4hu" **@{.},
"aa"+"vd"="aa"; "aa"+"4hu" **@{.},
"aa"+"vd"="aa"; "aa"+"4hu" **@{.},
"aa"+"vd"="aa"; "aa"+"4hu" **@{.},
"L1"="aa"; "aa"+"4vd" **@{.},
"aa"+"hd"="aa"; "aa"+"4vd" **@{.},
"aa"+"hd"="aa"; "aa"+"4vd" **@{.},
"aa"+"hd"="aa"; "aa"+"4vd" **@{.},
"aa"+"hd"="aa"; "aa"+"4vd" **@{.},
"R2"="aa"; "aa"+"4vd" **@{.},
"aa"+"hu"="aa"; "aa"+"4vd" **@{.},
"aa"+"hu"="aa"; "aa"+"4vd" **@{.},
"aa"+"hu"="aa"; "aa"+"4vd" **@{.},
"aa"+"hu"="aa"; "aa"+"4vd" **@{.},
"P1" *+!R{B_1'''},
"P2" *+!D{B_2'''},
"P3" *+!R{B_3'''},
"P4" *+!D{B_4'''},
"P5" *+!U{B_5'''},
"Q1" *+!D{B_6'''},
"Q2" *+!D{B_7'''},
"Q3" *+!U{B_8'''},
"Q4" *+!L{B_9'''},
"Q5" *+!U{B_{10}'''},
"P1"; "P1"+"hd"+"hd" **@{-}, "P1"+"hd"+"hd"+"hd"; "Q1" **@{-},
"P2"; "Q2" **@{-},
"P3"; "P3"+"hd"+"hd"+"hds" **@{-}, "P3"+"hd"+"hd"+"hd"+"hds"; "Q3" **@{-},
"P4"; "Q4"-"hd"-"hd"**@{-}, "Q4"-"hd"; "Q4" **@{-},
"P5"; "Q5" **@{-},
\end{xy} 
&&
\begin{xy}/r.75mm/:
(0,30)="or",
"or"+(10,-30)="a1" *{\bullet} *+!R{A_1},
"or"+(10,-10)="b1" *{\bullet} *+!R{A_2},
"or"+(20,-40)="a2" *{\bullet} *+!R{A_3},
"or"+(20,-20)="b2" *{\bullet} *+!R{A_4},
"or"+(30,-50)="a3" *{\bullet} *+!R{A_5},
"or"+(30,-30)="b3" *{\bullet} *+!L{A_6},
"or"+(40,-10)="a4" *{\bullet} *+!L{A_7},
"or"+(40,-40)="b4" *{\bullet} *+!L{A_8},
"or"+(50,-20)="a5" *{\bullet} *+!L{A_9},
"or"+(50,-50)="b5" *{\bullet} *+!L{A_{10}},
"a1"; "b1" **@{-} ?(.25) *@{>},
"a2"; "b2" **@{-} ?(.75) *@{>},
"a3"; "b3" **@{-} ?(.75) *@{>},
"a4"; "b4" **@{-} ?(.57) *@{>},
"a5"; "b5" **@{-} ?(.25) *@{>},
"b1"; "a4" **@{-},
"b2"; "a5"-(12,0) **@{-}, "a5"-(8,0); "a5" **@{-},
"b3"; "a1"+(12,0) **@{-}, "a1"+(8,0); "a1" **@{-},
"b4"; "a2"+(12,0) **@{-}, "a2"+(8,0); "a2" **@{-},
"b5"; "a3" **@{-},
"b1"+(14,  1)="aa" *!D{5},
"aa"+( 0,-10)="aa" *!D{4},
"aa"+( 0,-10)="aa" *!D{3},
"aa"+( 0,-10)="aa" *!D{2},
"aa"+(10,-10)="aa" *!D{1},
\end{xy}\\
\mbox{\normalsize (c) $\mathcal D''$\rule{0pt}{1.5em}}
&& \mbox{\normalsize  (d) $\mathcal D'''$}
&& \mbox{\normalsize   (e) $\mathcal D=\PG(\pi)$}
\end{array}
$$
\caption{Petal grid diagram defined by the petal permutation $\pi=(3,5,2,4,1)$}
\label{F:ex2}
\end{figure}
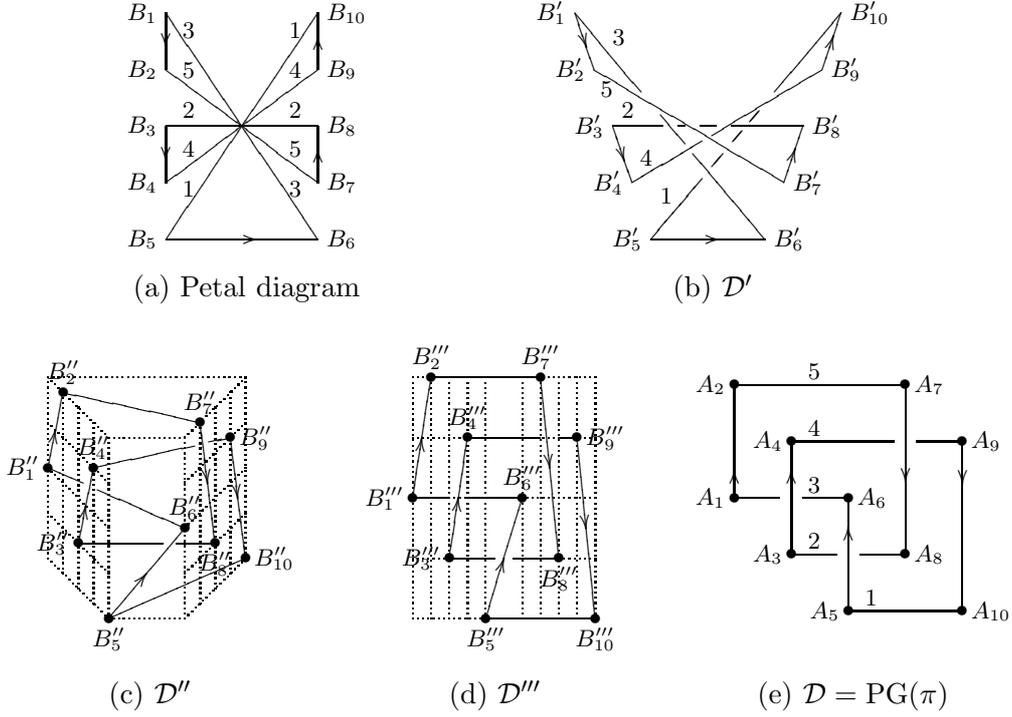

\begin{lemma}\label{l:gr-alg}
Let $\pi=(a_1,\ldots,a_p)$ be a petal permutation of a knot $K$.
Let $p=2n+1$.
Then $K$ has a petal grid diagram with $2p$ nodes
$$
A_i=\left( \left\lceil\textstyle \frac i2\right\rceil, a_i\right),
\qquad 1\le i\le 2p.
$$
The horizontal edges are $A_{i}A_{p+i}$
and the vertical edges are
$A_{2i-1}A_{2i}$ for $1\le i\le p$.
When we travel along the knot, we visit the nodes
in the order
$$ \cdots \to A_{2i-1}\to A_{2i}\to A_{p+2i}\to A_{p+2i+1} \to A_{2i+1}\to A_{2i+2}\to A_{p+2i+2}\to\cdots.$$
The inflection edge is $A_pA_{p+1}$.
The horizontal edge $A_iA_{p+i}$ ($1\le i\le p$)
has length $n=\frac{p-1}2$ if\/ $i$ is odd
and has length $n+1=\frac{p+1}2$ if\/ $i$ is even.
\end{lemma}

\begin{definition}
Let $\pi=(a_1,\ldots,a_p)$ be a $p$-permutation for an odd integer $p\ge 3$.
\begin{itemize}
\item[(i)]
$\PG(\pi)$ denotes the petal grid diagram in Lemma~\ref{l:gr-alg},
i.e.~the petal grid diagram with $2p$ nodes
$A_i=\left( \left\lceil\textstyle \frac i2\right\rceil, a_i\right)$
for $1\le i\le 2p$.
\item[(ii)]
$K(\pi)$ denotes the knot determined by $\PG(\pi)$.
\end{itemize}
\end{definition}

For example, if $\pi=(3,5,2,4,1)$, then $K(\pi)$ is the $(2,3)$-torus knot
as shown in Figure~\ref{F:ex2}(e).

\section{Braids}

We start this section by reviewing briefly some basic material for braid groups. 
See [Bir74,ECH+92,BKL98] for details.
Let $\sigma_1,\ldots,\sigma_{n-1}$ denote
the standard generators of the braid group $B_n$ on $n$ strands.
$B_n$ has the group presentation
$$
B_n  =  \left\langle \sigma_1,\ldots,\sigma_{n-1} \left|
\begin{array}{ll}
\sigma_i \sigma_j = \sigma_j \sigma_i & \mbox{if } |i-j| \ge 2, \\
\sigma_i \sigma_j \sigma_i = \sigma_j \sigma_i \sigma_j & \mbox{if } |i-j| = 1.
\end{array}
\right.\right\rangle.
$$
An $n$-braid $\alpha$ can be regarded as a collection of $n$ strands
$l=l_1\cup\cdots\cup l_n$ in $[0,1]\times \R^2$
such that $|\,l\cap (\{t\}\times \R^2)|=n$ for $0\le t\le 1$ and
$l_i\cap (\{0,1\}\times \R^2)=\{(0,a_i,0), (1,i,0)\}$
for some $a_i\in\{1,\ldots,n\}$ for each $i=1,\ldots,n$.
See Figure~\ref{f:gen}(a,b) for diagrams of $\sigma_i$ and $\sigma_i^{-1}$.
Notice that the $i$-th strand $l_i$ connects
the $i$th point in the right to the $a_i$th point in the left
and that $(a_1,\ldots, a_n)$ is an $n$-permutation.

\begin{definition}[induced permutation]
The permutation $(a_1,\ldots, a_n)$ in the above is
called the \emph{induced permutation} of
$\alpha$ and denoted by $\pi_\alpha$.
If $\pi_\alpha$ is the identity, $\alpha$ is called a \emph{pure braid}.
\end{definition}

\begin{definition}
The $n$-braids $\delta=\delta_n$ and $\Delta=\Delta_n$ are defined
as follows (see Figure~\ref{f:gen}(c,d)).
\begin{align*}
\delta&=\sigma_{n-1}\sigma_{n-2}\cdots\sigma_1\\
\Delta&=\sigma_1(\sigma_2\sigma_1)\cdots(\sigma_{n-1}\cdots\sigma_1)
\end{align*}
\end{definition}

Their induced permutations are
$\pi_\delta=(n,1,2,\ldots,n-1)$ and $\pi_\Delta=(n,n-1,\ldots,2,1)$.

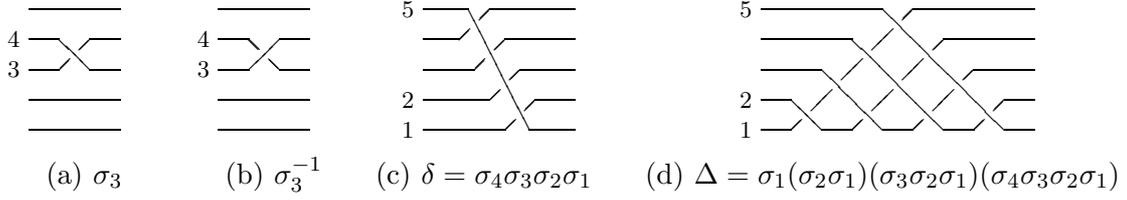
\begin{figure}\footnotesize
$$
\begin{array}{*7{c}}
\begin{xy}/r.4mm/:
(0,0)="or"="aa", (0,10)="vd",
"aa"; "aa"+(30,0) **@{-},
"aa"+"vd"="aa"="bb"; "bb"+(30,0) **@{-},
"aa"+"vd"="aa"="bb"; "bb"+(10,0)="bb" **@{-}; "bb"+(3.5,3.5)="bb" **@{-},
    "bb"+(3,3)="bb"; "bb"+(3.5,3.5)="bb" **@{-}; "aa"+(30,10) **@{-},
"aa"+"vd"="aa"="bb"; "bb"+(10,0)="bb" **@{-}; "bb"+(10,-10)="bb" **@{-};
    "aa"+(30,-10) **@{-},
    "aa" *+!R{4}, "aa"-"vd" *+!R{3},
"aa"+"vd"="aa"; "aa"+(30,0) **@{-},
\end{xy}
&&
\begin{xy}/r.4mm/:
(0,0)="or"="aa", (0,10)="vd",
(0,0)="or"="aa", (0,10)="vd",
"aa"; "aa"+(30,0) **@{-},
"aa"+"vd"="aa"="bb"; "bb"+(30,0) **@{-},
"aa"+"vd"="aa"="bb"; "bb"+(10,0)="bb" **@{-}; "bb"+(10,10)="bb" **@{-};
    "aa"+(30,10) **@{-},
"aa"+"vd"="aa"="bb"; "bb"+(10,0)="bb" **@{-}; "bb"+(3.5,-3.5)="bb" **@{-},
    "bb"+(3,-3)="bb"; "bb"+(3.5,-3.5)="bb" **@{-}; "aa"+(30,-10) **@{-},
    "aa" *+!R{4}, "aa"-"vd" *+!R{3},
"aa"+"vd"="aa"; "aa"+(30,0) **@{-},
\end{xy}
&&
\begin{xy}/r.4mm/:
(0,50)="or"="aa", (15,0)="llen", (50,0)="len", (0,10)="vd",
"aa"-"vd"="aa"="bb"; "bb"+"llen"="bb" **@{-};
    "bb"+(20,-40)="bb" **@{-}; "aa"+"len"+(0,-40) **@{-},
    "aa" *+!R{5},
"aa"-"vd"="aa"="bb"; "aa"+"llen"-(3,0)="bb" **@{-};
    "bb"+(3.5,3.5)="bb" **@{-}, "bb"+(3,3)="bb"; "bb"+(3.5,3.5)="bb" **@{-};
    "aa"+(0,10)+"len" **@{-},
"aa"-"vd"="aa"="bb"; "aa"+"llen"+(2,0)="bb" **@{-};
    "bb"+(3.5,3.5)="bb" **@{-}, "bb"+(3,3)="bb"; "bb"+(3.5,3.5)="bb" **@{-};
    "aa"+(0,10)+"len" **@{-},
"aa"-"vd"="aa"="bb"; "aa"+"llen"+(7,0)="bb" **@{-};
    "bb"+(3.5,3.5)="bb" **@{-}, "bb"+(3,3)="bb"; "bb"+(3.5,3.5)="bb" **@{-};
    "aa"+(0,10)+"len" **@{-},
    "aa" *+!R{2},
"aa"-"vd"="aa"="bb"; "aa"+"llen"+(12,0)="bb" **@{-};
    "bb"+(3.5,3.5)="bb" **@{-}, "bb"+(3,3)="bb"; "bb"+(3.5,3.5)="bb" **@{-};
    "aa"+(0,10)+"len" **@{-},
    "aa" *+!R{1},
\end{xy}
&&
\begin{xy}/r.4mm/:
(0,0)="or",
"or"+(0, 40)="LLU", "LLU"+(0,0)="LU",
"LU"+(90,-40)="RD", "RD"+(0,0)="RRD",
"LLU"="aa"; "LU"="bb"  **@{-};
    "bb"+(40,0)="bb" **@{-}; "bb"+(40,-40)="bb" **@{-};
    "RD" **@{-}; "RRD" **@{-},
    "aa" *+!R{5},
"LLU"-(0,10)="aa"; "LU"-(0,10)="bb" **@{-};
    "bb"+(30,0)="bb" **@{-}; "bb"+(30,-30)="bb" **@{-};
    "bb"+(10,0)="bb" **@{-}; "bb"+(3.5,3.5) **@{-}, "bb"+(6.5,6.5);
    "bb"+(10,10)="bb" **@{-}; "RD"+(0,10) **@{-}; "RRD"+(0,10) **@{-},
"LLU"-(0,20)="aa"; "LU"-(0,20)="bb" **@{-};
    "bb"+(20,0)="bb" **@{-}; "bb"+(20,-20)="bb" **@{-};
    "bb"+(10,0)="bb" **@{-}; "bb"+(3.5,3.5) **@{-}, "bb"+(6.5,6.5);
        "bb"+(13.5,13.5) **@{-}, "bb"+(16.5,16.5);
    "bb"+(20,20)="bb" **@{-}; "RD"+(0,20) **@{-}; "RRD"+(0,20) **@{-},
"LLU"-(0,30)="aa"; "LU"-(0,30)="bb" **@{-};
    "bb"+(10,0)="bb" **@{-}; "bb"+(10,-10)="bb" **@{-};
    "bb"+(10,0)="bb" **@{-}; "bb"+(3.5,3.5) **@{-}, "bb"+(6.5,6.5);
        "bb"+(13.5,13.5) **@{-}, "bb"+(16.5,16.5);
        "bb"+(23.5,23.5) **@{-}, "bb"+(26.5,26.5);
    "bb"+(30,30)="bb" **@{-}; "RD"+(0,30) **@{-}; "RRD"+(0,30) **@{-},
    "aa" *+!R{2},
"LLU"-(0,40)="aa"; "LU"-(0,40)="bb" **@{-};
    "bb"+(10,0)="bb" **@{-}; "bb"+(3.5,3.5) **@{-}, "bb"+(6.5,6.5);
        "bb"+(13.5,13.5) **@{-}, "bb"+(16.5,16.5);
        "bb"+(23.5,23.5) **@{-}, "bb"+(26.5,26.5);
        "bb"+(33.5,33.5) **@{-}, "bb"+(36.5,36.5);
    "bb"+(40,40)="bb" **@{-}; "RD"+(0,40) **@{-}; "RRD"+(0,40) **@{-},
    "aa" *+!R{1},
\end{xy}
\\
\qquad\mbox{\normalsize (a)  $\sigma_3$}
&&\qquad \mbox{\normalsize (b) $\sigma_3^{-1}$}
&& \mbox{\normalsize (c) $\delta=\sigma_4\sigma_3\sigma_2\sigma_{1}$}
&& \mbox{\normalsize (d) $\Delta
=\sigma_1(\sigma_2\sigma_1)(\sigma_3\sigma_2\sigma_1)(\sigma_4\sigma_3\sigma_2\sigma_1)$}
\rule{0pt}{1.5em}
\end{array}
$$
\caption{The braids $\sigma_3$, $\sigma_3^{-1}$, $\delta$ and $\Delta$ in $B_5$}
\label{f:gen}
\end{figure}

\begin{definition}
For $2\le k\le n$,
the $n$-braids $D_k$, $E_k$ and $U_k$ are defined as follows
(see Figure~\ref{f:DEU}).
\begin{align*}
D_k&=\sigma_{k-1}\cdots\sigma_2\sigma_1\\
E_k&=\sigma_1\sigma_2\cdots\sigma_{k-1}\\
U_k&=D_k E_k=(\sigma_{k-1}\cdots\sigma_2\sigma_1)(\sigma_1\sigma_2\cdots \sigma_{k-1})
\end{align*}
For notational convenience, we define $D_1=E_1=U_1=1$.
\end{definition}

\begin{figure}\footnotesize
$$
\begin{array}{*7{c}}
\begin{xy}/r.4mm/:
(0,25)="or"="aa", (15,0)="llen", (45,0)="len", (0,10)="vd",
"aa"; "aa"+"len" **@{-},
"aa"-"vd"="aa"; "aa"+"len" **@{-},
"aa"-"vd"="aa"="bb"; "bb"+"llen"="bb" **@{-}; "bb"+(15,-30)="bb" **@{-}; "aa"+"len"+(0,-30) **@{-},
"aa"-"vd"="aa"="bb"; "aa"+"llen"-(3,0)="bb" **@{-};
    "bb"+(3.5,3.5)="bb" **@{-}, "bb"+(3,3)="bb"; "bb"+(3.5,3.5)="bb" **@{-};
    "aa"+(0,10)+"len" **@{-},
"aa"-"vd"="aa"="bb"; "aa"+"llen"+(2,0)="bb" **@{-};
    "bb"+(3.5,3.5)="bb" **@{-}, "bb"+(3,3)="bb"; "bb"+(3.5,3.5)="bb" **@{-};
    "aa"+(0,10)+"len" **@{-},
"aa"-"vd"="aa"="bb"; "aa"+"llen"+(7,0)="bb" **@{-};
    "bb"+(3.5,3.5)="bb" **@{-}, "bb"+(3,3)="bb"; "bb"+(3.5,3.5)="bb" **@{-};
    "aa"+(0,10)+"len" **@{-},
"aa" *+!R{1},
"aa"+"vd"+"vd"+"vd"="aa" *+!R{4},
"aa"+"vd"+"vd" *+!R{6},\end{xy}
&\quad&
\begin{xy}/r.4mm/:
(0,25)="or"="aa", (15,0)="llen", (45,0)="len", (0,10)="vd",
"aa"; "aa"+"len" **@{-},
"aa"-"vd"="aa"; "aa"+"len" **@{-},
"aa"-"vd"="aa"="bb"; "aa"+"llen"+(6,0)="bb" **@{-};
    "bb"+(10,-10)="bb" **@{-}; "aa"+(0,-10)+"len" **@{-},
"aa"-"vd"="aa"="bb"; "aa"+"llen"+(3,0)="bb" **@{-};
    "bb"+(10,-10)="bb" **@{-}; "aa"+(0,-10)+"len" **@{-},
"aa"-"vd"="aa"="bb"; "aa"+"llen"+(0,0)="bb" **@{-};
    "bb"+(10,-10)="bb" **@{-}; "aa"+(0,-10)+"len" **@{-},
"aa"-"vd"="aa"="bb"; "aa"+"llen"+(1,0)="bb" **@{-};
    "bb"+(2,4)="bb" **@{-}, "bb"+(2.2,4.4)="bb"; "bb"+(2,4)="bb" **@{-},
    "bb"+(2.2,4.4)="bb"; "bb"+(2,4)="bb" **@{-},
    "bb"+(2.2,4.4)="bb"; "bb"+(2.4,4.8)="bb" **@{-};
    "aa"+(0,30)+"len" **@{-},
"aa" *+!R{1},
"aa"+"vd"+"vd"+"vd"="aa" *+!R{4},
"aa"+"vd"+"vd" *+!R{6},
\end{xy}
&\quad&
\begin{xy}/r.4mm/:
(0,25)="or"="aa", (15,0)="llen", (50,0)="len", (0,10)="vd",
"aa"; "aa"+"len" **@{-},
"aa"-"vd"="aa"; "aa"+"len" **@{-},
"aa"-"vd"="aa"="bb"; "bb"+"llen"="bb" **@{-};
    "bb"+(5,-40)="bb" **@{-};
    "bb"+(10,0)="bb" **@{-};
    "bb"+(1,8) **@{-}, "bb"+(1.5,12);
    "bb"+(2.25,18) **@{-}, "bb"+(2.75,22);
    "bb"+(3.50,28) **@{-}, "bb"+(4,32);
    "bb"+(5,40) **@{-};
    "aa"+"len" **@{-},
"aa"-"vd"="aa"; "aa"+"llen"+(-2,0) **@{-}, "aa"+"llen"+(4,0); "aa"+"len" **@{-},
"aa"-"vd"="aa"; "aa"+"llen"+(-1,0) **@{-}, "aa"+"llen"+(5,0); "aa"+"len" **@{-},
"aa"-"vd"="aa"; "aa"+"llen"+(0,0) **@{-}, "aa"+"llen"+(6,0); "aa"+"len" **@{-},
"aa" *+!R{1},
"aa"+"vd"+"vd"+"vd"="aa" *+!R{4},
"aa"+"vd"+"vd" *+!R{6},
\end{xy}\\
\mbox{\normalsize (a) $D_4=\sigma_3\sigma_2\sigma_1$}
&\qquad& \mbox{\normalsize (b) $E_4=\sigma_{1}\sigma_2\sigma_3$}
&\qquad& \mbox{\normalsize (c) $U_4=D_4E_4$}\rule{0pt}{2em}
\end{array}
$$
\caption{The braids $D_4$, $E_4$ and $U_4$ in $B_6$}\label{f:DEU}
\end{figure}
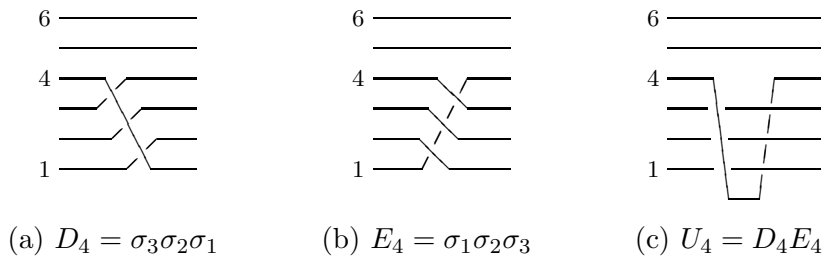

The braids $\delta$ and $\Delta$ are expressed as
$\delta=D_n$ and $\Delta=D_2D_3\cdots D_n$.
It is well known that $\Delta^2=\delta^n$ and that
the center of $B_n$ for $n\ge 3$ is the infinite cyclic group generated
by $\Delta^2=\delta^n$.
It is also well known that  $\Delta$ can be also written as
$\Delta=(\sigma_1\ldots\sigma_{n-1}) \cdots(\sigma_1\sigma_2)\sigma_1
=E_n\cdots E_3E_2$.

\begin{definition}[automorphism $\tau$]
$\tau:B_n\to B_n$ is the inner automorphism defined by
$$\tau(\alpha)=\delta^{-1}\alpha\delta.$$
\end{definition}

Notice that $\tau(\sigma_i)=\sigma_{i+1}$ for $i=1,\ldots,n-2$.

\begin{definition}[positive braid]
An $n$-braid $\beta$ is called a \emph{positive braid}
if it can be written as a word in positive powers of the generators
$\{\sigma_1,\ldots,\sigma_{n-1}\}$
without use of the inverse elements $\sigma_i^{-1}$.
\end{definition}

\begin{definition}[permutation braid]
A braid $\beta$ is called a \emph{permutation braid}
if it is a positive braid such that no pair of strands cross more than once.
\end{definition}

We usually denote positive braids and permutation braids
by capital letters such as
$P,Q,\ldots$.
It is well known that the set of permutation $n$-braids
is in one-to-one correspondence with the set of $n$-permutations,
hence we have the following lemma.

\begin{lemma}\label{L:perm}
If $Q^{-1}P$ is a pure braid for permutation braids $P$ and $Q$,
then $Q^{-1}P$ is the identity.
\end{lemma}

\begin{proof}
Since $Q^{-1}P$ is a pure braid,
$P$ and $Q$ have the same induced permutations, i.e.\ $\pi_P=\pi_Q$, hence $P=Q$.
\end{proof}

The braids $\delta$ and $\Delta$ are permutation braids
with induced permutations $\pi_\delta$ and $\pi_\Delta$ respectively.

\begin{definition}
For an $n$-permutation $\pi$,
$P(\pi)$ denotes the permutation $n$-braid with induced permutation $\pi$.
\end{definition}

The braid diagram for $P(\pi)$ with $\pi=(a_1,\ldots,a_n)$
can be drawn as follows.
For each $1\le i\le n$, draw a straight line segment $l_i$
from $(0,a_i)$ to $(1,i)$.
(By perturbing $l_i$'s a small amount if necessary,
we may assume that there are only double crossings.)
When two strands $l_i$ and $l_j$ cross,
the strand $l_i$ crosses over the strand $l_j$ if and only if $a_i>a_j$.

For example,
the permutation braid $P(\pi)$ with $\pi=(2,4,1,3,6,5,7)$
looks like Figure~\ref{F:X}(a).

The following lemma shows basic properties of $D_k$, $E_k$ and $U_k$.

\begin{lemma}\label{L:DEU}
For each $2\le k\le n$, the following hold.
\begin{enumerate}
\item[(i)] $\sigma_i D_k=D_k\sigma_{i+1}$ and $\sigma_{i+1}E_k=E_k\sigma_i$
    for $1\le i\le k-2$.
\item[(ii)] $U_k$ commutes with $\sigma_i$ for $1\le i\le k-2$.
\item[(iii)] $U_k$ commutes with $D_i$, $E_i$ and $U_i$
for $2\le i\le k-1$.
\item[(iv)] $U_{a_1}\cdots U_{a_k}=(D_{a_1}\cdots D_{a_k})(E_{a_k}\cdots E_{a_1})$
for $2\le a_1<a_2<\cdots <a_k\le n$.
\item[(v)] $\Delta^2=\delta^n=U_2U_3\cdots U_n$.
\end{enumerate}
\end{lemma}

\begin{proof}
(i)\ \ It is obvious from the definitions of $D_k$ and $E_k$.

(ii)\ \ By (i), $\sigma_i U_k=\sigma_i D_kE_k
=D_k\sigma_{i+1}E_k
=D_kE_k\sigma_i=U_k\sigma_i$.

(iii)\ \ It follows from (ii)
because $D_i$, $E_i$ and $U_i$ are represented by words
on $\sigma_1,\ldots,\sigma_{i-1}$.

(iv)\ \ We use induction on $k$.
It is obvious for $k=1$.
Assume that it holds for $k$, and let $2\le a_1<\cdots<a_k<a_{k+1}\le n$.
By (iii), $U_{a_{k+1}}$ commutes with $E_{a_i}$ for $2\le i\le k$.
By the induction hypothesis,
\begin{align*}
U_{a_1}\cdots U_{a_k}U_{a_{k+1}}
&= (D_{a_1}\cdots D_{a_k})(E_{a_k}\cdots E_{a_1}) U_{a_{k+1}}\\
& = (D_{a_1}\cdots D_{a_k})U_{a_{k+1}}(E_{a_k}\cdots E_{a_1})\\
& = (D_{a_1}\cdots D_{a_k})D_{a_{k+1}}E_{a_{k+1}}(E_{a_k}\cdots E_{a_1}).
\end{align*}
(v)\ \ By (iv),
$U_2\cdots U_n=(D_2\cdots D_n)(E_n\cdots E_2)=\Delta\cdot\Delta=\Delta^2$.
\end{proof}

By Lemma~\ref{L:DEU}(iii), $U_iU_j=U_jU_i$ for any $2\le i,j\le n$.
In fact, $U_2,\ldots,U_n$ generate
an abelian subgroup of $B_n$ with maximum rank.

\begin{definition}
Let $A$ and $B$ be $k$-subsets of $\{1,\ldots,n\}$.
Let $\bar A$ and $\bar B$ denote the complements of $A$ and $B$ respectively
in $\{1,\ldots,n\}$.
We write
$$
\begin{array}{ccc}
A=\{a_1,\ldots,a_k\},
    &\quad&\bar A=\{a_{k+1},\ldots,a_n\},\\
B=\{b_1,\ldots,b_k\},
    &\quad&\bar B=\{b_{k+1},\ldots,b_n\},
\end{array}
$$
where $a_1<\cdots<a_k$, $a_{k+1}<\cdots<a_n$,
$b_1<\cdots<b_k$ and $b_{k+1}<\cdots<b_n$.
We define the $n$-permutation $\pi_{A,B}$ and
the $n$-braids $X_{A,B}$, $\bar X_{A,B}$ and $U(A)$
as follows:
\begin{enumerate}
\item[(i)] $\pi_{A,B}$ is the $n$-permutation defined by $\pi_{A,B}(b_i)=a_i$
 for $1\le i\le n$;
\item[(ii)] $X_{A,B}=P(\pi_{A,B})$, the permutation braid defined by $\pi_{A,B}$;
\item[(iii)] $\bar X_{A,B}=(X_{B,A})^{-1}$;
\item[(iv)] $U(A)=U_{a_1}U_{a_2}\cdots U_{a_k}$.
\end{enumerate}
\end{definition}

\begin{example}
Consider the subsets $A=\{2,4,6\}$ and $B=\{1,2,5\}$ of $\{1,\ldots,7\}$.
The braid $X_{A,B}$ is in Figure~\ref{F:X}(a).
The subsets $A$ and $B$ are not uniquely determined by $X_{A,B}$.
For example, if $A'=\{2,4,5\}$ and $B'=\{1,2,6\}$,
then $X_{A',B'}=X_{A,B}$.
\end{example}

The permutation $\pi_{A,B}$ sends $B$ to $A$ (resp.\ $\bar B$ to $\bar A$)
preserving the orders, and $\pi_{B,A}=\pi_{A,B}^{-1}$.
The braid $\bar X_{A,B}$ is the same as $X_{A,B}$ except that each crossing
is negative.
Both $X_{A,B}$ and $\bar X_{A,B}$ connect
the points in the left with heights in $A$ to the points in the right
with heights in $B$.

\begin{lemma}\label{L:TA}
Let $A$ be a $k$-subset of $\{1,\ldots,n\}$,
and let $L=\{1,\ldots,k\}$ and $T=\{n-i: 0\le i\le k-1\}$.
Then $X_{T,A}X_{A,L}=X_{T,L}$.
\end{lemma}

\begin{proof}
It is straightforward from the picture
(see Figure~\ref{F:X}(b,c)),
hence we give only a sketchy proof.
Since $X_{T,A}$ and $X_{A,L}$ are positive braids,
so is $X_{T,A}X_{A,L}$.
It is easy to see that
each pair of strands of $X_{T,A}X_{A,L}$
cross at most once, hence $X_{T,A}X_{A,L}$ is a permutation braid.
Since $X_{T,A}X_{A,L}$ and $X_{T,L}$ are
permutation braids with the same induced permutations,
they are equal.
\end{proof}

\begin{figure}\footnotesize
$$
\begin{array}{ccccc}
\begin{xy}/r.4mm/:
(0,0)="or", (50,0)="hd", (0,10)="vd",
     "or"="a1"="aa" *{\bullet} *+!R{1}, "aa"+"hd"="b1" *{\bullet} *+!L{1},
"aa"+"vd"="a2"="aa" *{\bullet} *+!R{2}, "aa"+"hd"="b2" *{\bullet} *+!L{2},
"aa"+"vd"="a3"="aa" *{\bullet} *+!R{3}, "aa"+"hd"="b3" *{\bullet} *+!L{3},
"aa"+"vd"="a4"="aa" *{\bullet} *+!R{4}, "aa"+"hd"="b4" *{\bullet} *+!L{4},
"aa"+"vd"="a5"="aa" *{\bullet} *+!R{5}, "aa"+"hd"="b5" *{\bullet} *+!L{5},
"aa"+"vd"="a6"="aa" *{\bullet} *+!R{6}, "aa"+"hd"="b6" *{\bullet} *+!L{6},
"aa"+"vd"="a7"="aa" *{\bullet} *+!R{7}, "aa"+"hd"="b7" *{\bullet} *+!L{7},
"a1"="aa"; "aa"+(12.5,5) **@{-}, "aa"+(20,8); "aa"+(35,14) **@{-},
    "aa"+(40,16); "b3" **@{-},
"a2"; "b1" **@{-},
"a3"="aa"; "aa"+(12.5,2.5) **@{-}, "aa"+(20,4); "b4" **@{-},
"a4"; "b2" **@{-},
"a5"="aa"; "aa"+(20,4) **@{-}, "aa"+(30,6); "b6" **@{-},
"a6"; "b5" **@{-},
"a7"; "b7" **@{-},
\end{xy}
&\qquad&
\begin{xy}/r.4mm/:
(40,0)="hd", (60,0)="hdL", (-10,60)="vd",
(0,0)="or"="aa"="a1","aa"+"hd"="b1", "aa"+"hd"+"hd"="c1",
"aa"+(0,10)="aa"="a2", "aa"+"hd"="b2", "aa"+"hd"+"hd"="c2",
"aa"+(0,10)="aa"="a3", "aa"+"hd"="b3", "aa"+"hd"+"hd"="c3",
"aa"+(0,10)="aa"="a4", "aa"+"hd"="b4", "aa"+"hd"+"hd"="c4",
"aa"+(0,10)="aa"="a5", "aa"+"hd"="b5", "aa"+"hd"+"hd"="c5",
"aa"+(0,10)="aa"="a6", "aa"+"hd"="b6", "aa"+"hd"+"hd"="c6",
"aa"+(0,10)="aa"="a7", "aa"+"hd"="b7", "aa"+"hd"+"hd"="c7",
"a1"="aa" *{\bullet}, "aa"+"hd" *{\bullet}, "aa"+"hd"+"hd" *{\bullet},
"aa"+(0,10)="aa" *{\bullet}, "aa"+"hd" *{\bullet}, "aa"+"hd"+"hd" *{\bullet},
"aa"+(0,10)="aa" *{\bullet}, "aa"+"hd" *{\bullet}, "aa"+"hd"+"hd" *{\bullet},
"aa"+(0,10)="aa" *{\bullet}, "aa"+"hd" *{\bullet}, "aa"+"hd"+"hd" *{\bullet},
"aa"+(0,10)="aa" *{\bullet}, "aa"+"hd" *{\bullet}, "aa"+"hd"+"hd" *{\bullet},
"aa"+(0,10)="aa" *{\bullet}, "aa"+"hd" *{\bullet}, "aa"+"hd"+"hd" *{\bullet},
"aa"+(0,10)="aa" *{\bullet}, "aa"+"hd" *{\bullet}, "aa"+"hd"+"hd" *{\bullet},
"a7"; "b6" **@{-}; "c3" **@{-},
"a6"; "b4" **@{-}; "c2" **@{-},
"a5"; "b2" **@{-}; "c1" **@{-},
"a4"="aa";
    "aa"+(4.8,3.6)="aa" **@{-};
    "aa"+(4,3)="aa";
    "aa"+(6,4.5)="aa" **@{-};
    "aa"+(4,3)="aa";
    "aa"+(8,6)="aa" **@{-};
    "aa"+(4.8,3.6)="aa";
    "b7" **@{-}; "c7" **@{-},
"a3"="aa";
    "aa"+(12,6)="aa" **@{-},
    "aa"+(6,3)="aa";
    "aa"+(8,4)="aa" **@{-},
    "aa"+(6,3)="aa";
    "b5"="aa" **@{-};
    "aa"+(6,1.5)="aa" **@{-},
    "aa"+(8,2)="aa";
    "c6" **@{-},
"a2"="aa"; "aa"+(26,6.5)="aa" **@{-},
    "aa"+(8,2)="aa"; "b3"="aa" **@{-};
    "aa"+(8,4)="aa" **@{-},
    "aa"+(6,3)="aa";
    "aa"+(8,4)="aa" **@{-},
    "aa"+(6,3)="aa";
    "c5" **@{-},
"a1"; "b1"="aa" **@{-};
    "aa"+(8,6)="aa" **@{-};
    "aa"+(4,3)="aa";
    "aa"+(10,7.5)="aa" **@{-};
    "aa"+(4,3)="aa";
    "aa"+(6,4.5)="aa" **@{-};
    "aa"+(4,3)="aa"; "c4" **@{-},
"a1" *{\bullet} *+!R{1},
"a2" *{\bullet} *+!R{2},
"a3" *{\bullet} *+!R{3},
"a4" *{\bullet} *+!R{4},
"a5" *{\bullet} *+!R{5},
"a6" *{\bullet} *+!R{6},
"a7" *{\bullet} *+!R{7},
"c1" *{\bullet} *+!L{1},
"c2" *{\bullet} *+!L{2},
"c3" *{\bullet} *+!L{3},
"c4" *{\bullet} *+!L{4},
"c5" *{\bullet} *+!L{5},
"c6" *{\bullet} *+!L{6},
"c7" *{\bullet} *+!L{7},
\end{xy}
&\qquad&
\begin{xy}/r.4mm/:
(0,0)="or", (50,0)="hd", (0,10)="vd",
     "or"="a1"="aa" *{\bullet} *+!R{1}, "aa"+"hd"="b1" *{\bullet} *+!L{1},
"aa"+"vd"="a2"="aa" *{\bullet} *+!R{2}, "aa"+"hd"="b2" *{\bullet} *+!L{2},
"aa"+"vd"="a3"="aa" *{\bullet} *+!R{3}, "aa"+"hd"="b3" *{\bullet} *+!L{3},
"aa"+"vd"="a4"="aa" *{\bullet} *+!R{4}, "aa"+"hd"="b4" *{\bullet} *+!L{4},
"aa"+"vd"="a5"="aa" *{\bullet} *+!R{5}, "aa"+"hd"="b5" *{\bullet} *+!L{5},
"aa"+"vd"="a6"="aa" *{\bullet} *+!R{6}, "aa"+"hd"="b6" *{\bullet} *+!L{6},
"aa"+"vd"="a7"="aa" *{\bullet} *+!R{7}, "aa"+"hd"="b7" *{\bullet} *+!L{7},
"a1"="aa"; "aa"+(26,15.6)="aa" **@{-}, "aa"+(4,2.4)="aa";
    "aa"+(4,2.4)="aa" **@{-}, "aa"+(3,1.8)="aa";
    "aa"+(4,2.4)="aa" **@{-}, "aa"+(3,1.8)="aa";
    "b4" **@{-},
"a2"="aa"; "aa"+(19,11.4)="aa" **@{-}, "aa"+(4,2.4)="aa";
    "aa"+(4,2.4)="aa" **@{-}, "aa"+(3,1.8)="aa";
    "aa"+(4,2.4)="aa" **@{-}, "aa"+(3,1.8)="aa";
    "b5" **@{-},
"a3"="aa"; "aa"+(12,7.2)="aa" **@{-}, "aa"+(4,2.4)="aa";
    "aa"+(4,2.4)="aa" **@{-}, "aa"+(3,1.8)="aa";
    "aa"+(4,2.4)="aa" **@{-}, "aa"+(3,1.8)="aa";
    "b6" **@{-},
"a4"="aa"; "aa"+(5,3)="aa" **@{-}, "aa"+(4,2.4)="aa";
    "aa"+(4,2.4)="aa" **@{-}, "aa"+(3,1.8)="aa";
    "aa"+(4,2.4)="aa" **@{-}, "aa"+(3,1.8)="aa";
    "b7" **@{-},
"a5"; "b1" **@{-},
"a6"; "b2" **@{-},
"a7"; "b3" **@{-},
\end{xy}\\
\mbox{\normalsize (a) $X_{A,B}$} &&
\mbox{\normalsize (b) $X_{T,A}X_{A,L}$}&&
\mbox{\normalsize (c) $X_{T,L}$}
\rule{0pt}{2em}
\end{array}
$$
\caption{$X_{A,B}$, $X_{T,A}X_{A,L}$ and $X_{T,L}$
with $A=\{2,4,6\}$, $B=\{1,2,5\}$, $T=\{5,6,7\}$ and $L=\{1,2,3\}$.}
\label{F:X}
\end{figure}
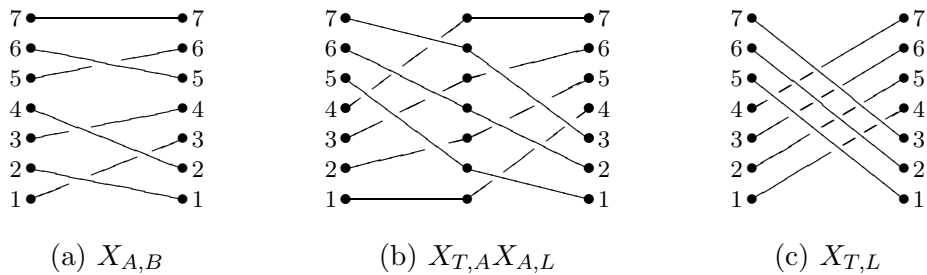

\begin{definition}
For $\alpha\in B_k$ and $\beta\in B_m$,
the notation $\alpha\splitBr\beta$ denotes
the $(k+m)$-braid obtained by stacking $\beta$ on top of $\alpha$ as
{\footnotesize \begin{xy}/r.8mm/:
(0,-1)="or", (50,0)="hd", (0,10)="vd",
(3,0)="hda", (0,2)="vda",
(4,0)="hd2", (0,1)="vd2",
"or"="aa",
"aa"-"hda"-"vda";
"aa"+"hda"-"vda" **@{-};
"aa"+"hda"+"vda" **@{-};
"aa"-"hda"+"vda" **@{-};
"aa"-"hda"-"vda" **@{-},
"aa" *{\alpha},
"aa"-"hda"-"vda"="bb",
"bb"+"vd2"="bb"; "bb"-"hd2" **@{-}, "bb"+"hda"+"hda"="cc"; "cc"+"hd2" **@{-},
"bb"+"vd2"="bb"; "bb"-"hd2" **@{-}, "bb"+"hda"+"hda"="cc"; "cc"+"hd2" **@{-},
"bb"+"vd2"="bb"; "bb"-"hd2" **@{-}, "bb"+"hda"+"hda"="cc"; "cc"+"hd2" **@{-},
"aa"+(0,5)="aa",
"aa"-"hda"-"vda";
"aa"+"hda"-"vda" **@{-};
"aa"+"hda"+"vda" **@{-};
"aa"-"hda"+"vda" **@{-};
"aa"-"hda"-"vda" **@{-},
"aa" *{\beta},
"aa"-"hda"-"vda"="bb",
"bb"+"vd2"="bb"; "bb"-"hd2" **@{-}, "bb"+"hda"+"hda"="cc"; "cc"+"hd2" **@{-},
"bb"+"vd2"="bb"; "bb"-"hd2" **@{-}, "bb"+"hda"+"hda"="cc"; "cc"+"hd2" **@{-},
"bb"+"vd2"="bb"; "bb"-"hd2" **@{-}, "bb"+"hda"+"hda"="cc"; "cc"+"hd2" **@{-},
\end{xy}
}.
Equivalently, if $\alpha$ and $\beta$ are represented by words
$W_1(\sigma_1,\ldots,\sigma_{k-1})$
and $W_2(\sigma_1,\ldots,\sigma_{m-1})$ respectively,
then $\alpha\splitBr\beta$ is represented by the word
$W_1(\sigma_1,\ldots,\sigma_{k-1})
W_2(\sigma_{k+1},\ldots,\sigma_{k+m-1})$.
\end{definition}

The following lemma is obvious. See Figure~\ref{F:spC}.

\begin{lemma}\label{L:spX}
Let $L=\{1,\ldots,k\}$ and $T=\{n-i:0\le i\le k-1\}$.
Then for $\alpha\in B_k$ and $\beta\in B_{n-k}$,
$$
X_{T,L}(\alpha\splitBr \beta)
=(\beta\splitBr\alpha) X_{T,L}.
$$
\end{lemma}

\begin{figure}\footnotesize
$$
\begin{array}{ccc}
\begin{xy}/r.33mm/:
(0,0)="or"="aa"="a1", (100,0)="hdL", (20,0)="hd",
"aa"+(0,10)="aa"="a2",
"aa"+(0,10)="aa"="a3",
"aa"+(0,10)="aa"="a4",
"aa"+(0,20)="aa"="a5",
"aa"+(0,10)="aa"="a6",
"aa"+(0,10)="aa"="a7",
"a1"+"hdL"="aa"="b1",
"aa"+(0,10)="aa"="b2",
"aa"+(0,10)="aa"="b3",
"aa"+(0,20)="aa"="b4",
"aa"+(0,10)="aa"="b5",
"aa"+(0,10)="aa"="b6",
"aa"+(0,10)="aa"="b7",
"a7"="aa"; "aa"+"hd"+(6,0)="aa" **@{-}; "aa"+(50,-50)="aa" **@{-}; "b3" **@{-},
"a6"="aa"; "aa"+"hd"+(3,0)="aa" **@{-}; "aa"+(50,-50)="aa" **@{-}; "b2" **@{-},
"a5"="aa"; "aa"+"hd"="aa" **@{-}; "aa"+(50,-50)="aa" **@{-}; "b1" **@{-},
"a4"="aa"; "aa"+"hd"="aa" **@{-};
    "aa"+(8,8)="aa" **@{-},
    "aa"+(3,3)="aa"; "aa"+(4,4)="aa" **@{-},
    "aa"+(3,3)="aa"; "aa"+(4,4)="aa" **@{-},
    "aa"+(3,3)="aa";
    "aa"+(15,15)="aa" **@{-};
    "b7" **@{-},
"a3"="aa"; "aa"+"hd"+(3,0)="aa" **@{-};
    "aa"+(12,12)="aa" **@{-},
    "aa"+(3,3)="aa"; "aa"+(4,4)="aa" **@{-},
    "aa"+(2,2)="aa"; "aa"+(4,4)="aa" **@{-},
    "aa"+(3,3)="aa";
    "aa"+(12,12)="aa" **@{-};
    "b6" **@{-},
"a2"="aa"; "aa"+"hd"+(6,0)="aa" **@{-};
    "aa"+(15,15)="aa" **@{-},
    "aa"+(3,3)="aa"; "aa"+(4,4)="aa" **@{-},
    "aa"+(3,3)="aa"; "aa"+(4,4)="aa" **@{-},
    "aa"+(3,3)="aa";
    "aa"+(8,8)="aa" **@{-};
    "b5" **@{-},
"a1"="aa"; "aa"+"hd"+(9,0)="aa" **@{-};
    "aa"+(19,19)="aa" **@{-},
    "aa"+(3,3)="aa"; "aa"+(4,4)="aa" **@{-},
    "aa"+(2,2)="aa"; "aa"+(4,4)="aa" **@{-},
    "aa"+(3,3)="aa";
    "aa"+(5,5)="aa" **@{-};
    "b4" **@{-},
(7,0)="hd", (0,14)="vd",
"b2"+"hd"="aa" *{\alpha},
    "aa"-"hd"-"vd";
    "aa"+"hd"-"vd" **@{-};
    "aa"+"hd"+"vd" **@{-};
    "aa"-"hd"+"vd" **@{-};
    "aa"-"hd"-"vd" **@{-};
(7,0)="hd", (0,19)="vd",
"b5"+(0,5)+"hd"="aa" *{\beta},
    "aa"-"hd"-"vd";
    "aa"+"hd"-"vd" **@{-};
    "aa"+"hd"+"vd" **@{-};
    "aa"-"hd"+"vd" **@{-};
    "aa"-"hd"-"vd" **@{-};
(15,0)="bb",
"b1"+"hd"+"hd"="aa"; "aa"+"bb" **@{-},
"b2"+"hd"+"hd"="aa"; "aa"+"bb" **@{-},
"b3"+"hd"+"hd"="aa"; "aa"+"bb" **@{-},
"b4"+"hd"+"hd"="aa"; "aa"+"bb" **@{-},
"b5"+"hd"+"hd"="aa"; "aa"+"bb" **@{-},
"b6"+"hd"+"hd"="aa"; "aa"+"bb" **@{-},
"b7"+"hd"+"hd"="aa"; "aa"+"bb" **@{-},
\end{xy}
&\qquad&
\qquad
\begin{xy}/r.33mm/:
(0,0)="or"="aa"="a1", (100,0)="hdL", (20,0)="hd",
"aa"+(0,10)="aa"="a2",
"aa"+(0,10)="aa"="a3",
"aa"+(0,10)="aa"="a4",
"aa"+(0,20)="aa"="a5",
"aa"+(0,10)="aa"="a6",
"aa"+(0,10)="aa"="a7",
"a1"+"hdL"="aa"="b1",
"aa"+(0,10)="aa"="b2",
"aa"+(0,10)="aa"="b3",
"aa"+(0,20)="aa"="b4",
"aa"+(0,10)="aa"="b5",
"aa"+(0,10)="aa"="b6",
"aa"+(0,10)="aa"="b7",
"a7"="aa"; "aa"+"hd"+(6,0)="aa" **@{-}; "aa"+(50,-50)="aa" **@{-}; "b3" **@{-},
"a6"="aa"; "aa"+"hd"+(3,0)="aa" **@{-}; "aa"+(50,-50)="aa" **@{-}; "b2" **@{-},
"a5"="aa"; "aa"+"hd"="aa" **@{-}; "aa"+(50,-50)="aa" **@{-}; "b1" **@{-},
"a4"="aa"; "aa"+"hd"="aa" **@{-};
    "aa"+(8,8)="aa" **@{-},
    "aa"+(3,3)="aa"; "aa"+(4,4)="aa" **@{-},
    "aa"+(3,3)="aa"; "aa"+(4,4)="aa" **@{-},
    "aa"+(3,3)="aa";
    "aa"+(15,15)="aa" **@{-};
    "b7" **@{-},
"a3"="aa"; "aa"+"hd"+(3,0)="aa" **@{-};
    "aa"+(12,12)="aa" **@{-},
    "aa"+(3,3)="aa"; "aa"+(4,4)="aa" **@{-},
    "aa"+(2,2)="aa"; "aa"+(4,4)="aa" **@{-},
    "aa"+(3,3)="aa";
    "aa"+(12,12)="aa" **@{-};
    "b6" **@{-},
"a2"="aa"; "aa"+"hd"+(6,0)="aa" **@{-};
    "aa"+(15,15)="aa" **@{-},
    "aa"+(3,3)="aa"; "aa"+(4,4)="aa" **@{-},
    "aa"+(3,3)="aa"; "aa"+(4,4)="aa" **@{-},
    "aa"+(3,3)="aa";
    "aa"+(8,8)="aa" **@{-};
    "b5" **@{-},
"a1"="aa"; "aa"+"hd"+(9,0)="aa" **@{-};
    "aa"+(19,19)="aa" **@{-},
    "aa"+(3,3)="aa"; "aa"+(4,4)="aa" **@{-},
    "aa"+(2,2)="aa"; "aa"+(4,4)="aa" **@{-},
    "aa"+(3,3)="aa";
    "aa"+(5,5)="aa" **@{-};
    "b4" **@{-},
(7,0)="hd", (0,14)="vd",
"a6"-"hd"="aa" *{\alpha},
    "aa"-"hd"-"vd";
    "aa"+"hd"-"vd" **@{-};
    "aa"+"hd"+"vd" **@{-};
    "aa"-"hd"+"vd" **@{-};
    "aa"-"hd"-"vd" **@{-};
(7,0)="hd", (0,19)="vd",
"a2"+(0,5)-"hd"="aa" *{\beta},
    "aa"-"hd"-"vd";
    "aa"+"hd"-"vd" **@{-};
    "aa"+"hd"+"vd" **@{-};
    "aa"-"hd"+"vd" **@{-};
    "aa"-"hd"-"vd" **@{-};
(15,0)="bb",
"a1"-"hd"-"hd"="aa"; "aa"-"bb" **@{-},
"a2"-"hd"-"hd"="aa"; "aa"-"bb" **@{-},
"a3"-"hd"-"hd"="aa"; "aa"-"bb" **@{-},
"a4"-"hd"-"hd"="aa"; "aa"-"bb" **@{-},
"a5"-"hd"-"hd"="aa"; "aa"-"bb" **@{-},
"a6"-"hd"-"hd"="aa"; "aa"-"bb" **@{-},
"a7"-"hd"-"hd"="aa"; "aa"-"bb" **@{-},
\end{xy}\\
\mbox{\normalsize (a) $X_{T,L}(\alpha\splitBr \beta)$} &&
\mbox{\normalsize (b) $(\beta\splitBr\alpha) X_{T,L}$}\rule{0pt}{2em}
\end{array}
$$
\caption{$X_{T,L}(\alpha\splitBr \beta)=(\beta\splitBr\alpha) X_{T,L}$,
where $L=\{1,2,3\}$ and $T=\{5,6,7\}$}
\label{F:spC}
\end{figure}
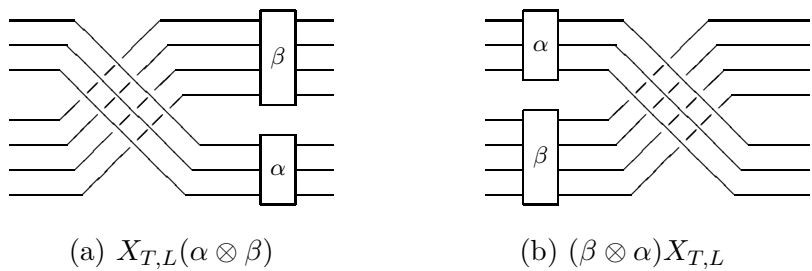

The following lemma is well known to experts.
For completeness, we include a simple proof.

\begin{lemma}\label{L:Pdec}
Let $\pi$ be an $n$-permutation,
and let $1\le k\le n$.
Then there is a decomposition
$$P(\pi)=(P_1\splitBr P_2) X_{L,A},$$
where $P_1\in B_k$ and $P_2\in B_{n-k}$ are permutation braids,
$L=\{1,\ldots,k\}$ and $A=\pi^{-1}(L)$.
\end{lemma}

Figure~\ref{f:P} shows an example of the above decomposition.

\begin{proof}
Let $\mathcal A_{n,k}$ be the set of all triples $(P_1,P_2,A)$
such that $P_1$ and $P_2$ are permutation braids in $B_k$ and in $B_{n-k}$,
respectively, and $A$ is a $k$-subset of $\{1,\ldots,n\}$.
Let $\mathcal P_n$ be the set of permutation $n$-braids.
Notice that $|\mathcal A_{n,k}|=k!\times (n-k)!\times{n\choose k}=n!=|\mathcal P_n|$.

Define $\phi:\mathcal A_{n,k}\to \mathcal P_n$
by $\phi(P_1,P_2,A)=(P_1\splitBr P_2) X_{L,A}$.
The braid $(P_1\splitBr P_2) X_{L,A}$ is a permutation braid
because it is a positive braid such that each pair of strands cross at most once.
Hence $\phi$ is well defined.
Considering induced permutations, it is easy to see that $\phi$ is injective.
Therefore $\phi$ is a bijection from $\mathcal A_{n,k}$ to $\mathcal P_n$,
hence any permutation $n$-braid $P(\pi)$
has the desired decomposition $P(\pi)=(P_1\splitBr P_2) X_{L,A}$.
\end{proof}

\begin{figure}\footnotesize
$$
\begin{xy}/r.4mm/:
(0,0)="or", (50,0)="hd", (5,0)="hda", (0,10)="vd",
     "or"="a1"="aa" *{\bullet} *+!R{1}, "aa"+"hd"="b1", "aa"+"hd"+"hd"="c1"="bb"; "bb"+(10,0) **@{-} *{\bullet} *+!L{1},
"aa"+"vd"="a2"="aa" *{\bullet} *+!R{2}, "aa"+"hd"="b2", "aa"+"hd"+"hd"="c2"="bb"; "bb"+(10,0) **@{-} *{\bullet} *{\bullet} *+!L{2},
"aa"+"vd"="a3"="aa" *{\bullet} *+!R{3}, "aa"+"hd"="b3", "aa"+"hd"+"hd"="c3"="bb"; "bb"+(10,0) **@{-} *{\bullet} *{\bullet} *+!L{3},
"aa"+"vd"="a4"="aa" *{\bullet} *+!R{4}, "aa"+"hd"="b4", "aa"+"hd"+"hd"="c4"="bb"; "bb"+(10,0) **@{-} *{\bullet} *{\bullet} *+!L{4},
"aa"+"vd"="a5"="aa" *{\bullet} *+!R{5}, "aa"+"hd"="b5", "aa"+"hd"+"hd"="c5"="bb"; "bb"+(10,0) **@{-} *{\bullet} *{\bullet} *+!L{5},
"aa"+"vd"="a6"="aa" *{\bullet} *+!R{6}, "aa"+"hd"="b6", "aa"+"hd"+"hd"="c6"="bb"; "bb"+(10,0) **@{-} *{\bullet} *{\bullet} *+!L{6},
"aa"+"vd"="a7"="aa" *{\bullet} *+!R{7}, "aa"+"hd"="b7", "aa"+"hd"+"hd"="c7"="bb"; "bb"+(10,0) **@{-} *{\bullet} *{\bullet} *+!L{7},
"a1"="aa"; "aa"+"hda"+(9,0)="aa" **@{-};
    "aa"+(2,4)="aa" **@{-},
    "aa"+(2,4)="aa";
    "aa"+(2,4)="aa" **@{-};
    "aa"+(2,4)="aa";
    "aa"+(2,4)="aa" **@{-};
    "b3" **@{-},
"a2"="aa"; "aa"+"hda"+(8,0)="aa" **@{-};
    "aa"+(10,-10)="aa" **@{-}; "b1" **@{-},
"a3"="aa"; "aa"+"hda"+(10,0)="aa" **@{-};
    "aa"+(10,-10)="aa" **@{-}; "b2" **@{-},
"a4"="aa"; "aa"+"hda"+(20,0)="aa" **@{-};
    "aa"+(2,4)="aa" **@{-},
    "aa"+(1,2)="aa";
    "aa"+(2,4)="aa" **@{-};
    "b5" **@{-},
"a5"="aa"; "aa"+"hda"+(9,0)="aa" **@{-};
    "aa"+(2,4)="aa" **@{-},
    "aa"+(2,4)="aa";
    "aa"+(2,4)="aa" **@{-};
    "aa"+(2,4)="aa";
    "aa"+(2,4)="aa" **@{-};
    "b7" **@{-},
"a6"="aa"; "aa"+"hda"+(8,0)="aa" **@{-};
    "aa"+(20,-20)="aa" **@{-}; "b4" **@{-},
"a7"="aa"; "aa"+"hda"+(10,0)="aa" **@{-};
    "aa"+(10,-10)="aa" **@{-}; "b6" **@{-},
"b7"; "c7" **@{-},
"b6"; "c4" **@{-},
"b5"; "c3" **@{-},
"b4"; "c1" **@{-},
"b3"="aa"; "aa"+(5,3)="aa" **@{-},
    "aa"+(5,3)="aa"; "aa"+(7.5,4.5)="aa" **@{-},
    "aa"+(5,3)="aa"; "aa"+(5,3)="aa" **@{-},
    "aa"+(5,3)="aa"; "c6" **@{-},
"b2"="aa"; "aa"+(15,9)="aa" **@{-},
    "aa"+(5,3)="aa"; "aa"+(7.5,4.5)="aa" **@{-},
    "aa"+(5,3)="aa"; "aa"+(5,3)="aa" **@{-},
    "aa"+(5,3)="aa"; "c5" **@{-},
"b1"="aa"; "aa"+(35,7)="aa" **@{-}, "aa"+(5,1)="aa"; "c2" **@{-},
"a2"+"hda"+(17,0)="B2",
    "B2"+(-12,-15);
    "B2"+( 12,-15) **@{-};
    "B2"+( 12, 13) **@{-};
    "B2"+(-12, 13) **@{-};
    "B2"+(-12,-15) **@{-};
    "B2"+(0,-14) *++!U{P_1},
"a5"+"hda"+(17,5)="B2",
    "B2"+(-12,-18);
    "B2"+( 12,-18) **@{-};
    "B2"+( 12, 20) **@{-};
    "B2"+(-12, 20) **@{-};
    "B2"+(-12,-18) **@{-};
    "B2"+(0,20) *+!D{P_2},
"b4"+(25,0)="B2",
    "B2"+(-28,-35);
    "B2"+( 28,-35) **@{-};
    "B2"+( 28, 35) **@{-};
    "B2"+(-28, 35) **@{-};
    "B2"+(-28,-35) **@{-};
    "B2"+(0,-37) *+!U{X_{L,A}},
\end{xy}
$$
\caption{The permutation 7-braid $P(\pi)$ with $\pi=(6,2,4,7,3,1,5)$
has decomposition $P(\pi)=(P_1\splitBr P_2)X_{L,A}$
with $P_1=\sigma_1\sigma_2$, $P_2=\sigma_2\sigma_3\sigma_1$,
$L=\{1,2,3\}$ and $A=\{2,5,6\}$.}
\label{f:P}
\end{figure}
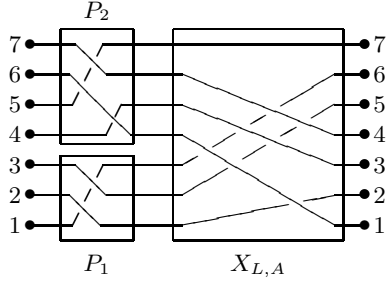

\begin{lemma}\label{L:X}
Let $L=\{1,\ldots,k\}$ and $T=\{n-i:0\le i\le k-1\}$.
The following hold.
\begin{enumerate}
\item[(i)] For $A=\{a_1,\ldots,a_k\}$ with $1\le a_1<\cdots<a_k\le n$,
\begin{align*}
& D_{a_1}\cdots D_{a_k}=X_{A,L}(\Delta_k\splitBr 1),\\
&E_{a_k}\cdots E_{a_1}=(\Delta_k\splitBr 1)X_{L,A},\\
&U(A)=X_{A,L} (\Delta_k^2\splitBr 1) X_{L,A}.
\end{align*}
\item[(ii)] $\delta^k=X_{T,L} (\Delta_k^2\splitBr 1)$.
\item[(iii)] $\delta^k(\alpha\splitBr \beta)=(\beta\splitBr\alpha)\delta^k$
for any $\alpha\in B_k$ and $\beta\in B_{n-k}$.
\end{enumerate}
\end{lemma}

\begin{proof}
(i)\ \
See Figures~\ref{F:aa}(a,b) and \ref{F:bb}.
The first two identities are obvious.
By Lemma~\ref{L:DEU}(iv),
\begin{align*}
U(A)&=U_{a_1}\cdots U_{a_k}
=(D_{a_1}\cdots D_{a_k})(E_{a_k}\cdots E_{a_1})\\
&= X_{A,L}(\Delta_k\splitBr 1)\cdot (\Delta_k\splitBr 1)X_{L,A}
= X_{A,L}(\Delta_k^2\splitBr 1)X_{L,A}.
\end{align*}

(ii)\ \
It is obvious that
$\delta^k=(1\splitBr \Delta_k)X_{T,L}(\Delta_k\splitBr 1)$.
See Figure~\ref{F:aa}(c).
By Lemma~\ref{L:spX},
$$
\delta^k
=(1\splitBr \Delta_k)X_{T,L}(\Delta_k\splitBr 1)
=X_{T,L}(\Delta_k\splitBr 1)(\Delta_k\splitBr 1)
=X_{T,L}(\Delta_k^2\splitBr 1).
$$

(iii)\ \
By (ii) and Lemma~\ref{L:spX}, we have
\begin{align*}
\delta^k(\alpha\splitBr \beta)
& = X_{T,L}(\Delta_k^2\splitBr 1)(\alpha\splitBr \beta)
= X_{T,L}(\alpha\splitBr \beta)(\Delta_k^2\splitBr 1)\\
&= (\beta\splitBr \alpha)X_{T,L}(\Delta_k^2\splitBr 1)
=(\beta\splitBr \alpha)\delta^k.
\end{align*}
\vskip-\baselineskip
\end{proof}

\begin{figure}\footnotesize
$$
\begin{array}{ccccc}
\begin{xy}/r.35mm/:
(20,0)="hd", (-12,60)="vd", (1,-5)="td", (70,0)="hdL",
(0,0)="or"="a1"="aa",
"aa"+(0,10)="aa"="a2",
"aa"+(0,10)="aa"="a3",
"aa"+(0,10)="aa"="a4",
"aa"+(0,10)="aa"="a5",
"aa"+(0,10)="aa"="a6",
"aa"+(0,10)="aa"="a7",
"a7"+"hdL"="aa"="b7",
"aa"-(0,20)="aa"="b6",
"aa"-(0,30)="aa"="b5",
"aa"-(0,10)="aa"="b4",
"aa"-(0,15)="aa"="b3",
"aa"-(0,10)="aa"="b2",
"aa"-(0,10)="aa"="b1",
"a7"; "b7" **@{-},
"a6"="aa"; "aa"+"hd"+(14,0)="aa" **@{-};
    "aa"+(16,-80)+"td"="aa" **@{-}; "b1" **@{-},
"a5"="aa"; "aa"+"hd"+(14,0)="aa" **@{-}, "aa"+(4,0)="aa"; "b6" **@{-},
"a4"="aa"; "aa"+"hd"+(8,0)="aa" **@{-};
    "aa"+(10,-50)+"td"="aa" **@{-};
    "aa"+(8,0)="aa" **@{-}, "aa"+(4,0)="aa"; "b2" **@{-},
"a3"="aa"; "aa"+"hd"="aa" **@{-};
    "aa"+(6,-30)+"td"="aa" **@{-};
    "aa"+(8,0)="aa" **@{-}, "aa"+(4,0)="aa";
    "aa"+(6,0)="aa" **@{-}, "aa"+(4,0)="aa"; "b3" **@{-},
"a2"="aa"; "aa"+"hd"+(0,0)="aa" **@{-},
    "aa"+(4,0)="aa"; "aa"+(6,0)="aa" **@{-};
    "aa"+(4,0)="aa"; "aa"+(6,0)="aa" **@{-};
    "aa"+(4,0)="aa"; "b5" **@{-},
"a1"="aa"; "aa"+"hd"+(2,0)="aa" **@{-},
    "aa"+(4,0)="aa"; "aa"+(6,0)="aa" **@{-};
    "aa"+(4,0)="aa"; "aa"+(6,0)="aa" **@{-};
    "aa"+(4,0)="aa"; "b4" **@{-},
"a1" *{\bullet} *+!R{1},
"a2" *{\bullet} *+!R{2},
"a3" *{\bullet} *+!R{3},
"a4" *{\bullet} *+!R{4},
"a5" *{\bullet} *+!R{5},
"a6" *{\bullet} *+!R{6},
"a7" *{\bullet} *+!R{7},
"b1" *{\bullet} *+!L{1},
"b2" *{\bullet} *+!L{2},
"b3" *{\bullet} *+!L{3},
"b4" *{\bullet} *+!L{4},
"b5" *{\bullet} *+!L{5},
"b6" *{\bullet} *+!L{6},
"b7" *{\bullet} *+!L{7},
\end{xy}
&\quad&
\begin{xy}/r.35mm/:
(20,0)="hd", (-12,60)="vd", (-1,-5)="td", (70,0)="hdL",
(0,0)="or"="a1"="aa",
"aa"+(0,10)="aa"="a2",
"aa"+(0,10)="aa"="a3",
"aa"+(0,10)="aa"="a4",
"aa"+(0,10)="aa"="a5",
"aa"+(0,10)="aa"="a6",
"aa"+(0,10)="aa"="a7",
"a7"-"hdL"="aa"="b7",
"aa"-(0,20)="aa"="b6",
"aa"-(0,30)="aa"="b5",
"aa"-(0,10)="aa"="b4",
"aa"-(0,15)="aa"="b3",
"aa"-(0,10)="aa"="b2",
"aa"-(0,10)="aa"="b1",
"a7"; "b7" **@{-},
"a6"="aa"; "aa"-"hd"-(14,0)="aa" **@{-};
    "aa"-(1.5,7.5)="aa" **@{-}, "aa"-(1,5)="aa";
    "aa"+(-5,-25)="aa" **@{-}, "aa"-(1,5)="aa";
    "aa"+(-1,-5)="aa" **@{-}, "aa"-(1,5)="aa";
    "aa"+(-2,-10)="aa" **@{-}, "aa"-(1,5)="aa";
    "aa"+(-1,-5)="aa" **@{-}, "aa"-(1,5)="aa";
    "aa"+(-1.5,-7.5)="aa" **@{-}; "b1" **@{-},
"a5"="aa"; "b6" **@{-},
"a4"="aa"; "aa"-"hd"-(8,0)="aa" **@{-};
    "aa"-(3.5,17.5)="aa" **@{-}, "aa"-(1,5)="aa";
    "aa"+(-1,-5)="aa" **@{-}, "aa"-(1,5)="aa";
    "aa"+(-2,-10)="aa" **@{-}, "aa"-(1,5)="aa";
    "aa"+(-1.5,-7.5)="aa" **@{-}; "b2" **@{-},
"a3"="aa"; "aa"-"hd"="aa" **@{-};
    "aa"-(1.5,7.5)="aa" **@{-}, "aa"-(1,5)="aa";
    "aa"+(-1,-5)="aa" **@{-}, "aa"-(1,5)="aa";
    "aa"+(-2.5,-12.5)="aa" **@{-}; "b3" **@{-},
"a2"="aa"; "b5" **@{-},
"a1"="aa"; "b4" **@{-},
"b1" *{\bullet} *+!R{1},
"b2" *{\bullet} *+!R{2},
"b3" *{\bullet} *+!R{3},
"b4" *{\bullet} *+!R{4},
"b5" *{\bullet} *+!R{5},
"b6" *{\bullet} *+!R{6},
"b7" *{\bullet} *+!R{7},
"a1" *{\bullet} *+!L{1},
"a2" *{\bullet} *+!L{2},
"a3" *{\bullet} *+!L{3},
"a4" *{\bullet} *+!L{4},
"a5" *{\bullet} *+!L{5},
"a6" *{\bullet} *+!L{6},
"a7" *{\bullet} *+!L{7},
\end{xy}
&\quad&
\begin{xy}/r.35mm/:
(15,0)="hd", (-15,75)="vd",
(0,0)="or"="a1"="aa",
"aa"+(0,10)="aa"="a2",
"aa"+(0,10)="aa"="a3",
"aa"+(0,10)="aa"="a4",
"aa"+(0,10)="aa"="a5",
"aa"+(0,10)="aa"="a6",
"aa"+(0,10)="aa"="a7",
"a1"+(70,-35)="aa"="b1",
"aa"+(0,10)="aa"="b2",
"aa"+(0,10)="aa"="b3",
"aa"+(0,15)="aa"="b4",
"aa"+(0,10)="aa"="b5",
"aa"+(0,10)="aa"="b6",
"aa"+(0,10)="aa"="b7",
"a7"="aa"; "aa"+"hd"="aa" **@{-};
    "aa"-"vd"="aa" **@{-};
    "aa"+(7,0)="aa" **@{-}, "aa"+(4,0)="aa";
    "aa"+(7,0)="aa" **@{-}, "aa"+(4,0)="aa"; "b3" **@{-},
"a6"="aa"; "aa"+"hd"-(.5,0)="aa" **@{-}, "aa"+(4,0)="aa";
    "aa"+(8,0)="aa" **@{-};
    "aa"-"vd"="aa" **@{-};
    "aa"+(8,0)="aa" **@{-}, "aa"+(4,0)="aa"; "b2" **@{-},
"a5"="aa"; "aa"+"hd"+(2,0)="aa" **@{-},
    "aa"+(4,0)="aa"; "aa"+(6,0)="aa" **@{-};
    "aa"+(4,0)="aa"; "aa"+(8,0)="aa" **@{-};
    "aa"-"vd"="aa" **@{-}; "b1" **@{-},
"a4"="aa"; "aa"+"hd"+(4,0)="aa" **@{-},
    "aa"+(4,0)="aa"; "aa"+(6,0)="aa" **@{-};
    "aa"+(4,0)="aa"; "aa"+(6,0)="aa" **@{-};
    "aa"+(4,0)="aa"; "b7" **@{-},
"a3"="aa"; "aa"+"hd"+(6,0)="aa" **@{-},
    "aa"+(4,0)="aa"; "aa"+(6,0)="aa" **@{-};
    "aa"+(4,0)="aa"; "aa"+(6,0)="aa" **@{-};
    "aa"+(4,0)="aa"; "b6" **@{-},
"a2"="aa"; "aa"+"hd"+(8,0)="aa" **@{-},
    "aa"+(4,0)="aa"; "aa"+(6,0)="aa" **@{-};
    "aa"+(4,0)="aa"; "aa"+(6,0)="aa" **@{-};
    "aa"+(4,0)="aa"; "b5" **@{-},
"a1"="aa"; "aa"+"hd"+(10,0)="aa" **@{-},
    "aa"+(4,0)="aa"; "aa"+(6,0)="aa" **@{-};
    "aa"+(4,0)="aa"; "aa"+(6,0)="aa" **@{-};
    "aa"+(4,0)="aa"; "b4" **@{-},
"a1" *{\bullet} *+!R{1},
"a2" *{\bullet} *+!R{2},
"a3" *{\bullet} *+!R{3},
"a4" *{\bullet} *+!R{4},
"a5" *{\bullet} *+!R{5},
"a6" *{\bullet} *+!R{6},
"a7" *{\bullet} *+!R{7},
"b1" *{\bullet} *+!L{1},
"b2" *{\bullet} *+!L{2},
"b3" *{\bullet} *+!L{3},
"b4" *{\bullet} *+!L{4},
"b5" *{\bullet} *+!L{5},
"b6" *{\bullet} *+!L{6},
"b7" *{\bullet} *+!L{7},
\end{xy}\\ 
\mbox{\normalsize (a) $D_3D_4D_6$}&&
\mbox{\normalsize (b) $E_6E_4E_3$}&&
\mbox{\normalsize (c) $\delta^3$}\rule{0pt}{2em}
\end{array}
$$
\caption{The braids $D_3D_4D_6$, $E_6E_4E_3$
and $\delta^3$ in $B_7$}
\label{F:aa}
\end{figure}
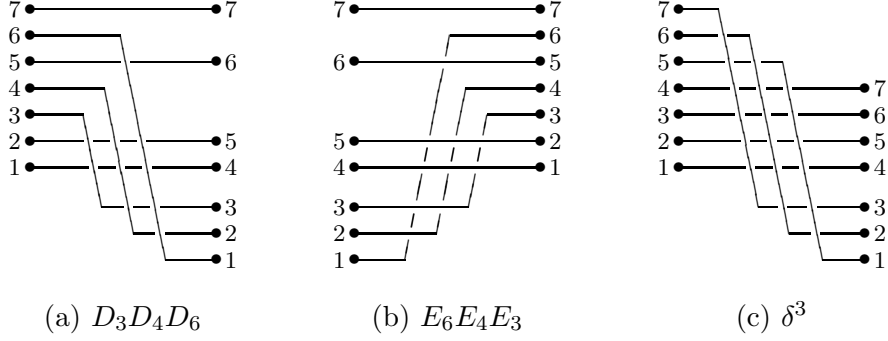

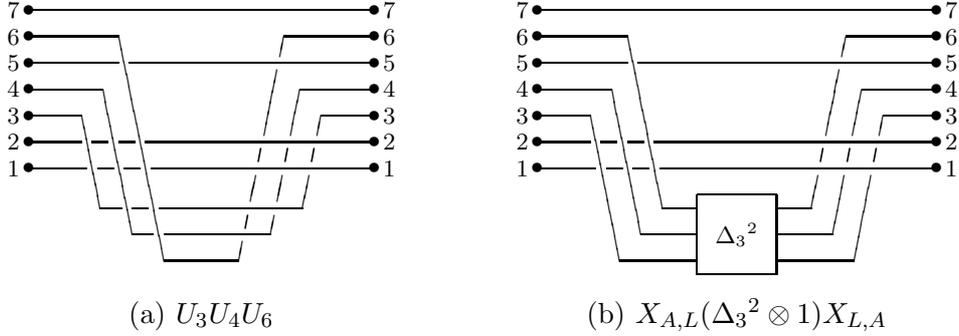
\begin{figure}\footnotesize
$$
\begin{array}{ccc}
\begin{xy}/r.35mm/:
(20,0)="hd", (-12,60)="vd", (1,-5)="td", (70,0)="hdL",
(0,0)="or"="a1"="aa",
"aa"+(0,10)="aa"="a2",
"aa"+(0,10)="aa"="a3",
"aa"+(0,10)="aa"="a4",
"aa"+(0,10)="aa"="a5",
"aa"+(0,10)="aa"="a6",
"aa"+(0,10)="aa"="a7",
"a7"+"hdL"="aa"="b7",
"aa"-(0,20)="aa"="b6",
"aa"-(0,30)="aa"="b5",
"aa"-(0,10)="aa"="b4",
"aa"-(0,15)="aa"="b3",
"aa"-(0,10)="aa"="b2",
"aa"-(0,10)="aa"="b1",
"a7"; "b7" **@{-},
"a6"="aa"; "aa"+"hd"+(14,0)="aa" **@{-};
    "aa"+(16,-80)+"td"="aa" **@{-}; "b1" **@{-},
"a5"="aa"; "aa"+"hd"+(14,0)="aa" **@{-}, "aa"+(4,0)="aa"; "b6" **@{-},
"a4"="aa"; "aa"+"hd"+(8,0)="aa" **@{-};
    "aa"+(10,-50)+"td"="aa" **@{-};
    "aa"+(8,0)="aa" **@{-}, "aa"+(4,0)="aa"; "b2" **@{-},
"a3"="aa"; "aa"+"hd"="aa" **@{-};
    "aa"+(6,-30)+"td"="aa" **@{-};
    "aa"+(8,0)="aa" **@{-}, "aa"+(4,0)="aa";
    "aa"+(6,0)="aa" **@{-}, "aa"+(4,0)="aa"; "b3" **@{-},
"a2"="aa"; "aa"+"hd"+(0,0)="aa" **@{-},
    "aa"+(4,0)="aa"; "aa"+(6,0)="aa" **@{-};
    "aa"+(4,0)="aa"; "aa"+(6,0)="aa" **@{-};
    "aa"+(4,0)="aa"; "b5" **@{-},
"a1"="aa"; "aa"+"hd"+(2,0)="aa" **@{-},
    "aa"+(4,0)="aa"; "aa"+(6,0)="aa" **@{-};
    "aa"+(4,0)="aa"; "aa"+(6,0)="aa" **@{-};
    "aa"+(4,0)="aa"; "b4" **@{-},
"a1" *{\bullet} *+!R{1},
"a2" *{\bullet} *+!R{2},
"a3" *{\bullet} *+!R{3},
"a4" *{\bullet} *+!R{4},
"a5" *{\bullet} *+!R{5},
"a6" *{\bullet} *+!R{6},
"a7" *{\bullet} *+!R{7},
(20,0)="hd", (-12,60)="vd", (-1,-5)="td", (70,0)="hdL",
(130,0)="or"="a1"="aa",
"aa"+(0,10)="aa"="a2",
"aa"+(0,10)="aa"="a3",
"aa"+(0,10)="aa"="a4",
"aa"+(0,10)="aa"="a5",
"aa"+(0,10)="aa"="a6",
"aa"+(0,10)="aa"="a7",
"a7"-"hdL"="aa"="b7",
"aa"-(0,20)="aa"="b6",
"aa"-(0,30)="aa"="b5",
"aa"-(0,10)="aa"="b4",
"aa"-(0,15)="aa"="b3",
"aa"-(0,10)="aa"="b2",
"aa"-(0,10)="aa"="b1",
"a7"; "b7" **@{-},
"a6"="aa"; "aa"-"hd"-(14,0)="aa" **@{-};
    "aa"-(1.5,7.5)="aa" **@{-}, "aa"-(1,5)="aa";
    "aa"+(-5,-25)="aa" **@{-}, "aa"-(1,5)="aa";
    "aa"+(-1,-5)="aa" **@{-}, "aa"-(1,5)="aa";
    "aa"+(-2,-10)="aa" **@{-}, "aa"-(1,5)="aa";
    "aa"+(-1,-5)="aa" **@{-}, "aa"-(1,5)="aa";
    "aa"+(-1.5,-7.5)="aa" **@{-}; "b1" **@{-},
"a5"="aa"; "b6" **@{-},
"a4"="aa"; "aa"-"hd"-(8,0)="aa" **@{-};
    "aa"-(3.5,17.5)="aa" **@{-}, "aa"-(1,5)="aa";
    "aa"+(-1,-5)="aa" **@{-}, "aa"-(1,5)="aa";
    "aa"+(-2,-10)="aa" **@{-}, "aa"-(1,5)="aa";
    "aa"+(-1.5,-7.5)="aa" **@{-}; "b2" **@{-},
"a3"="aa"; "aa"-"hd"="aa" **@{-};
    "aa"-(1.5,7.5)="aa" **@{-}, "aa"-(1,5)="aa";
    "aa"+(-1,-5)="aa" **@{-}, "aa"-(1,5)="aa";
    "aa"+(-2.5,-12.5)="aa" **@{-}; "b3" **@{-},
"a2"="aa"; "b5" **@{-},
"a1"="aa"; "b4" **@{-},
"a1" *{\bullet} *+!L{1},
"a2" *{\bullet} *+!L{2},
"a3" *{\bullet} *+!L{3},
"a4" *{\bullet} *+!L{4},
"a5" *{\bullet} *+!L{5},
"a6" *{\bullet} *+!L{6},
"a7" *{\bullet} *+!L{7},
\end{xy}
&\qquad&
\begin{xy}/r.35mm/:
(20,0)="hd", (-12,60)="vd", (1,-5)="td", (60,0)="hdL",
(10,0)="bxa", "bxa"+"bxa"="bx",
(0,0)="or"="a1"="aa",
"aa"+(0,10)="aa"="a2",
"aa"+(0,10)="aa"="a3",
"aa"+(0,10)="aa"="a4",
"aa"+(0,10)="aa"="a5",
"aa"+(0,10)="aa"="a6",
"aa"+(0,10)="aa"="a7",
"a7"+"hdL"="aa"="b7",
"aa"-(0,20)="aa"="b6",
"aa"-(0,30)="aa"="b5",
"aa"-(0,10)="aa"="b4",
"aa"-(0,15)="aa"="b3",
"aa"-(0,10)="aa"="b2",
"aa"-(0,10)="aa"="b1",
"a1"+"hdL"+"hdL"+"bx"="aa"="c1",
"aa"+(0,10)="aa"="c2",
"aa"+(0,10)="aa"="c3",
"aa"+(0,10)="aa"="c4",
"aa"+(0,10)="aa"="c5",
"aa"+(0,10)="aa"="c6",
"aa"+(0,10)="aa"="c7",
"a7"; "b7" **@{-},
"a6"="aa"; "aa"+"hd"+(14,0)="aa" **@{-};
    "aa"+(12,-60)+"td"="aa" **@{-}; "b3" **@{-},
"a5"="aa"; "aa"+"hd"+(14,0)="aa" **@{-}, "aa"+(4,0)="aa"; "b6" **@{-},
"a4"="aa"; "aa"+"hd"+(8,0)="aa" **@{-};
    "aa"+(10,-50)+"td"="aa" **@{-}; "b2" **@{-},
"a3"="aa"; "aa"+"hd"="aa" **@{-};
    "aa"+(10,-50)+"td"="aa" **@{-}; "b1" **@{-},
"a2"="aa"; "aa"+"hd"+(0,0)="aa" **@{-},
    "aa"+(4,0)="aa"; "aa"+(6,0)="aa" **@{-};
    "aa"+(4,0)="aa"; "aa"+(6,0)="aa" **@{-};
    "aa"+(4,0)="aa"; "b5" **@{-},
"a1"="aa"; "aa"+"hd"+(2,0)="aa" **@{-},
    "aa"+(4,0)="aa"; "aa"+(6,0)="aa" **@{-};
    "aa"+(4,0)="aa"; "aa"+(6,0)="aa" **@{-};
    "aa"+(4,0)="aa"; "b4" **@{-},
"a1" *{\bullet} *+!R{1},
"a2" *{\bullet} *+!R{2},
"a3" *{\bullet} *+!R{3},
"a4" *{\bullet} *+!R{4},
"a5" *{\bullet} *+!R{5},
"a6" *{\bullet} *+!R{6},
"a7" *{\bullet} *+!R{7},
(20,0)="hd", (-12,60)="vd", (-1,-5)="td", (60,0)="hdL",
"or"+"hdL"+"hdL"+"bx"+(10,0)="or"="a1"="aa",
"aa"+(0,10)="aa"="a2",
"aa"+(0,10)="aa"="a3",
"aa"+(0,10)="aa"="a4",
"aa"+(0,10)="aa"="a5",
"aa"+(0,10)="aa"="a6",
"aa"+(0,10)="aa"="a7",
"a7"-"hdL"="aa"="b7",
"aa"-(0,20)="aa"="b6",
"aa"-(0,30)="aa"="b5",
"aa"-(0,10)="aa"="b4",
"aa"-(0,15)="aa"="b3",
"aa"-(0,10)="aa"="b2",
"aa"-(0,10)="aa"="b1",
"a7"; "b7" **@{-},
"a6"="aa"; "aa"-"hd"-(14,0)="aa" **@{-};
    "aa"-(1.5,7.5)="aa" **@{-}, "aa"-(1,5)="aa";
    "aa"+(-5,-25)="aa" **@{-}, "aa"-(1,5)="aa";
    "aa"+(-1,-5)="aa" **@{-}, "aa"-(1,5)="aa";
    "aa"+(-1.5,-7.5)+(-1,-5)="aa" **@{-}; "b3" **@{-},
"a5"="aa"; "b6" **@{-},
"a4"="aa"; "aa"-"hd"-(8,0)="aa" **@{-};
    "aa"-(3.5,17.5)="aa" **@{-}, "aa"-(1,5)="aa";
    "aa"+(-1,-5)="aa" **@{-}, "aa"-(1,5)="aa";
    "aa"+(-1.5,-7.5)-(3,15)="aa" **@{-}; "b2" **@{-},
"a3"="aa"; "aa"-"hd"="aa" **@{-};
    "aa"-(1.5,7.5)="aa" **@{-}, "aa"-(1,5)="aa";
    "aa"+(-1,-5)="aa" **@{-}, "aa"-(1,5)="aa";
    "aa"+(-2.5,-12.5)+(-4,-20)="aa" **@{-}; "b1" **@{-},
"a2"="aa"; "b5" **@{-},
"a1"="aa"; "b4" **@{-},
"a1" *{\bullet} *+!L{1},
"a2" *{\bullet} *+!L{2},
"a3" *{\bullet} *+!L{3},
"a4" *{\bullet} *+!L{4},
"a5" *{\bullet} *+!L{5},
"a6" *{\bullet} *+!L{6},
"a7" *{\bullet} *+!L{7},
"b7"="aa"; "aa"-"bx"-(10,0) **@{-},
"b6"="aa"; "aa"-"bx"-(10,0) **@{-},
"b5"="aa"; "aa"-"bx"-(10,0) **@{-},
"b4"="aa"; "aa"-"bx"-(10,0) **@{-},
"b2"+(-15,0)="aa" *{{\Delta_3}^2},
"aa"+(-15,-15);
"aa"+( 15,-15) **@{-};
"aa"+( 15, 15) **@{-};
"aa"+(-15, 15) **@{-};
"aa"+(-15,-15) **@{-};
\end{xy}\\
\mbox{\normalsize (a) $U_3U_4U_6$}&&
\mbox{\normalsize (b) $X_{A,L}({\Delta_3}^2\splitBr 1)X_{L,A}$}\rule{0pt}{2em}
\end{array}
$$
\caption{$U_3U_4U_6=X_{A,L}({\Delta_3}^2\splitBr 1)X_{L,A}$,
where $A=\{3,4,6\}$ and $L=\{1,2,3\}$}
\label{F:bb}
\end{figure}

\begin{proposition}
Let $L=\{1,\ldots,k\}$ and $T=\{n-i: 0\le i\le k-1\}$.
For any $k$-subset $A$ of $\{1,\ldots,n\}$,
$$
U(A)=\bar X_{A,T}\delta^k X_{L,A}.
$$
\end{proposition}

\begin{proof}
By Lemma~\ref{L:TA}, $X_{T,L}=X_{T,A} X_{A,L}$, hence
$X_{A,L}=X_{T,A}^{-1}X_{T,L}=\bar X_{A,T}X_{T,L}$.
By Lemma~\ref{L:X},
$U(A)=X_{A,L} (\Delta_k^2\splitBr 1) X_{L,A}$ and
$\delta^k=X_{T,L} (\Delta_k^2\splitBr 1)$.
Combining these,
$$
U(A)
= X_{A,L} (\Delta_k^2\splitBr 1) X_{L,A}
= \bar X_{A,T}X_{T,L} (\Delta_k^2\splitBr 1) X_{L,A}
= \bar X_{A,T}\delta^k X_{L,A}.
$$
\par\vskip-\baselinestretch\baselineskip
\end{proof}

\begin{theorem}\label{T:Br}
Let $n$ and $k$ be relatively prime integers such that $2\le k\le n-1$.
Then the $n$-braid $\delta^k$ is conjugate to
$$
\delta U_{a_1}\cdots U_{a_{k-1}},
$$
where $a_i=\left\lceil \frac {ni}k\right\rceil$
for $1\le i\le k-1$.
More precisely,
if $X_k=P(\pi_k)$ where $\pi_k$ is the $n$-permutation defined by
$\pi_k(i)\equiv ki\bmod n$ for $1\le i\le n$,
then
$$
X_k^{-1}\delta^k X_k=\delta U_{a_1}\cdots U_{a_{k-1}}.
$$
\end{theorem}

\begin{proof}
Let $A=\{a_1,\ldots,a_{k-1}\}$, $L=\{1,\ldots,k-1\}$ and $T=\{n-i:0\le i\le k-2\}$.
Let
$$\alpha_k=\tau(X_k)^{-1}\delta^{k-1} X_k.$$
Since
$\delta \alpha_k
=\delta \tau(X_k)^{-1}\delta^{k-1} X_k
=\delta \cdot \delta^{-1} X_k^{-1}\delta \cdot \delta^{k-1} X_k
=X_k^{-1}\delta^k X_k$,
it suffices to show
$$
\alpha_k=U_{a_1}\cdots U_{a_{k-1}}\, .
$$

Since $\pi_k(n)\equiv kn\equiv n\bmod n$, the $n$-th strand of $X_k$
does not cross any other strands, hence $X_k=X\splitBr 1$
for some permutation $(n-1)$-braid $X$.
Therefore $\tau(X_k)=\delta^{-1} X_k\delta=1\splitBr X$ is a permutation braid.
(See Figure~\ref{F:ex7} for the braid diagrams of $X_k$ and $\tau(X_k)$
for the case where $n=7$ and $k=3$.)

\medskip\noindent
\textbf{Claim 1.}\ \
$\pi_k(A)=L$ and $(\pi_{\delta^{-1}}\circ\pi_k\circ\pi_{\delta})(A)=T$.

\smallskip
\begin{proof}[Proof of Claim 1]
Let $1\le i\le k-1$.
Since $a_i=\lceil\frac{ni}k\rceil$, there is an integer $0\le r_i\le k-1$ such that
$$
ni=ka_i-r_i.
$$
Since $n$ and $k$ are relatively prime, $ni$ is not a multiple of $k$,
hence $r_i\ne 0$.
If $r_i=r_j$ for some $1\le i<j\le k-1$,
then $n(j-i)=k(a_j-a_i)$, hence $n(j-i)$ is a multiple of $k$.
This contradicts that $n$ and $k$ are relatively prime because $1\le j-i<k$.
Therefore $r_1,\ldots,r_{k-1}$ are mutually distinct, hence
$$
\{r_1,\ldots,r_{k-1}\}=\{1,\ldots,k-1\}=L.
$$
Since $\pi_k(a_i)\equiv ka_i=ni+r_i\equiv r_i\bmod n$,
we have $\pi_k(a_i)=r_i$ for $1\le i\le k-1$. Therefore
$$
\pi_k(A)=L.
$$

On the other hand,
$\pi_{\delta^{-1}}\circ\pi_k\circ\pi_{\delta}$ sends each $a_i$
up to modulo $n$
as follows.
$$
a_i
\stackrel{\pi_{\delta}}{\longmapsto} a_i-1
\stackrel{\pi_k}{\longmapsto} k(a_i-1)
\stackrel{\pi_{\delta^{-1}}}{\longmapsto} k(a_i-1)+1=ka_i-k+1
$$
Notice that $ka_i-k+1=ni+r_i-k+1\equiv n+r_i-k+1\bmod n$.
Since $1\le r_i\le k-1$, we have $n+2-k\le n+r_i-k+1\le n$.
Therefore
$$
(\pi_{\delta^{-1}}\circ\pi_k\circ\pi_{\delta})(a_i)=n+r_i-k+1\in T.
$$
Since $|A|=|T|$, we have $(\pi_{\delta^{-1}}\circ\pi_k\circ\pi_{\delta})(A)= T$.
\end{proof}

\medskip\noindent
\textbf{Claim 2.}\ \
$X_k$ and $\tau(X_k)$ have decompositions
\begin{align*}
X_k&=(P_1\splitBr P_2) X_{L,A},\\
\tau(X_k)&=(Q_2\splitBr Q_1) X_{T,A},
\end{align*}
where $P_1$ and $Q_1$ (resp.\ $P_2$ and $Q_2$)
are permutation braids in $B_{k-1}$ (resp.\ $B_{n-k+1}$).
\smallskip

\begin{proof}[Proof of Claim 2]
$X_k$ is the permutation braid with induced permutation $\pi_k$.
Since $\pi_k$ sends $A$ to $L$ by Claim 1,
$X_k$ has the desired decomposition
$X_k=(P_1\splitBr P_2) X_{L,A}$
for permutation braids $P_1$ and $P_2$ in $B_{k-1}$ and in $B_{n-k+1}$
respectively (by Lemma~\ref{L:Pdec}).

$\tau(X_k)=\delta^{-1} X_k\delta$ is the permutation braid
with induced permutation $\pi_{\delta^{-1}}\circ\pi_k\circ\pi_{\delta}$.
Since $\pi_{\delta^{-1}}\circ\pi_k\circ\pi_{\delta}$ sends
$A$ to $T$ (by Claim 1),
$\tau(X_k)$ has the desired decomposition
$\tau(X_k)=(Q_2\splitBr Q_1) X_{T,A}$
for permutation braids $Q_1$ and $Q_2$ in $B_{k-1}$ and in $B_{n-k+1}$
respectively (by Lemma~\ref{L:Pdec}).
\end{proof}

\medskip\noindent
\textbf{Claim 3.}\ \
$\alpha_k$ is a pure braid.
\smallskip

\begin{proof}[Proof of Claim 3]
Since $\alpha_k=\tau(X_k)^{-1}\delta^{k-1} X_k=\delta^{-1} X_k^{-1}\delta^k X_k$,
the induced permutation of $\alpha_k$
is $\pi_{\alpha_k}=\pi_{\delta^{-1}}\circ \pi_k^{-1}\circ\pi_{\delta^k}\circ \pi_k$.
Up to modulo $n$, $\pi_{\alpha_k}$ sends $1\le i\le n$ as follows.
$$
i
\stackrel{\pi_k}{\longmapsto} ki
\stackrel{\pi_{\delta^{k}}}{\longmapsto} ki-k=k(i-1)
\stackrel{\pi_k^{-1}}{\longmapsto} (i-1)
\stackrel{\pi_{\delta^{-1}}}{\longmapsto} i
$$
Therefore $\alpha_k$ is a pure braid.
\end{proof}

By Lemmas~\ref{L:X},~\ref{L:TA} and~\ref{L:spX}, we have
$$
\delta^{k-1}=X_{T,L}(\Delta_{k-1}^2\splitBr 1),\quad
X_{T,A}^{-1}X_{T,L}=X_{A,L},\quad
(Q_2^{-1}\splitBr Q_1^{-1})X_{T,L}=X_{T,L}(Q_1^{-1}\splitBr Q_2^{-1}).
$$
Using the decompositions $X_k=(P_1\splitBr P_2) X_{L,A}$
and $\tau(X_k)=(Q_2\splitBr Q_1) X_{T,A}$ in Claim 2,
\begin{align*}
\alpha_k
&=\tau(X_k)^{-1}\delta^{k-1}X_k\\
&=X_{T,A}^{-1} (Q_2^{-1}\splitBr Q_1^{-1}) ~\cdot~
    X_{T,L}(\Delta_{k-1}^2\splitBr 1) ~\cdot~
    (P_1\splitBr P_2) X_{L,A}\\
&= X_{T,A}^{-1} ~\cdot~X_{T,L} ~\cdot~
(Q_1^{-1}\splitBr Q_2^{-1})(\Delta_{k-1}^2\splitBr 1)(P_1\splitBr P_2) ~\cdot~ X_{L,A}\\
&=X_{T,A}^{-1} X_{T,L}~\cdot~
(\Delta_{k-1}^2\splitBr 1)(Q_1^{-1}P_1\splitBr Q_2^{-1}P_2)~\cdot~ X_{L,A}\\
&=X_{A,L}~\cdot~
(\Delta_{k-1}^2\splitBr 1)(Q_1^{-1}P_1\splitBr Q_2^{-1}P_2)~\cdot~ X_{L,A}.
\end{align*}
See Figure~\ref{F:al} for the relation
$$(Q_2^{-1}\splitBr Q_1^{-1})X_{T,L}(\Delta_{k-1}^2\splitBr 1)(P_1\splitBr P_2)
=X_{T,L}(\Delta_{k-1}^2\splitBr 1)(Q_1^{-1}P_1\splitBr Q_2^{-1}P_2).
$$
The braid $\alpha_k$ is a pure braid (by Claim 3).
If we delete $Q_1^{-1}P_1\splitBr Q_2^{-1}P_2$
from the expression of $\alpha_k$,
we have $X_{A,L} (\Delta_{k-1}^2\splitBr 1) X_{L,A}$,
which is $U(A)$ (by Lemma~\ref{L:X}(i)),
hence it is also a pure braid.
Because both $\alpha_k$ and  $X_{A,L} (\Delta_{k-1}^2\splitBr 1) X_{L,A}$
are pure braids, $Q_1^{-1}P_1\splitBr Q_2^{-1}P_2$ is also a pure braid.
Since $P_i$'s and $Q_i$'s are permutation braids,
both $Q_1^{-1}P_1$ and $Q_2^{-1}P_2$ are the identity braids
(by Lemma~\ref{L:perm}).
Therefore $\alpha_k=X_{A,L} (\Delta_{k-1}^2\splitBr 1) X_{L,A}=U(A)$.
\end{proof}

\begin{figure}\footnotesize
$$
\begin{array}{ccc}
\begin{xy}/r.33mm/:
(0,0)="or"="aa"="a1", (6,0)="hd", (55,0)+"hd"+"hd"="hdL",
"aa"+(0,10)="aa"="a2",
"aa"+(0,10)="aa"="a3",
"aa"+(0,10)="aa"="a4",
"aa"+(0,20)="aa"="a5",
"aa"+(0,10)="aa"="a6",
"aa"+(0,10)="aa"="a7",
"a1"+"hdL"="aa"="b1",
"aa"+(0,10)="aa"="b2",
"aa"+(0,10)="aa"="b3",
"aa"+(0,20)="aa"="b4",
"aa"+(0,10)="aa"="b5",
"aa"+(0,10)="aa"="b6",
"aa"+(0,10)="aa"="b7",
"a7"="aa"; "aa"+"hd"+(6,0)="aa" **@{-}; "aa"+(50,-50)="aa" **@{-}; "b3" **@{-},
"a6"="aa"; "aa"+"hd"+(3,0)="aa" **@{-}; "aa"+(50,-50)="aa" **@{-}; "b2" **@{-},
"a5"="aa"; "aa"+"hd"="aa" **@{-}; "aa"+(50,-50)="aa" **@{-}; "b1" **@{-},
"a4"="aa"; "aa"+"hd"="aa" **@{-};
    "aa"+(8,8)="aa" **@{-},
    "aa"+(3,3)="aa"; "aa"+(4,4)="aa" **@{-},
    "aa"+(3,3)="aa"; "aa"+(4,4)="aa" **@{-},
    "aa"+(3,3)="aa";
    "aa"+(15,15)="aa" **@{-};
    "b7" **@{-},
"a3"="aa"; "aa"+"hd"+(3,0)="aa" **@{-};
    "aa"+(12,12)="aa" **@{-},
    "aa"+(3,3)="aa"; "aa"+(4,4)="aa" **@{-},
    "aa"+(2,2)="aa"; "aa"+(4,4)="aa" **@{-},
    "aa"+(3,3)="aa";
    "aa"+(12,12)="aa" **@{-};
    "b6" **@{-},
"a2"="aa"; "aa"+"hd"+(6,0)="aa" **@{-};
    "aa"+(15,15)="aa" **@{-},
    "aa"+(3,3)="aa"; "aa"+(4,4)="aa" **@{-},
    "aa"+(3,3)="aa"; "aa"+(4,4)="aa" **@{-},
    "aa"+(3,3)="aa";
    "aa"+(8,8)="aa" **@{-};
    "b5" **@{-},
"a1"="aa"; "aa"+"hd"+(9,0)="aa" **@{-};
    "aa"+(19,19)="aa" **@{-},
    "aa"+(3,3)="aa"; "aa"+(4,4)="aa" **@{-},
    "aa"+(2,2)="aa"; "aa"+(4,4)="aa" **@{-},
    "aa"+(3,3)="aa";
    "aa"+(5,5)="aa" **@{-};
    "b4" **@{-},
(15,0)="hd", (0,14)="vd",
"a6"-"hd"="aa" *{Q_1^{-1}},
    "aa"-"hd"-"vd";
    "aa"+"hd"-"vd" **@{-};
    "aa"+"hd"+"vd" **@{-};
    "aa"-"hd"+"vd" **@{-};
    "aa"-"hd"-"vd" **@{-};
"b2"+"hd"="aa" *{\Delta_{k-1}^2},
    "aa"-"hd"-"vd";
    "aa"+"hd"-"vd" **@{-};
    "aa"+"hd"+"vd" **@{-};
    "aa"-"hd"+"vd" **@{-};
    "aa"-"hd"-"vd" **@{-};
"b2"+"hd"+"hd"+"hd"+"hd"="aa" *{P_1},
    "aa"-"hd"-"vd";
    "aa"+"hd"-"vd" **@{-};
    "aa"+"hd"+"vd" **@{-};
    "aa"-"hd"+"vd" **@{-};
    "aa"-"hd"-"vd" **@{-};
(0,19)="vd",
"a2"+(0,5)-"hd"="aa" *{Q_2^{-1}},
    "aa"-"hd"-"vd";
    "aa"+"hd"-"vd" **@{-};
    "aa"+"hd"+"vd" **@{-};
    "aa"-"hd"+"vd" **@{-};
    "aa"-"hd"-"vd" **@{-};
"b5"+(0,5)+"hd"+"hd"+"hd"+"hd"="aa" *{P_2},
    "aa"-"hd"-"vd";
    "aa"+"hd"-"vd" **@{-};
    "aa"+"hd"+"vd" **@{-};
    "aa"-"hd"+"vd" **@{-};
    "aa"-"hd"-"vd" **@{-};
(10,0)="bb",
"a1"-"hd"-"hd"="aa"; "aa"-"bb" **@{-},
"a2"-"hd"-"hd"="aa"; "aa"-"bb" **@{-},
"a3"-"hd"-"hd"="aa"; "aa"-"bb" **@{-},
"a4"-"hd"-"hd"="aa"; "aa"-"bb" **@{-},
"a5"-"hd"-"hd"="aa"; "aa"-"bb" **@{-},
"a6"-"hd"-"hd"="aa"; "aa"-"bb" **@{-},
"a7"-"hd"-"hd"="aa"; "aa"-"bb" **@{-},
"b1"+"hd"+"hd"="aa"; "aa"+"hd"="aa" **@{-}, "aa"+"hd"+"hd"="aa"; "aa"+"bb" **@{-},
"b2"+"hd"+"hd"="aa"; "aa"+"hd"="aa" **@{-}, "aa"+"hd"+"hd"="aa"; "aa"+"bb" **@{-},
"b3"+"hd"+"hd"="aa"; "aa"+"hd"="aa" **@{-}, "aa"+"hd"+"hd"="aa"; "aa"+"bb" **@{-},
"b4"="aa"; "aa"+"hd"+"hd"+"hd"="aa" **@{-}, "aa"+"hd"+"hd"="aa"; "aa"+"bb" **@{-},
"b5"="aa"; "aa"+"hd"+"hd"+"hd"="aa" **@{-}, "aa"+"hd"+"hd"="aa"; "aa"+"bb" **@{-},
"b6"="aa"; "aa"+"hd"+"hd"+"hd"="aa" **@{-}, "aa"+"hd"+"hd"="aa"; "aa"+"bb" **@{-},
"b7"="aa"; "aa"+"hd"+"hd"+"hd"="aa" **@{-}, "aa"+"hd"+"hd"="aa"; "aa"+"bb" **@{-},
\end{xy}
&\quad&
\begin{xy}/r.33mm/:
(0,0)="or"="aa"="a1", (6,0)="hd", (55,0)+"hd"+"hd"="hdL",
"aa"+(0,10)="aa"="a2",
"aa"+(0,10)="aa"="a3",
"aa"+(0,10)="aa"="a4",
"aa"+(0,20)="aa"="a5",
"aa"+(0,10)="aa"="a6",
"aa"+(0,10)="aa"="a7",
"a1"+"hdL"="aa"="b1",
"aa"+(0,10)="aa"="b2",
"aa"+(0,10)="aa"="b3",
"aa"+(0,20)="aa"="b4",
"aa"+(0,10)="aa"="b5",
"aa"+(0,10)="aa"="b6",
"aa"+(0,10)="aa"="b7",
"a7"="aa"; "aa"+"hd"+(6,0)="aa" **@{-}; "aa"+(50,-50)="aa" **@{-}; "b3" **@{-},
"a6"="aa"; "aa"+"hd"+(3,0)="aa" **@{-}; "aa"+(50,-50)="aa" **@{-}; "b2" **@{-},
"a5"="aa"; "aa"+"hd"="aa" **@{-}; "aa"+(50,-50)="aa" **@{-}; "b1" **@{-},
"a4"="aa"; "aa"+"hd"="aa" **@{-};
    "aa"+(8,8)="aa" **@{-},
    "aa"+(3,3)="aa"; "aa"+(4,4)="aa" **@{-},
    "aa"+(3,3)="aa"; "aa"+(4,4)="aa" **@{-},
    "aa"+(3,3)="aa";
    "aa"+(15,15)="aa" **@{-};
    "b7" **@{-},
"a3"="aa"; "aa"+"hd"+(3,0)="aa" **@{-};
    "aa"+(12,12)="aa" **@{-},
    "aa"+(3,3)="aa"; "aa"+(4,4)="aa" **@{-},
    "aa"+(2,2)="aa"; "aa"+(4,4)="aa" **@{-},
    "aa"+(3,3)="aa";
    "aa"+(12,12)="aa" **@{-};
    "b6" **@{-},
"a2"="aa"; "aa"+"hd"+(6,0)="aa" **@{-};
    "aa"+(15,15)="aa" **@{-},
    "aa"+(3,3)="aa"; "aa"+(4,4)="aa" **@{-},
    "aa"+(3,3)="aa"; "aa"+(4,4)="aa" **@{-},
    "aa"+(3,3)="aa";
    "aa"+(8,8)="aa" **@{-};
    "b5" **@{-},
"a1"="aa"; "aa"+"hd"+(9,0)="aa" **@{-};
    "aa"+(19,19)="aa" **@{-},
    "aa"+(3,3)="aa"; "aa"+(4,4)="aa" **@{-},
    "aa"+(2,2)="aa"; "aa"+(4,4)="aa" **@{-},
    "aa"+(3,3)="aa";
    "aa"+(5,5)="aa" **@{-};
    "b4" **@{-},
(15,0)="hd", (0,14)="vd",
"b2"+"hd"="aa" *{\Delta_{k-1}^2},
    "aa"-"hd"-"vd";
    "aa"+"hd"-"vd" **@{-};
    "aa"+"hd"+"vd" **@{-};
    "aa"-"hd"+"vd" **@{-};
    "aa"-"hd"-"vd" **@{-};
"b2"+"hd"+"hd"+"hd"+"hd"="aa" *{Q_1^{-1}P_1},
    "aa"-"hd"-"vd";
    "aa"+"hd"-"vd" **@{-};
    "aa"+"hd"+"vd" **@{-};
    "aa"-"hd"+"vd" **@{-};
    "aa"-"hd"-"vd" **@{-};
(0,19)="vd",
"b5"+(0,5)+"hd"+"hd"+"hd"+"hd"="aa" *{Q_2^{-1}P_2},
    "aa"-"hd"-"vd";
    "aa"+"hd"-"vd" **@{-};
    "aa"+"hd"+"vd" **@{-};
    "aa"-"hd"+"vd" **@{-};
    "aa"-"hd"-"vd" **@{-};
(10,0)="bb",
"b1"+"hd"+"hd"="aa"; "aa"+"hd"="aa" **@{-}, "aa"+"hd"+"hd"="aa"; "aa"+"bb" **@{-},
"b2"+"hd"+"hd"="aa"; "aa"+"hd"="aa" **@{-}, "aa"+"hd"+"hd"="aa"; "aa"+"bb" **@{-},
"b3"+"hd"+"hd"="aa"; "aa"+"hd"="aa" **@{-}, "aa"+"hd"+"hd"="aa"; "aa"+"bb" **@{-},
"b4"="aa"; "aa"+"hd"+"hd"+"hd"="aa" **@{-}, "aa"+"hd"+"hd"="aa"; "aa"+"bb" **@{-},
"b5"="aa"; "aa"+"hd"+"hd"+"hd"="aa" **@{-}, "aa"+"hd"+"hd"="aa"; "aa"+"bb" **@{-},
"b6"="aa"; "aa"+"hd"+"hd"+"hd"="aa" **@{-}, "aa"+"hd"+"hd"="aa"; "aa"+"bb" **@{-},
"b7"="aa"; "aa"+"hd"+"hd"+"hd"="aa" **@{-}, "aa"+"hd"+"hd"="aa"; "aa"+"bb" **@{-},
\end{xy}\\
\mbox{\normalsize (a) $(Q_2^{-1}\splitBr Q_1^{-1})X_{T,L}(\Delta_{k-1}^2\splitBr 1)(P_1\splitBr P_2)$}&&
\mbox{\normalsize (b) $X_{T,L}(\Delta_{k-1}^2\splitBr 1)(Q_1^{-1}P_1\splitBr Q_2^{-1}P_2)$}\rule{0pt}{2em}
\end{array}
$$
\caption{
The braid $(Q_2^{-1}\splitBr Q_1^{-1})X_{T,L}(\Delta_{k-1}^2\splitBr 1)(P_1\splitBr P_2)$
is the same as the braid $X_{T,L}(\Delta_{k-1}^2\splitBr 1)(Q_1^{-1}P_1\splitBr Q_2^{-1}P_2)$.}
\label{F:al}
\end{figure}
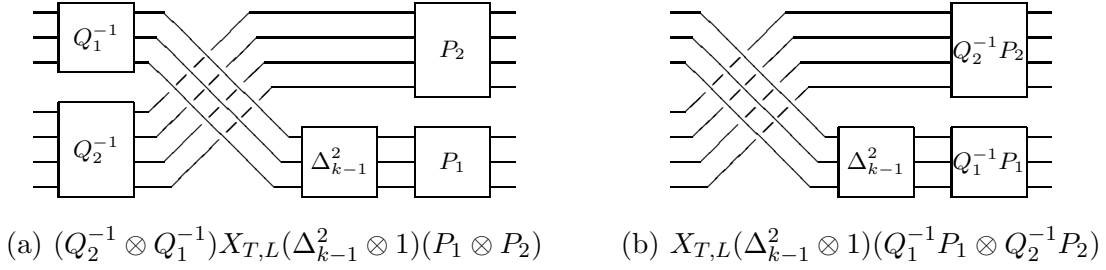

\begin{example}
Let $n=7$ and $k=3$.
Then the $n$-permutation  $\pi_k$ in Theorem~\ref{T:Br}
is $\pi_3=(3,6,2,5,1,4,7)$ because
$\pi_3(i)\equiv 3i\bmod 7$ for $1\le i\le 7$.
Since $a_1=\lceil\frac{7}{3}\rceil=3$ and
$a_2=\lceil\frac{7\cdot 2}{3}\rceil=5$, $A=\{a_1,a_2\}=\{3,5\}$,
hence $\pi_3(A)=\{1,2\}=L$.
See Figure~\ref{F:ex7} for the braids $X_3=P(\pi_3)$ and
$\tau(X_3)=\delta^{-1}X_3\delta$.
\end{example}

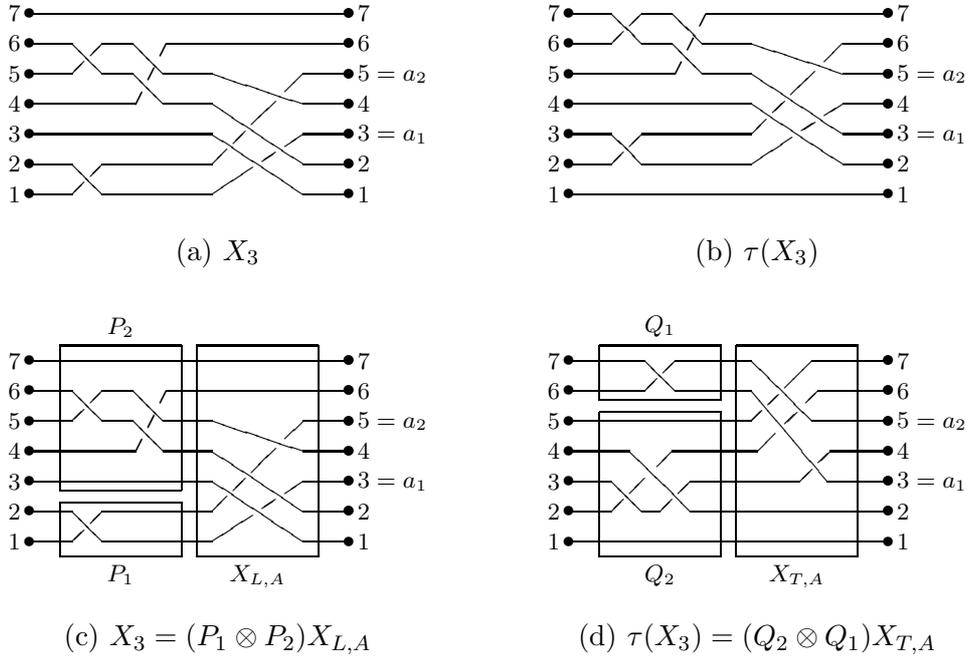
\begin{figure}\footnotesize
$$
\begin{array}{ccc}
\begin{xy}/r.4mm/:
(0,0)="or", (50,0)="hd", (30,0)="hdb", (5,0)="hda", "hda"+"hda"="hda2", (0,10)="vd",
     "or"="a1"="aa" *{\bullet} *+!R{1}, "aa"+"hd"+"hda2"="b1"="bb", "bb"+"hdb"="c1"="bb"; "bb"+"hda"+(10,0) **@{-} *{\bullet} *+!L{1},
"aa"+"vd"="a2"="aa" *{\bullet} *+!R{2}, "aa"+"hd"+"hda2"="b2"="bb", "bb"+"hdb"="c2"="bb"; "bb"+"hda"+(10,0) **@{-} *{\bullet} *{\bullet} *+!L{2},
"aa"+"vd"="a3"="aa" *{\bullet} *+!R{3}, "aa"+"hd"+"hda2"="b3"="bb", "bb"+"hdb"="c3"="bb"; "bb"+"hda"+(10,0) **@{-} *{\bullet} *{\bullet} *+!L{3=a_1},
"aa"+"vd"="a4"="aa" *{\bullet} *+!R{4}, "aa"+"hd"+"hda2"="b4"="bb", "bb"+"hdb"="c4"="bb"; "bb"+"hda"+(10,0) **@{-} *{\bullet} *{\bullet} *+!L{4},
"aa"+"vd"="a5"="aa" *{\bullet} *+!R{5}, "aa"+"hd"+"hda2"="b5"="bb", "bb"+"hdb"="c5"="bb"; "bb"+"hda"+(10,0) **@{-} *{\bullet} *{\bullet} *+!L{5=a_2},
"aa"+"vd"="a6"="aa" *{\bullet} *+!R{6}, "aa"+"hd"+"hda2"="b6"="bb", "bb"+"hdb"="c6"="bb"; "bb"+"hda"+(10,0) **@{-} *{\bullet} *{\bullet} *+!L{6},
"aa"+"vd"="a7"="aa" *{\bullet} *+!R{7}, "aa"+"hd"+"hda2"="b7"="bb", "bb"+"hdb"="c7"="bb"; "bb"+"hda"+(10,0) **@{-} *{\bullet} *{\bullet} *+!L{7},
"a1"="aa"; "aa"+"hda"+(9,0)="aa" **@{-};
    "aa"+(4,4)="aa" **@{-},
    "aa"+(2,2)="aa";
    "aa"+(4,4)="aa" **@{-};
    "b2" **@{-},
"a2"="aa"; "aa"+"hda"+(9,0)="aa" **@{-};
    "aa"+(10,-10)="aa" **@{-}; "b1" **@{-},
"a3"="aa"; "b3" **@{-},
"a4"="aa"; "aa"+"hda"+(30,0)="aa" **@{-};
    "aa"+(2,4)="aa" **@{-},
    "aa"+(2,4)="aa";
    "aa"+(2,4)="aa" **@{-};
    "aa"+(2,4)="aa";
    "aa"+(2,4)="aa" **@{-};
    "b6" **@{-},
"a5"="aa"; "aa"+"hda"+(9,0)="aa" **@{-};
    "aa"+(4,4)="aa" **@{-},
    "aa"+(2,2)="aa";
    "aa"+(4,4)="aa" **@{-};
    "aa"+(10,0)="aa" **@{-};
    "aa"+(10,-10)="aa" **@{-};
    "b5" **@{-},
"a6"="aa"; "aa"+"hda"+(9,0)="aa" **@{-};
    "aa"+(10,-10)="aa" **@{-};
    "aa"+(10,0)="aa" **@{-};
    "aa"+(10,-10)="aa" **@{-};
    "b4" **@{-},
"a7"="aa"; "b7" **@{-},
"b7"; "c7" **@{-},
"b6"; "c6" **@{-},
"b5"; "c4" **@{-},
"b4"; "c2" **@{-},
"b3"; "c1" **@{-},
"b2"="aa"; "aa"+(5,5)="aa" **@{-},
    "aa"+(2,2)="aa"; "aa"+(4,4)="aa" **@{-},
    "aa"+(2,2)="aa"; "aa"+(8,8)="aa" **@{-},
    "aa"+(3,3)="aa"; "c5" **@{-},
"b1"="aa"; "aa"+(13.5,9)="aa" **@{-},
    "aa"+(3,2)="aa"; "aa"+(4.5,3)="aa" **@{-},
    "aa"+(3,2)="aa"; "c3" **@{-},
\end{xy}
&\qquad&
\begin{xy}/r.4mm/:
(0,0)="or", (50,0)="hd", (30,0)="hdb", (5,0)="hda", "hda"+"hda"="hda2", (0,10)="vd",
     "or"="a1"="aa" *{\bullet} *+!R{1}, "aa"+"hd"+"hda2"="b1"="bb", "bb"+"hdb"="c1"="bb"; "bb"+"hda"+(10,0) **@{-} *{\bullet} *+!L{1},
"aa"+"vd"="a2"="aa" *{\bullet} *+!R{2}, "aa"+"hd"+"hda2"="b2"="bb", "bb"+"hdb"="c2"="bb"; "bb"+"hda"+(10,0) **@{-} *{\bullet} *{\bullet} *+!L{2},
"aa"+"vd"="a3"="aa" *{\bullet} *+!R{3}, "aa"+"hd"+"hda2"="b3"="bb", "bb"+"hdb"="c3"="bb"; "bb"+"hda"+(10,0) **@{-} *{\bullet} *{\bullet} *+!L{3=a_1},
"aa"+"vd"="a4"="aa" *{\bullet} *+!R{4}, "aa"+"hd"+"hda2"="b4"="bb", "bb"+"hdb"="c4"="bb"; "bb"+"hda"+(10,0) **@{-} *{\bullet} *{\bullet} *+!L{4},
"aa"+"vd"="a5"="aa" *{\bullet} *+!R{5}, "aa"+"hd"+"hda2"="b5"="bb", "bb"+"hdb"="c5"="bb"; "bb"+"hda"+(10,0) **@{-} *{\bullet} *{\bullet} *+!L{5=a_2},
"aa"+"vd"="a6"="aa" *{\bullet} *+!R{6}, "aa"+"hd"+"hda2"="b6"="bb", "bb"+"hdb"="c6"="bb"; "bb"+"hda"+(10,0) **@{-} *{\bullet} *{\bullet} *+!L{6},
"aa"+"vd"="a7"="aa" *{\bullet} *+!R{7}, "aa"+"hd"+"hda2"="b7"="bb", "bb"+"hdb"="c7"="bb"; "bb"+"hda"+(10,0) **@{-} *{\bullet} *{\bullet} *+!L{7},
"a2"="aa"; "aa"+"hda"+(9,0)="aa" **@{-};
    "aa"+(4,4)="aa" **@{-},
    "aa"+(2,2)="aa";
    "aa"+(4,4)="aa" **@{-};
    "b3" **@{-},
"a3"="aa"; "aa"+"hda"+(9,0)="aa" **@{-};
    "aa"+(10,-10)="aa" **@{-};
    "b2" **@{-},
"a4"="aa"; "b4" **@{-},
"a5"="aa"; "aa"+"hda"+(30,0)="aa" **@{-};
    "aa"+(2,4)="aa" **@{-},
    "aa"+(2,4)="aa";
    "aa"+(2,4)="aa" **@{-};
    "aa"+(2,4)="aa";
    "aa"+(2,4)="aa" **@{-};
    "b7" **@{-},
"a6"="aa"; "aa"+"hda"+(9,0)="aa" **@{-};
    "aa"+(4,4)="aa" **@{-},
    "aa"+(2,2)="aa";
    "aa"+(4,4)="aa" **@{-};
    "aa"+(10,0)="aa" **@{-};
    "aa"+(10,-10)="aa" **@{-};
    "b6" **@{-},
"a7"="aa"; "aa"+"hda"+(9,0)="aa" **@{-};
    "aa"+(10,-10)="aa" **@{-};
    "aa"+(10,0)="aa" **@{-};
    "aa"+(10,-10)="aa" **@{-};
    "b5" **@{-},
"a1"="aa"; "b1" **@{-},
"b7"; "c7" **@{-},
"b6"; "c5" **@{-},
"b5"; "c3" **@{-},
"b4"; "c2" **@{-},
"b1"; "c1" **@{-},
"b3"="aa"; "aa"+(5,5)="aa" **@{-},
    "aa"+(2,2)="aa"; "aa"+(4,4)="aa" **@{-},
    "aa"+(2,2)="aa"; "aa"+(8,8)="aa" **@{-},
    "aa"+(3,3)="aa"; "c6" **@{-},
"b2"="aa"; "aa"+(13.5,9)="aa" **@{-},
    "aa"+(3,2)="aa"; "aa"+(4.5,3)="aa" **@{-},
    "aa"+(3,2)="aa"; "c4" **@{-},
\end{xy}\\
\mbox{\normalsize (a) $X_3$}&&
\mbox{\normalsize (b) $\tau(X_3)$}\rule{0pt}{2em}
\\[1.5em]
\begin{xy}/r.4mm/:
(0,0)="or", (50,0)="hd", (30,0)="hdb", (5,0)="hda", "hda"+"hda"="hda2", (0,10)="vd",
     "or"="a1"="aa" *{\bullet} *+!R{1}, "aa"+"hd"+"hda2"="b1"="bb", "bb"+"hdb"="c1"="bb"; "bb"+"hda"+(10,0) **@{-} *{\bullet} *+!L{1},
"aa"+"vd"="a2"="aa" *{\bullet} *+!R{2}, "aa"+"hd"+"hda2"="b2"="bb", "bb"+"hdb"="c2"="bb"; "bb"+"hda"+(10,0) **@{-} *{\bullet} *{\bullet} *+!L{2},
"aa"+"vd"="a3"="aa" *{\bullet} *+!R{3}, "aa"+"hd"+"hda2"="b3"="bb", "bb"+"hdb"="c3"="bb"; "bb"+"hda"+(10,0) **@{-} *{\bullet} *{\bullet} *+!L{3=a_1},
"aa"+"vd"="a4"="aa" *{\bullet} *+!R{4}, "aa"+"hd"+"hda2"="b4"="bb", "bb"+"hdb"="c4"="bb"; "bb"+"hda"+(10,0) **@{-} *{\bullet} *{\bullet} *+!L{4},
"aa"+"vd"="a5"="aa" *{\bullet} *+!R{5}, "aa"+"hd"+"hda2"="b5"="bb", "bb"+"hdb"="c5"="bb"; "bb"+"hda"+(10,0) **@{-} *{\bullet} *{\bullet} *+!L{5=a_2},
"aa"+"vd"="a6"="aa" *{\bullet} *+!R{6}, "aa"+"hd"+"hda2"="b6"="bb", "bb"+"hdb"="c6"="bb"; "bb"+"hda"+(10,0) **@{-} *{\bullet} *{\bullet} *+!L{6},
"aa"+"vd"="a7"="aa" *{\bullet} *+!R{7}, "aa"+"hd"+"hda2"="b7"="bb", "bb"+"hdb"="c7"="bb"; "bb"+"hda"+(10,0) **@{-} *{\bullet} *{\bullet} *+!L{7},
"a1"="aa"; "aa"+"hda"+(9,0)="aa" **@{-};
    "aa"+(4,4)="aa" **@{-},
    "aa"+(2,2)="aa";
    "aa"+(4,4)="aa" **@{-};
    "b2" **@{-},
"a2"="aa"; "aa"+"hda"+(9,0)="aa" **@{-};
    "aa"+(10,-10)="aa" **@{-}; "b1" **@{-},
"a3"="aa"; "b3" **@{-},
"a4"="aa"; "aa"+"hda"+(30,0)="aa" **@{-};
    "aa"+(2,4)="aa" **@{-},
    "aa"+(2,4)="aa";
    "aa"+(2,4)="aa" **@{-};
    "aa"+(2,4)="aa";
    "aa"+(2,4)="aa" **@{-};
    "b6" **@{-},
"a5"="aa"; "aa"+"hda"+(9,0)="aa" **@{-};
    "aa"+(4,4)="aa" **@{-},
    "aa"+(2,2)="aa";
    "aa"+(4,4)="aa" **@{-};
    "aa"+(10,0)="aa" **@{-};
    "aa"+(10,-10)="aa" **@{-};
    "b5" **@{-},
"a6"="aa"; "aa"+"hda"+(9,0)="aa" **@{-};
    "aa"+(10,-10)="aa" **@{-};
    "aa"+(10,0)="aa" **@{-};
    "aa"+(10,-10)="aa" **@{-};
    "b4" **@{-},
"a7"="aa"; "b7" **@{-},
"b7"; "c7" **@{-},
"b6"; "c6" **@{-},
"b5"; "c4" **@{-},
"b4"; "c2" **@{-},
"b3"; "c1" **@{-},
"b2"="aa"; "aa"+(5,5)="aa" **@{-},
    "aa"+(2,2)="aa"; "aa"+(4,4)="aa" **@{-},
    "aa"+(2,2)="aa"; "aa"+(8,8)="aa" **@{-},
    "aa"+(3,3)="aa"; "c5" **@{-},
"b1"="aa"; "aa"+(13.5,9)="aa" **@{-},
    "aa"+(3,2)="aa"; "aa"+(4.5,3)="aa" **@{-},
    "aa"+(3,2)="aa"; "c3" **@{-},
"a1"+"hda"+(25,5)="B2",
    "B2"+(-20,-10);
    "B2"+( 20,-10) **@{-};
    "B2"+( 20,  8) **@{-};
    "B2"+(-20,  8) **@{-};
    "B2"+(-20,-10) **@{-};
    "B2"+(0,-10) *+!U{P_1},
"a5"+"hda"+(25,0)="B2",
    "B2"+(-20,-23);
    "B2"+( 20,-23) **@{-};
    "B2"+( 20, 25) **@{-};
    "B2"+(-20, 25) **@{-};
    "B2"+(-20,-23) **@{-};
    "B2"+(0,25) *+!D{P_2},
"b4"+(15,0)="B2",
    "B2"+(-20,-35);
    "B2"+( 20,-35) **@{-};
    "B2"+( 20, 35) **@{-};
    "B2"+(-20, 35) **@{-};
    "B2"+(-20,-35) **@{-};
    "B2"+(0,-35) *+!U{X_{L,A}},
\end{xy}
&\qquad&
\begin{xy}/r.4mm/:
(0,0)="or", (50,0)="hd", (30,0)="hdb", (5,0)="hda", "hda"+"hda"="hda2", (0,10)="vd",
     "or"="a1"="aa" *{\bullet} *+!R{1}, "aa"+"hd"+"hda2"="b1"="bb", "bb"+"hdb"="c1"="bb"; "bb"+"hda"+(10,0) **@{-} *{\bullet} *+!L{1},
"aa"+"vd"="a2"="aa" *{\bullet} *+!R{2}, "aa"+"hd"+"hda2"="b2"="bb", "bb"+"hdb"="c2"="bb"; "bb"+"hda"+(10,0) **@{-} *{\bullet} *{\bullet} *+!L{2},
"aa"+"vd"="a3"="aa" *{\bullet} *+!R{3}, "aa"+"hd"+"hda2"="b3"="bb", "bb"+"hdb"="c3"="bb"; "bb"+"hda"+(10,0) **@{-} *{\bullet} *{\bullet} *+!L{3=a_1},
"aa"+"vd"="a4"="aa" *{\bullet} *+!R{4}, "aa"+"hd"+"hda2"="b4"="bb", "bb"+"hdb"="c4"="bb"; "bb"+"hda"+(10,0) **@{-} *{\bullet} *{\bullet} *+!L{4},
"aa"+"vd"="a5"="aa" *{\bullet} *+!R{5}, "aa"+"hd"+"hda2"="b5"="bb", "bb"+"hdb"="c5"="bb"; "bb"+"hda"+(10,0) **@{-} *{\bullet} *{\bullet} *+!L{5=a_2},
"aa"+"vd"="a6"="aa" *{\bullet} *+!R{6}, "aa"+"hd"+"hda2"="b6"="bb", "bb"+"hdb"="c6"="bb"; "bb"+"hda"+(10,0) **@{-} *{\bullet} *{\bullet} *+!L{6},
"aa"+"vd"="a7"="aa" *{\bullet} *+!R{7}, "aa"+"hd"+"hda2"="b7"="bb", "bb"+"hdb"="c7"="bb"; "bb"+"hda"+(10,0) **@{-} *{\bullet} *{\bullet} *+!L{7},
"a1"="aa"; "b1" **@{-},
"a2"="aa"; "aa"+"hda"+(9,0)="aa" **@{-};
    "aa"+(4,4)="aa" **@{-},
    "aa"+(2,2)="aa";
    "aa"+(5,5)="aa" **@{-};
    "aa"+(4,4)="aa";
    "aa"+(5,5)="aa" **@{-};
    "b4" **@{-},
"a3"="aa"; "aa"+"hda"+(9,0)="aa" **@{-};
    "aa"+(10,-10)="aa" **@{-};
    "aa"+(6, 0)="aa" **@{-};
    "aa"+(4,4)="aa" **@{-},
    "aa"+(2,2)="aa";
    "aa"+(4,4)="aa" **@{-};
    "b3" **@{-},
"a4"="aa"; "aa"+"hda"+(15,0)="aa" **@{-};
    "aa"+(20,-20)="aa" **@{-};
    "b2" **@{-},
"a5"="aa"; "b5" **@{-},
"a6"="aa"; "aa"+"hda"+(20,0)="aa" **@{-};
    "aa"+(4,4)="aa" **@{-},
    "aa"+(2,2)="aa";
    "aa"+(4,4)="aa" **@{-};
    "b7" **@{-},
"a7"="aa"; "aa"+"hda"+(20,0)="aa" **@{-};
    "aa"+(10,-10)="aa" **@{-};
    "b6" **@{-},
"b7"; "c5"+(-10,0) **@{-}; "c5" **@{-},
"b6"; "c3"+(-5,0) **@{-}; "c3" **@{-},
"b2"; "c2" **@{-},
"b1"; "c1" **@{-},
"b5"="aa"; "aa"+(3.5,3.5)="aa" **@{-},
    "aa"+(2,2)="aa"; "aa"+(3.5,3.5)="aa" **@{-},
    "aa"+(2,2)="aa"; "aa"+(9,9)="aa" **@{-};
    "c7" **@{-},
"b4"="aa"; "aa"+(2,0)="aa" **@{-}; "aa"+(6,6)="aa" **@{-},
    "aa"+(3,3)="aa"; "aa"+(4,4)="aa" **@{-},
    "aa"+(3,3)="aa"; "aa"+(4,4)="aa" **@{-};
    "c6" **@{-},
"b3"="aa"; "aa"+(16,0)="aa" **@{-}; "aa"+(4,4)="aa" **@{-},
    "aa"+(2,2)="aa"; "aa"+(4,4)="aa" **@{-};
    "c4" **@{-},
"a6"+"hda"+(25,5)="B2",
    "B2"+(-20,10);
    "B2"+( 20,10) **@{-};
    "B2"+( 20,-8) **@{-};
    "B2"+(-20,-8) **@{-};
    "B2"+(-20,10) **@{-};
    "B2"+(0,10) *+!D{Q_1},
"a3"+"hda"+(25,0)="B2",
    "B2"+(-20,-25);
    "B2"+( 20,-25) **@{-};
    "B2"+( 20, 23) **@{-};
    "B2"+(-20, 23) **@{-};
    "B2"+(-20,-25) **@{-};
    "B2"+(0,-25) *+!U{Q_2},
"b4"+(15,0)="B2",
    "B2"+(-20,-35);
    "B2"+( 20,-35) **@{-};
    "B2"+( 20, 35) **@{-};
    "B2"+(-20, 35) **@{-};
    "B2"+(-20,-35) **@{-};
    "B2"+(0,-35) *+!U{X_{T,A}},
\end{xy}\\
\mbox{\normalsize (c) $X_3=(P_1\splitBr P_2)X_{L,A}$}&&
\mbox{\normalsize (d) $\tau(X_3)=(Q_2\splitBr Q_1)X_{T,A}$}\rule{0pt}{2em}
\end{array}
$$
\caption{The braids $X_3$ and $\tau(X_3)$}
\label{F:ex7}
\end{figure}

\begin{example}
By Theorem~\ref{T:Br},
the $n$-braid $\delta^k$ is conjugate to $\delta U_{a_1}\cdots U_{a_{k-1}}$,
where $a_i=\lceil \frac{ni}{k}\rceil$.
See the following for $n=7$ and $2\le k\le 6$,
where $\sim$ denotes conjugacy.
$$
\delta^2\sim \delta U_4,\quad
\delta^3\sim \delta U_3U_5,\quad
\delta^4\sim \delta U_2U_4U_6,\quad
\delta^5\sim \delta U_2U_3U_5U_6,\quad
\delta^6\sim \delta U_2U_3U_4U_5U_6
$$
\end{example}

\begin{corollary}\label{C:ns}
Let $n$ and $s$ be relatively prime integers with $2\le n<s$,
and let $s=nm+k$ with $m\ge 1$ and $1\le k\le n-1$.
Then the $n$-braid $\delta^s$ is conjugate to
$$\delta (U_2\cdots U_n)^m U_{a_1}\cdots U_{a_{k-1}},$$
where $a_i=\lceil \frac{ni}k\rceil$
for $1\le i\le k-1$.
\end{corollary}

\begin{proof}
If $k=1$, there is nothing to prove because $\delta^s=\delta\cdot(\delta^n)^m
=\delta (U_2\cdots U_n)^m$.

Let $k\ge 2$.
Since $n$ and $k$ are relatively prime,
$\delta^k$ is conjugate to $\delta U_{a_1}\cdots U_{a_{k-1}}$ by Theorem~\ref{T:Br}.
Since $\delta^n=\Delta^2$ is central and $\Delta^2=U_2\cdots U_n$,
$\delta^s=\delta^{nm}\delta^k$ is conjugate to
$$\delta^{nm}\cdot \delta U_{a_1}\cdots U_{a_{k-1}}
=\delta\cdot \delta^{nm}\cdot U_{a_1}\cdots U_{a_{k-1}}
=\delta (U_2\cdots U_n)^m U_{a_1}\cdots U_{a_{k-1}}.$$
\vskip-\baselineskip
\end{proof}

\section{Petal numbers of torus knots}
Recall that $\PG(\pi)$ denotes the petal grid diagram obtained from
the permutation $\pi$
and that $K(\pi)$ denotes the knot represented by $\PG(\pi)$.

\begin{definition}
Let $\pi
=(a_1,a_3,\cdots,a_{2n+1})\odot(a_2,a_4,\ldots,a_{2n})
=(a_1,a_2,\ldots,a_{2n+1})
$ be a petal permutation.
\begin{itemize}
\item[(i)]
$\pi$ is called \emph{braided} if
$a_{1}=n+1$; $a_{2i+1}\le n$  and $a_{2i}\ge n+2$ for $i=1,\ldots,n$.
\item[(ii)]
$\pi$ is called \emph{strongly braided} if $a_{2i+1}=n+1-i$ for $i=0,1,\ldots,n$, i.e.\
$$\pi=(n+1, n,\ldots,1)\odot(a_2,a_4,\ldots,a_{2n}).$$
\end{itemize}
\end{definition}

\begin{example}
Figure~\ref{F:SB} shows the petal grid diagrams
determined by the following permutations.
\begin{align*}
\pi_1&=(5,4,3,2,1)\odot(9,8,7,6)=(5,9,4,8,3,7,2,6,1)\\
\pi_2&=(5,4,3,2,1)\odot(7,6,8,9)=(5,7,4,6,3,8,2,9,1)\\
\pi_3&=(5,3,4,1,2)\odot(8,7,9,6)=(5,8,3,7,4,9,1,6,2)\\
\pi_4&=(4,3,2,1,9)\odot(8,7,6,5)=(4,8,3,7,2,6,1,5,9)\\
\pi_5&=(9,8,7,6,5)\odot (4,3,2,1)=(9,4,8,3,7,2,6,1,5)
\end{align*}
$\pi_1$ and $\pi_2$ are strongly braided.
$\pi_3$ is braided, but not strongly braided.
$\pi_4$ and $\pi_5$ are not braided.
Notice that the petal grid diagrams $\PG(\pi_1)$, $\PG(\pi_2)$ and $\PG(\pi_3)$
are closed braid diagrams.
(Suppose that the nodes of the grid diagrams
belong to $\{1,\ldots,2n+1\}\times\{1,\ldots,2n+1\}$.
Then all the strands
in $\PG(\pi_1)$, $\PG(\pi_2)$ and $\PG(\pi_3)$
move clockwise around the point $C=(n+\frac32,n+\frac32)$.)
Also, notice that even though $\pi_5$ is not a braided petal permutation in our definition,
the petal grid diagram $\PG(\pi_5)$ looks like a closed braid diagram
where the strands move counterclockwise around the point $C$.
\end{example}

\begin{figure}\footnotesize
$$
\begin{xy}/r.4mm/:
(0,0)="or", (10,0)="hd", (0,-10)="vd", (10,-10)="dd",
"or"="A1"="aa", 
    "aa"+"dd"="A2"="aa",
    "aa"+"dd"="A3"="aa",
    "aa"+"dd"="A4"="aa",
"A1"+(60,0)="B1"="aa", 
    "aa"+"dd"="B2"="aa",
    "aa"+"dd"="B3"="aa",
    "aa"+"dd"="B4"="aa",
"B1"+(0,-60)="C1"="aa", 
    "aa"+"dd"="C2"="aa",
    "aa"+"dd"="C3"="aa",
    "aa"+"dd"="C4"="aa",
"A1"+(0,-50)="D1"="aa", 
    "aa"+"dd"="D2"="aa",
    "aa"+"dd"="D3"="aa",
    "aa"+"dd"="D4"="aa",
    "aa"+"dd"="D5"="aa",
    "D1"+(40,0)="D0",
"B1"; "C1" **@{-} ?(.66) *@{>},
"B2"; "C2" **@{-} ?(.5) *@{>},
"B3"; "C3" **@{-} ?(.34) *@{>},
"B4"; "C4" **@{-} ?(.18) *@{>},
"D1"; "A1" **@{-} ?(.3) *@{>},
"D2"; "A2" **@{-} ?(.5) *@{>},
"D3"; "A3" **@{-} ?(.7) *@{>},
"D4"; "A4" **@{-} ?(.9) *@{>},
"D5"; "D0" **@{-} ?(.65) *@{>},
"B1"; "A1" **@{-},
"B2"="aa"; "aa"-(8,0)="aa" **@{-},
    "aa"-(4,0); "A2" **@{-},
"B3"="aa"; "aa"-(8,0)="aa" **@{-},
    "aa"-(4,0)="aa"; "aa"-(6,0)="aa" **@{-},
    "aa"-(4,0); "A3" **@{-},
"B4"="aa"; "aa"-(8,0)="aa" **@{-},
    "aa"-(4,0)="aa"; "aa"-(6,0)="aa" **@{-},
    "aa"-(4,0)="aa"; "aa"-(6,0)="aa" **@{-},
    "aa"-(4,0); "A4" **@{-},
"D1"="aa"; "aa"+(8,0)="aa" **@{-},
    "aa"+(4,0)="aa"; "aa"+(6,0)="aa" **@{-},
    "aa"+(4,0)="aa"; "aa"+(6,0)="aa" **@{-},
    "aa"+(4,0)="aa"; "D0" **@{-},
"D2"="aa"; "aa"+(8,0)="aa" **@{-},
    "aa"+(4,0)="aa"; "aa"+(6,0)="aa" **@{-},
    "aa"+(4,0)="aa"; "aa"+(6,0)="aa" **@{-},
    "aa"+(4,0)="aa"; "C1" **@{-},
"D3"="aa"; "aa"+(8,0)="aa" **@{-},
    "aa"+(4,0)="aa"; "aa"+(6,0)="aa" **@{-},
    "aa"+(4,0)="aa"; "C2" **@{-},
"D4"="aa"; "aa"+(8,0)="aa" **@{-},
    "aa"+(4,0)="aa"; "C3" **@{-},
"D5"; "C4" **@{-},
"B1"+(-10,1)="aa", "aa" *!D{9}, 
"aa"+(0,-10)="aa" *!D{8},
"aa"+(0,-10)="aa" *!D{7},
"aa"+(0,-10)="aa" *!D{6},
"aa"+(-14,-20) *!D{5},
"aa"+(0,-30)="aa" *!D{4},
"aa"+(0,-10)="aa" *!D{3},
"aa"+(0,-10)="aa" *!D{2},
"aa"+(0,-10)="aa" *!D{1},
"D1"+(-.5,5)="aa", "aa" *!R{1}, 
"aa"+(10,0)="aa" *!R{2},
"aa"+(10,0)="aa" *!R{3},
"aa"+(10,0)="aa" *!R{4},
"aa"+(12,-10) *!L{5},
"aa"+(30,0)="aa" *!R{6},
"aa"+(10,0)="aa" *!R{7},
"aa"+(10,0)="aa" *!R{8},
"aa"+(10,0)="aa" *!R{9},
"D5"+(10,-20) *{\normalsize
    \begin{array}[t]{l}
    \mbox{(a)\ \ $\PG(\pi_1)$ with}\\
    \pi_1=(5,4,3,2,1)\odot(9,8,7,6)
    \end{array}},
\end{xy}
\quad
\begin{xy}/r.4mm/:
(0,0)="or", (10,0)="hd", (0,-10)="vd", (10,-10)="dd",
"or"+(0,-20)="A1"="aa", 
    "aa"+"dd"="A2"="aa",
    "aa"+(10,20)="A3"="aa",
    "aa"+(10,10)="A4"="aa",
"A1"+(60,0)="B1"="aa", 
    "aa"+"dd"="B2"="aa",
    "aa"+(10,20)="B3"="aa",
    "aa"+(10,10)="B4"="aa",
"B1"+(0,-40)="C1"="aa", 
    "aa"+"dd"="C2"="aa",
    "aa"+"dd"="C3"="aa",
    "aa"+"dd"="C4"="aa",
"A1"+(0,-30)="D1"="aa", 
    "aa"+"dd"="D2"="aa",
    "aa"+"dd"="D3"="aa",
    "aa"+"dd"="D4"="aa",
    "aa"+"dd"="D5"="aa",
    "D1"+(40,0)="D0",
"B1"; "C1" **@{-} ?(.5) *@{>},
"B2"; "C2" **@{-} ?(.25) *@{>},
"B3"; "C3" **@{-} ?(.43) *@{>},
"B4"; "C4" **@{-} ?(.44) *@{>},
"D1"; "A1" **@{-} ?(.45) *@{>},
"D2"; "A2" **@{-} ?(.8) *@{>},
"D3"; "A3" **@{-} ?(.55) *@{>},
"D4"; "A4" **@{-} ?(.55) *@{>},
"D5"; "D0" **@{-} ?(.65) *@{>},
"A1"="aa"; "aa"+(18,0)="aa" **@{-},
    "aa"+(4,0)="aa"; "aa"+(6,0)="aa" **@{-},
    "aa"+(4,0)="aa"; "B1" **@{-},
"A2"="aa"; "aa"+(8,0)="aa" **@{-},
    "aa"+(4,0)="aa"; "aa"+(6,0)="aa" **@{-},
    "aa"+(4,0)="aa"; "B2"-(12,0)="aa" **@{-},
    "aa"+(4,0)="aa"; "B2" **@{-},
"A3"="aa"; "aa"+(8,0)="aa" **@{-},
    "aa"+(4,0)="aa"; "B3" **@{-},
"A4"="aa"; "B4" **@{-},
"D1"="aa"; "aa"+(8,0)="aa" **@{-},
    "aa"+(4,0)="aa"; "aa"+(6,0)="aa" **@{-},
    "aa"+(4,0)="aa"; "aa"+(6,0)="aa" **@{-},
    "aa"+(4,0)="aa"; "D0" **@{-},
"D2"="aa"; "aa"+(8,0)="aa" **@{-},
    "aa"+(4,0)="aa"; "aa"+(6,0)="aa" **@{-},
    "aa"+(4,0)="aa"; "aa"+(6,0)="aa" **@{-},
    "aa"+(4,0)="aa"; "C1" **@{-},
"D3"="aa"; "aa"+(8,0)="aa" **@{-},
    "aa"+(4,0)="aa"; "aa"+(6,0)="aa" **@{-},
    "aa"+(4,0)="aa"; "C2" **@{-},
"D4"="aa"; "aa"+(8,0)="aa" **@{-},
    "aa"+(4,0)="aa"; "C3" **@{-},
"D5"; "C4" **@{-},
"B1"+(-10,21)="aa", "aa" *!D{9}, 
"aa"+(0,-10)="aa" *!D{8},
"aa"+(0,-10)="aa" *!D{7},
"aa"+(0,-10)="aa" *!D{6},
"aa"+(-15,-20) *!D{5},
"aa"+(0,-30)="aa" *!D{4},
"aa"+(0,-10)="aa" *!D{3},
"aa"+(0,-10)="aa" *!D{2},
"aa"+(0,-10)="aa" *!D{1},
"D1"+(-.5,5)="aa", "aa" *!R{1}, 
"aa"+(10,0)="aa" *!R{2},
"aa"+(10,0)="aa" *!R{3},
"aa"+(10,0)="aa" *!R{4},
"aa"+(12,-10) *!L{5},
"aa"+(30,0)="aa" *!R{6},
"aa"+(10,0)="aa" *!R{7},
"aa"+(10,0)="aa" *!R{8},
"aa"+(10,0)="aa" *!R{9},
"D5"+(10,-20) *{\normalsize
    \begin{array}{l}
    \mbox{(b)\ \ $\PG(\pi_2)$ with}\\
    \pi_2=(5,4,3,2,1)\odot(7,6,8,9)
    \end{array}},
\end{xy}
\quad
\begin{xy}/r.4mm/:
(0,0)="or", (10,0)="hd", (0,-10)="vd", (10,-10)="dd",
"or"+(0,-10)="A1"="aa", 
    "A1"+"dd"="A2"="aa",
    "A1"+(20, 10)="A3",
    "A2"+(20,-10)="A4",
"A1"+(60,0)="B1", 
    "B1"+"dd"="B2",
    "B2"+(10,20)="B3",
    "B2"+(20,-10)="B4",
"B1"+(0,-60)="C1", 
    "C1"+(10,10)="C2",
    "C1"+(20,-20)="C3",
    "C3"+(10,10)="C4",
"A1"+(0,-40)="D1", 
    "D1"+(10,-20)="D2",
    "D2"+(10,10)="D3",
    "D3"+(10,-30)="D4",
    "D4"+(10,10)="D5",
    "D1"+(40,0)="D0",
"B1"; "C1" **@{-} ?(.5) *@{>},
"B2"; "C2" **@{-} ?(.5) *@{>},
"B3"; "C3" **@{-} ?(.45) *@{>},
"B4"; "C4" **@{-} ?(.2) *@{>},
"D1"; "A1" **@{-} ?(.35) *@{>},
"D2"; "A2" **@{-} ?(.67) *@{>},
"D3"; "A3" **@{-} ?(.4) *@{>},
"D4"; "A4" **@{-} ?(.9) *@{>},
"D5"; "D0" **@{-} ?(.55) *@{>},
"A1"; "A1"+(18,0)="aa" **@{-},
    "aa"+(4,0)="aa"; "B1" **@{-},
"A2"; "A2"+(8,0) **@{-},
    "A2"+(12,0); "B2"+(-12,0)="aa" **@{-},
    "B2"-(8,0); "B2" **@{-},
"A3"; "B3" **@{-},
"A4"; "A4"+(28,0)="aa" **@{-},
    "aa"+(4,0)="aa"; "aa"+(6,0)="aa" **@{-},
    "aa"+(4,0)="aa"; "aa"+(6,0)="aa" **@{-},
    "aa"+(4,0); "B4" **@{-},
"D1"="aa"; "aa"+(8,0)="aa" **@{-},
    "aa"+(4,0)="aa"; "aa"+(6,0)="aa" **@{-},
    "aa"+(4,0)="aa"; "aa"+(6,0)="aa" **@{-},
    "aa"+(4,0)="aa"; "D0" **@{-},
"D2"="aa"; "aa"+(18,0)="aa" **@{-},
    "aa"+(4,0)="aa"; "aa"+(6,0)="aa" **@{-},
    "aa"+(4,0)="aa"; "C1" **@{-},
"D3"="aa"; "aa"+(8,0)="aa" **@{-},
    "aa"+(4,0)="aa"; "aa"+(6,0)="aa" **@{-},
    "aa"+(4,0)="aa"; "aa"+(16,0)="aa" **@{-},
    "aa"+(4,0)="aa"; "C2" **@{-},
"D4"; "C3" **@{-},
"D5"="aa"; "C4"-(12,0)="aa" **@{-},
    "aa"+(4,0); "C4" **@{-},
"B1"+(-10,11)="aa", "aa" *!D{9}, 
"aa"+(0,-10)="aa" *!D{8},
"aa"+(0,-10)="aa" *!D{7},
"aa"+(0,-10)="aa" *!D{6},
"aa"+(-15,-20) *!D{5},
"aa"+(0,-30)="aa" *!D{4},
"aa"+(0,-10)="aa" *!D{3},
"aa"+(0,-10)="aa" *!D{2},
"aa"+(0,-10)="aa" *!D{1},
"D1"+(-.5,5)="aa" *!R{1}, 
"aa"+(10,0)="aa" *!R{2},
"aa"+(10,0)="aa" *!R{3},
"aa"+(10,0)="aa" *!R{4},
"aa"+(12,-10) *!L{5},
"aa"+(30,0)="aa" *!R{6},
"aa"+(10,0)="aa" *!R{7},
"aa"+(10,0)="aa" *!R{8},
"aa"+(10,0)="aa" *!R{9},
"D5"+(10,-30) *{\normalsize
    \begin{array}{l}
    \mbox{(c)\ \ $\PG(\pi_3)$ with}\\
    \pi_3=(5,3,4,1,2)\odot(8,7,9,6)
    \end{array}},
\end{xy}
$$

$$
\begin{xy}/r.4mm/:
(0,0)="or", (10,0)="hd", (0,-10)="vd", (10,-10)="dd",
"or"+(0,-20)="A1"="aa", 
    "aa"+"dd"="A2"="aa",
    "aa"+"dd"="A3"="aa",
    "aa"+"dd"="A4"="aa",
"A1"+(50,10)="B1"="aa", 
    "aa"+"dd"="B2"="aa",
    "aa"+"dd"="B3"="aa",
    "aa"+"dd"="B4"="aa",
    "aa"+"dd"="B5"="aa",
    "aa"+(0,40)="B0",
"B1"+(0,-60)="C1"="aa", 
    "aa"+"dd"="C2"="aa",
    "aa"+"dd"="C3"="aa",
    "aa"+"dd"="C4"="aa",
"A1"+(0,-50)="D1"="aa", 
    "aa"+"dd"="D2"="aa",
    "aa"+"dd"="D3"="aa",
    "aa"+"dd"="D4"="aa",
"B1"; "C1" **@{-} ?(.81) *@{>},
"B2"; "C2" **@{-} ?(.64) *@{>},
"B3"; "C3" **@{-} ?(.47) *@{>},
"B4"; "C4" **@{-} ?(.3) *@{>},
"B0"; "B1" **@{-}, "B5" **@{-} ?(.3) *@{<},
"D1"; "A1" **@{-} ?(.3) *@{>},
"D2"; "A2" **@{-} ?(.5) *@{>},
"D3"; "A3" **@{-} ?(.7) *@{>},
"D4"; "A4" **@{-} ?(.9) *@{>},
"A1"; "B2"-(12,0)="aa" **@{-},
    "aa"+(4,0)="aa"; "B2" **@{-},
"A2"; "B3"-(22,0)="aa" **@{-},
    "aa"+(4,0)="aa"; "aa"+(6,0)="aa" **@{-},
    "aa"+(4,0)="aa"; "B3" **@{-},
"A3"; "B4"-(32,0)="aa" **@{-},
    "aa"+(4,0)="aa"; "aa"+(6,0)="aa" **@{-},
    "aa"+(4,0)="aa"; "aa"+(6,0)="aa" **@{-},
    "aa"+(4,0)="aa"; "B4" **@{-},
"A4"; "B5"-(42,0)="aa" **@{-},
    "aa"+(4,0)="aa"; "aa"+(6,0)="aa" **@{-},
    "aa"+(4,0)="aa"; "aa"+(6,0)="aa" **@{-},
    "aa"+(4,0)="aa"; "aa"+(6,0)="aa" **@{-},
    "aa"+(4,0)="aa"; "B5" **@{-},
"D1"="aa"; "aa"+(8,0)="aa" **@{-},
    "aa"+(4,0)="aa"; "aa"+(6,0)="aa" **@{-},
    "aa"+(4,0)="aa"; "aa"+(6,0)="aa" **@{-},
    "aa"+(4,0)="aa"; "C1" **@{-},
"D2"="aa"; "aa"+(8,0)="aa" **@{-},
    "aa"+(4,0)="aa"; "aa"+(6,0)="aa" **@{-},
    "aa"+(4,0)="aa"; "C2" **@{-},
"D3"="aa"; "aa"+(8,0)="aa" **@{-},
    "aa"+(4,0)="aa"; "C3" **@{-},
"D4"; "C4" **@{-},
"B1"+(5,1)="aa" *!D{9}, 
"aa"+(-15,-10)="aa" *!D{8},
"aa"+(0,-10)="aa" *!D{7},
"aa"+(0,-10)="aa" *!D{6},
"aa"+(0,-10)="aa" *!D{5},
"aa"+(0,-20)="aa" *!D{4},
"aa"+(0,-10)="aa" *!D{3},
"aa"+(0,-10)="aa" *!D{2},
"aa"+(0,-10)="aa" *!D{1},
"D1"+(-.5,5)="aa", "aa" *!R{1}, 
"aa"+(10,0)="aa" *!R{2},
"aa"+(10,0)="aa" *!R{3},
"aa"+(10,0)="aa" *!R{4},
"aa"+(20,0)="aa" *!R{5},
"aa"+(10,0)="aa" *!R{6},
"aa"+(10,0)="aa" *!R{7},
"aa"+(10,0)="aa" *!R{8},
"aa"+(10,30)="aa" *!R{9},
"D4"+(10,-20) *{\normalsize
    \begin{array}{l}
    \mbox{(d)\ \ $\PG(\pi_4)$ with}\\
    \pi_4=(4,3,2,1,9)\odot(8,7,6,5)
    \end{array}},
\end{xy}
\qquad
\qquad
\begin{xy}/r.4mm/:
(0,0)="or", (10,0)="hd", (0,-10)="vd", (10,-10)="dd",
"or"+(0,-10)="A1"="aa", 
    "aa"+"dd"="A2"="aa",
    "aa"+"dd"="A3"="aa",
    "aa"+"dd"="A4"="aa",
"A1"+(50,0)="B1"="aa", 
    "aa"+"dd"="B2"="aa",
    "aa"+"dd"="B3"="aa",
    "aa"+"dd"="B4"="aa",
    "aa"+"dd"="B5"="aa",
    "aa"+(-40,0)="B0",
"B2"+(0,-50)="C1"="aa", 
    "aa"+"dd"="C2"="aa",
    "aa"+"dd"="C3"="aa",
    "aa"+"dd"="C4"="aa",
"A1"+(0,-60)="D1"="aa", 
    "aa"+"dd"="D2"="aa",
    "aa"+"dd"="D3"="aa",
    "aa"+"dd"="D4"="aa",
"B2"; "C1" **@{-} ?(.7) *@{<},
"B3"; "C2" **@{-} ?(.5) *@{<},
"B4"; "C3" **@{-} ?(.3) *@{<},
"B5"; "C4" **@{-} ?(.1) *@{<},
"B1"; "B0" **@{-} ?(.35) *@{<},
"D1"; "A1" **@{-} ?(.35) *@{<},
"D2"; "A2" **@{-} ?(.5) *@{<},
"D3"; "A3" **@{-} ?(.65) *@{<},
"D4"; "A4" **@{-} ?(.8) *@{<},
"A1"; "B1" **@{-},
"A2"; "B2"-(12,0)="aa" **@{-},
    "aa"+(4,0)="aa"; "B2" **@{-},
"A3"; "B3"-(22,0)="aa" **@{-},
    "aa"+(4,0)="aa"; "aa"+(6,0)="aa" **@{-},
    "aa"+(4,0)="aa"; "B3" **@{-},
"A4"; "B4"-(32,0)="aa" **@{-},
    "aa"+(4,0)="aa"; "aa"+(6,0)="aa" **@{-},
    "aa"+(4,0)="aa"; "aa"+(6,0)="aa" **@{-},
    "aa"+(4,0)="aa"; "B4" **@{-},
"B0"; "B5"-(32,0)="aa" **@{-},
    "aa"+(4,0)="aa"; "aa"+(6,0)="aa" **@{-},
    "aa"+(4,0)="aa"; "aa"+(6,0)="aa" **@{-},
    "aa"+(4,0)="aa"; "B5" **@{-},
"D1"="aa"; "aa"+(8,0)="aa" **@{-},
    "aa"+(4,0)="aa"; "aa"+(6,0)="aa" **@{-},
    "aa"+(4,0)="aa"; "aa"+(6,0)="aa" **@{-},
    "aa"+(4,0)="aa"; "C1" **@{-},
"D2"="aa"; "aa"+(8,0)="aa" **@{-},
    "aa"+(4,0)="aa"; "aa"+(6,0)="aa" **@{-},
    "aa"+(4,0)="aa"; "C2" **@{-},
"D3"="aa"; "aa"+(8,0)="aa" **@{-},
    "aa"+(4,0)="aa"; "C3" **@{-},
"D4"; "C4" **@{-},
"B1"+(-10,1)="aa" *!D{9}, 
"aa"+(0,-10)="aa" *!D{8},
"aa"+(0,-10)="aa" *!D{7},
"aa"+(0,-10)="aa" *!D{6},
"aa"+(15,-10)="aa" *!D{5},
"aa"+(-15,-20)="aa" *!D{4},
"aa"+(0,-10)="aa" *!D{3},
"aa"+(0,-10)="aa" *!D{2},
"aa"+(0,-10)="aa" *!D{1},
"D1"+(-.5,5)="aa", "aa" *!R{1}, 
"aa"+(10,0)="aa" *!R{2},
"aa"+(10,0)="aa" *!R{3},
"aa"+(10,0)="aa" *!R{4},
"aa"+(20,20)="aa" *!R{5},
"aa"+(10,-20)="aa" *!R{6},
"aa"+(10,0)="aa" *!R{7},
"aa"+(10,0)="aa" *!R{8},
"aa"+(10,0)="aa" *!R{9},
"D4"+(10,-20) *{\normalsize
    \begin{array}{l}
    \mbox{(e)\ \ $\PG(\pi_5)$ with}\\
    \pi_5=(9,8,7,6,5)\odot (4,3,2,1)
    \end{array}},
\end{xy}
$$

\caption{The petal grid diagrams $\PG(\pi_1),\ldots,\PG(\pi_5)$}
\label{F:SB}
\end{figure}
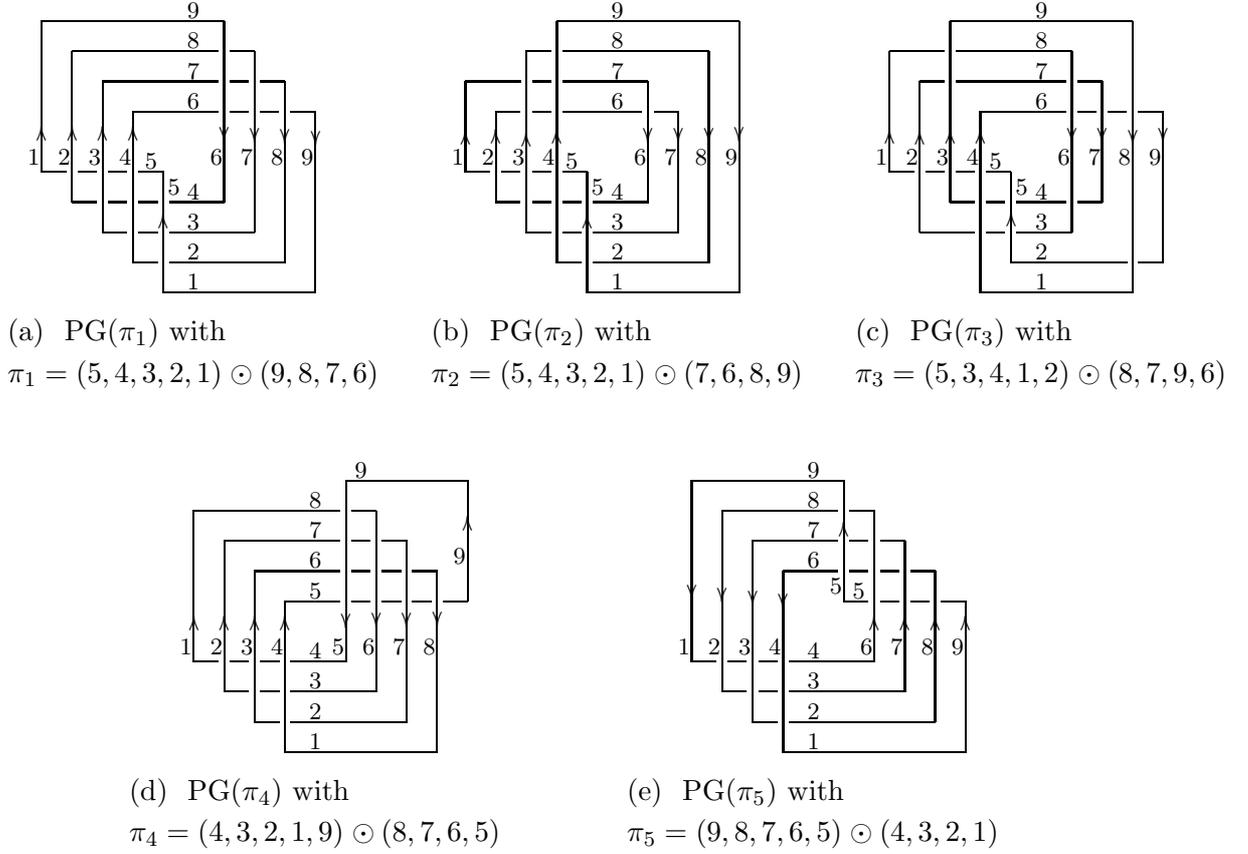

\begin{definition}[braid insertion]
Let $\pi$ be a braided petal permutation of length $2n+1$,
and let $1\le k\le n$.
For a $k$-braid $\alpha$, $\PG(\pi;\alpha)$ denotes the knot diagram obtained from $\PG(\pi)$
by inserting the braid $\alpha$ in the lowest $k$ strands as in Figure~\ref{F:ins}.
We call $\PG(\pi;\alpha)$ the \emph{insertion of $\alpha$ into $\PG(\pi)$}.
$K(\pi;\alpha)$ denotes the knot represented by $\PG(\pi;\alpha)$.
\end{definition}

\begin{remark}\label{rmk}
Let $\pi_n$ be the strongly braided petal permutation
$$
\pi_n=(n+1,n,\ldots,1)\odot(2n+1,2n,\ldots,n+2).
$$
For example,
$\pi_2=(3,2,1)\odot(5,4)$,
$\pi_3=(4,3,2,1)\odot(7,6,5)$
and $\pi_4=(5,4,3,2,1)\odot(9,8,7,6)$.
Then $K(\pi_n)$ is the torus knot $T_{n,n+1}$.
See Figure~\ref{F:ex2}(e) for $n=2$ and
Figure~\ref{F:SB}(a) for $n=4$.
Observe the following.
\begin{itemize}
\item[(i)] $T_{n,n+1}$ is isotopic to the closure of
the $n$-braid $\delta^{n+1}=\Delta^2\delta$.
The petal grid diagram $\PG(\pi_n)$
is a closed braid diagram for $T_{n,n+1}$.

\item[(ii)] For an $n$-braid $\alpha$,
$K(\pi_n;\alpha)$ is isotopic to the closure of the braid
$\delta^{n+1}\alpha=\Delta^2\delta\alpha$.
\end{itemize}
\end{remark}

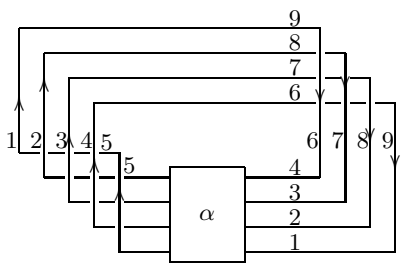
\begin{figure}\footnotesize
$$
\begin{xy}/r.33mm/:
(0,0)="or", (10,0)="hd", (0,-10)="vd", (10,-10)="dd", (60,0)="len",
(0,4)="vlen", (15,0)="hlen",
"or"+(0,0)="A1"="aa", 
    "aa"+"dd"="A2"="aa",
    "aa"+"dd"="A3"="aa",
    "aa"+"dd"="A4"="aa",
"A1"+"len"+(60,0)="B1"="aa", 
    "aa"+"dd"="B2"="aa",
    "aa"+"dd"="B3"="aa",
    "aa"+"dd"="B4"="aa",
"B1"+(0,-60)="C1"="aa", 
    "aa"+"dd"="C2"="aa",
    "aa"+"dd"="C3"="aa",
    "aa"+"dd"="C4"="aa",
"A1"+(0,-50)="D1"="aa", 
    "aa"+"dd"="D2"="aa",
    "aa"+"dd"="D3"="aa",
    "aa"+"dd"="D4"="aa",
    "aa"+"dd"="D5"="aa",
    "D1"+(40,0)="D0",
"D5"+(20,0)="LL1"="aa",
    "aa"+(0,10)="LL2"="aa",
    "aa"+(0,10)="LL3"="aa",
    "aa"+(0,10)="LL4"="aa",
    "LL1"+"hlen"+"hlen"="RR1"="aa",
    "aa"+(0,10)="RR2"="aa",
    "aa"+(0,10)="RR3"="aa",
    "aa"+(0,10)="RR4"="aa",
"B1"; "C1" **@{-} ?(.5) *@{>},
"B2"; "C2" **@{-} ?(.25) *@{>},
"B3"; "C3" **@{-} ?(.43) *@{>},
"B4"; "C4" **@{-} ?(.44) *@{>},
"D1"; "A1" **@{-} ?(.45) *@{>},
"D2"; "A2" **@{-} ?(.8) *@{>},
"D3"; "A3" **@{-} ?(.55) *@{>},
"D4"; "A4" **@{-} ?(.55) *@{>},
"D5"; "D0" **@{-} ?(.65) *@{>},
"A1"="aa"; "B1" **@{-},
"A2"="aa"; "B2"+(-12,0)="aa" **@{-},
    "aa"+(4,0)="aa"; "B2" **@{-},
"A3"="aa"; "B3"+(-22,0)="aa" **@{-},
    "aa"+(4,0)="aa"; "aa"+(6,0)="aa" **@{-},
    "aa"+(4,0)="aa"; "B3" **@{-},
"A4"="aa"; "B4"+(-32,0)="aa" **@{-},
    "aa"+(4,0)="aa"; "aa"+(6,0)="aa" **@{-},
    "aa"+(4,0)="aa"; "aa"+(6,0)="aa" **@{-},
    "aa"+(4,0)="aa"; "B4" **@{-},
"D1"="aa"; "aa"+(8,0)="aa" **@{-},
    "aa"+(4,0)="aa"; "aa"+(6,0)="aa" **@{-},
    "aa"+(4,0)="aa"; "aa"+(6,0)="aa" **@{-},
    "aa"+(4,0)="aa"; "D0" **@{-},
"D2"="aa"; "aa"+(8,0)="aa" **@{-},
    "aa"+(4,0)="aa"; "aa"+(6,0)="aa" **@{-},
    "aa"+(4,0)="aa"; "aa"+(6,0)="aa" **@{-},
    "aa"+(4,0)="aa"; "LL4" **@{-}, "RR4"; "C1" **@{-},
"D3"="aa"; "aa"+(8,0)="aa" **@{-},
    "aa"+(4,0)="aa"; "aa"+(6,0)="aa" **@{-},
    "aa"+(4,0)="aa"; "LL3" **@{-}, "RR3"; "C2" **@{-},
"D4"="aa"; "aa"+(8,0)="aa" **@{-},
    "aa"+(4,0)="aa"; "LL2" **@{-}, "RR2"; "C3" **@{-},
"D5"; "LL1" **@{-}, "RR1"; "C4" **@{-},
"B1"+(-10,1)="aa", "aa" *!D{9}, 
"aa"+(0,-10)="aa" *!D{8},
"aa"+(0,-10)="aa" *!D{7},
"aa"+(0,-10)="aa" *!D{6},
"aa"-"len"+(-15,-20) *!D{5},
"aa"+(0,-30)="aa" *!D{4},
"aa"+(0,-10)="aa" *!D{3},
"aa"+(0,-10)="aa" *!D{2},
"aa"+(0,-10)="aa" *!D{1},
"D1"+(-.5,5)="aa", "aa" *!R{1}, 
"aa"+(10,0)="aa" *!R{2},
"aa"+(10,0)="aa" *!R{3},
"aa"+(10,0)="aa" *!R{4},
"aa"+(12,-10) *!L{5},
"aa"+"len"+(30,0)="aa" *!R{6},
"aa"+(10,0)="aa" *!R{7},
"aa"+(10,0)="aa" *!R{8},
"aa"+(10,0)="aa" *!R{9},
"LL4"+"vlen";
"LL1"-"vlen" **@{-};
"RR1"-"vlen" **@{-};
"RR4"+"vlen" **@{-};
"LL4"+"vlen" **@{-},
"LL2"+(0,5)+"hlen" *{\alpha},
\end{xy}
$$
\caption{$\PG(\pi;\alpha)$ with $\pi=(5,4,3,2,1)\odot(9,8,7,6)$ and $\alpha\in B_4$}
\label{F:ins}
\end{figure}

\begin{definition}[stabilization of a petal permutation]
Let $\pi=(a_1,\ldots,a_p)$ be a petal permutation, and let $p=2n+1$ and $1\le k\le n$.
Then $a_j=k$ for some $j$, hence
$$\pi=(a_1,\ldots,a_{j-1},k,a_{j+1},\ldots, a_p).$$
For $1\le i\le p$, let $a_i'$ be the integer defined by
$a_i'=a_i$ if $a_i\le k$ and $a_i'=a_i+1$ if $a_i>k$.
The \emph{stabilization of $\pi$ at $k$}, denoted $s_k(\pi)$,
is the $(p+2)$-permutation defined by
$$
s_k(\pi)=(a_1',\ldots,a_{j-1}', k+1,p+2,k,a_{j+1}',\ldots,a_p').
$$
\end{definition}

In other words, $s_k(\pi)$ is obtained from $\pi$ by replacing $a_i\ne k$ with $a_i'$
and then replacing $k$ with $(k+1,p+2,k)$.
(Equivalently, $s_k(\pi)$ is obtained from $\pi=(a_1,\ldots,k,\ldots,a_p)$
by replacing $a_i$ with $a_i'$
and then inserting $(k+1,p+2)$ in the left of $k$.)

\begin{example}
In the following, we show the stabilization $s_3$.
For readability, we underlined the number 3 in $\pi$
and the numbers in $s_3(\pi)$ which replace 3.

Let $\pi=(6,11,5,10,4,9,\underline{3},8,2,7,1)=(6,5,4,\underline{3},2,1)\odot(11,10,9,8,7)$.
Then
$$
s_3(\pi)=(7,12,6,11,5,10,\underline{4,13,3},9,2,8,1)
=(7,6,5,\underline{4,3},2,1)\odot(12,11,10,\underline{13},9,8).
$$
See Figure~\ref{F:st0} for  $\PG(\pi)$ and $\PG(s_3(\pi))$.
(In $\PG(\pi)$, the horizontal edge with height 3 is drawn in double line.
In $\PG(s_3(\pi))$, the horizontal edges with heights 3, 4 and 13
and the vertical edges connecting these horizontal edges are drawn in double line.)

If $\pi=(4,8,\underline{3},7,2,6,1,9,5)=(4,\underline{3},2,1,5)\odot(8,7,6,9)$,
then
$$s_3(\pi)=(5,9,\underline{4,11,3},8,2,7,1,10,6)
=(5,\underline{4,3},2,1,6)\odot(9,\underline{11},8,7,10).
$$
\end{example}

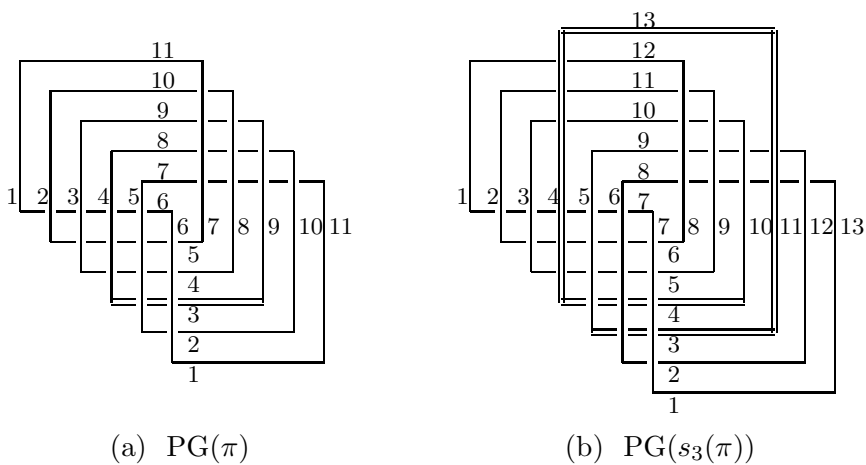
\begin{figure}\footnotesize
$$\begin{array}{ccc}
\begin{xy}/r.4mm/:
(0,0)="or", (10,0)="hd", (0,-10)="vd", (10,-10)="dd",
"or"+(10,-10)="A1"="aa", 
"aa"+"dd"="A2"="aa",
"aa"+"dd"="A3"="aa",
"aa"+"dd"="A4"="aa",
"aa"+"dd"="A5"="aa",
"or"+(70,-10)="B1"="aa", 
"aa"+"dd"="B2"="aa",
"aa"+"dd"="B3"="aa",
"aa"+"dd"="B4"="aa",
"aa"+"dd"="B5"="aa",
"B1"+(-10,-10)="B2L"="aa", 
"aa"+(0,-10)="B3L"="aa",
"aa"+(0,-10)="B4L"="aa",
"aa"+(0,-10)="B5L"="aa",
"or"+(70,-70)="C1"="aa", 
"aa"+"dd"="C2"="aa",
"aa"+"dd"="C3"="aa",
"aa"+"dd"="C4"="aa",
"aa"+"dd"="C5"="aa",
"or"+(10,-60)="D1"="aa", 
"aa"+"dd"="D2"="aa",
"aa"+"dd"="D3"="aa",
"aa"+"dd"="D4"="aa",
"aa"+"dd"="D5"="aa",
"aa"+"dd"="D6"="aa",
"D1"+(50,0)="D0",
"D0"+(10,0)="D1R"="aa", 
"aa"+(0,-10)="D2R"="aa",
"aa"+(0,-10)="D3R"="aa",
"aa"+(0,-10)="D4R"="aa",
"aa"+(0,-10)="D5R"="aa",
"C1"; "B1" **@{-},
"C2"; "B2" **@{-},
"C3"; "B3" **@{-},
"C4"; "B4" **@{-},
"C5"; "B5" **@{-},
"A1"; "D1" **@{-},
"A2"; "D2" **@{-},
"A3"; "D3" **@{-},
"A4"; "D4" **@{-},
"A5"; "D5" **@{-},
"A1"; "B1" **@{-},
"A2"; "B2L" **@{-},
"A3"; "B3L" **@{-},
"A4"; "B4L" **@{-},
"A5"; "B5L" **@{-},
"B2L"="aa"; "aa"+(8,0) **@{-},
"aa"+(10,0)="aa", "aa"+(2,0); "aa"+(10,0) **@{-},
"B3L"="aa"; "aa"+(8,0) **@{-},
"aa"+(10,0)="aa", "aa"+(2,0); "aa"+(8,0) **@{-},
"aa"+(10,0)="aa", "aa"+(2,0); "aa"+(10,0) **@{-},
"B4L"="aa"; "aa"+(8,0) **@{-},
"aa"+(10,0)="aa", "aa"+(2,0); "aa"+(8,0) **@{-},
"aa"+(10,0)="aa", "aa"+(2,0); "aa"+(8,0) **@{-},
"aa"+(10,0)="aa", "aa"+(2,0); "aa"+(10,0) **@{-},
"B5L"="aa"; "aa"+(8,0) **@{-},
"aa"+(10,0)="aa", "aa"+(2,0); "aa"+(8,0) **@{-},
"aa"+(10,0)="aa", "aa"+(2,0); "aa"+(8,0) **@{-},
"aa"+(10,0)="aa", "aa"+(2,0); "aa"+(8,0) **@{-},
"aa"+(10,0)="aa", "aa"+(2,0); "aa"+(10,0) **@{-},
"C2"; "D3R" **@{-},
"C3"; "D4R" **@{=}, 
"C4"; "D5R" **@{-},
"C5"; "D6" **@{-},
"D0"; "D6" **@{-},
"D1"="aa"; "aa"+(8,0) **@{-},
"aa"+(10,0)="aa", "aa"+(2,0); "aa"+(8,0) **@{-},
"aa"+(10,0)="aa", "aa"+(2,0); "aa"+(8,0) **@{-},
"aa"+(10,0)="aa", "aa"+(2,0); "aa"+(8,0) **@{-},
"aa"+(10,0)="aa", "aa"+(2,0); "aa"+(10,0) **@{-},
"D2"="aa"; "aa"+(8,0) **@{-},
"aa"+(10,0)="aa", "aa"+(2,0); "aa"+(8,0) **@{-},
"aa"+(10,0)="aa", "aa"+(2,0); "aa"+(8,0) **@{-},
"aa"+(10,0)="aa", "aa"+(2,0); "aa"+(8,0) **@{-},
"aa"+(10,0)="aa", "aa"+(2,0); "aa"+(10,0) **@{-},
"D3"="aa"; "aa"+(8,0) **@{-},
"aa"+(10,0)="aa", "aa"+(2,0); "aa"+(8,0) **@{-},
"aa"+(10,0)="aa", "aa"+(2,0); "aa"+(8,0) **@{-},
"aa"+(10,0)="aa", "aa"+(2,0); "aa"+(10,0) **@{-},
"D4"="aa"; "aa"+(8,0) **@{=},   
"aa"+(10,0)="aa", "aa"+(2,0); "aa"+(8,0) **@{=},
"aa"+(10,0)="aa", "aa"+(2,0); "aa"+(10,0) **@{=},
"D5"="aa"; "aa"+(8,0) **@{-},
"aa"+(10,0)="aa", "aa"+(2,0); "aa"+(10,0) **@{-},
"B1"+(-13,1)="aa", "aa" *!D{11}, 
"aa"+(0,-10)="aa" *!D{10},
"aa"+(0,-10)="aa" *!D{9},
"aa"+(0,-10)="aa" *!D{8},
"aa"+(0,-10)="aa" *!D{7},
"aa"+(0,-10)="aa" *!D{6},
"aa"+(10,-12.5)="aa" *!U{5},
"aa"+(0,-10)="aa" *!U{4},
"aa"+(0,-10)="aa" *!U{3},
"aa"+(0,-10)="aa" *!U{2},
"aa"+(0,-10)="aa" *!U{1},
"D1"+(-.5, 5)="aa", "aa" *!R{1}, 
"aa"+(10,0)="aa" *!R{2},
"aa"+(10,0)="aa" *!R{3},
"aa"+(10,0)="aa" *!R{4},
"aa"+(10,0)="aa" *!R{5},
"aa"+(12,-10)="aa" *!L{6},
"aa"+(10,0)="aa" *!L{7},
"aa"+(10,0)="aa" *!L{8},
"aa"+(10,0)="aa" *!L{9},
"aa"+(10,0)="aa" *!L{10},
"aa"+(10,0)="aa" *!L{11},
\end{xy}
&\qquad&
\begin{xy}/r.4mm/:
(0,0)="or", (10,0)="hd", (0,-10)="vd", (10,-10)="dd",
"or"+(10,-10)="A1"="aa", 
"aa"+"dd"="A2"="aa",
"aa"+"dd"="A3"="aa",
"A3"+(10,30)="A3n" ,
"A3n"+(0,-10)="A1R" ,
"A1R"+(0,-10)="A2R" ,
"A2R"+(0,-10)="A3R" ,
"aa"+"dd"+(10,0)="A4"="aa",
"aa"+"dd"="A5"="aa",
"or"+(80,-10)="B1"="aa", 
"aa"+"dd"="B2"="aa",
"aa"+"dd"="B3"="aa",
"B3"+(10,30)="B3n" ,
"aa"+"dd"+(10,0)="B4"="aa",
"aa"+"dd"="B5"="aa",
"B1"+(0,-10)="B2L"="aa", 
"aa"+(0,-10)="B3L"="aa",
"aa"+(0,-10)="B4L"="aa",
"aa"+(0,-10)="B5L"="aa",
"or"+(80,-70)="C1"="aa", 
"aa"+"dd"="C2"="aa",
"aa"+"dd"="C3"="aa",
"aa"+"dd"="C3n"="aa",
"aa"+"dd"="C4"="aa",
"aa"+"dd"="C5"="aa",
"or"+(10,-60)="D1"="aa", 
"aa"+"dd"="D2"="aa",
"aa"+"dd"="D3"="aa",
"aa"+"dd"="D3n"="aa",
"aa"+"dd"="D4"="aa",
"aa"+"dd"="D5"="aa",
"aa"+"dd"="D6"="aa",
"D1"+(60,0)="D0",
"D1"+(0,-10)="D2R"="aa",
"aa"+(0,-10)="D3R"="aa",
"aa"+(0,-10)="D4R"="aa",
"aa"+(0,-10)="D5R"="aa",
"C1"; "B1" **@{-},
"C2"; "B2" **@{-},
"C3"; "B3" **@{-},
"C4"; "B4" **@{-},
"C5"; "B5" **@{-},
"C3n"; "B3n" **@{=}, 
"A1"; "D1" **@{-},
"A2"; "D2" **@{-},
"A3"; "D3" **@{-},
"A4"; "D4" **@{-},
"A5"; "D5" **@{-},
"A3n"; "D3n" **@{=}, 
"A3n"; "B3n" **@{=},
"A1"; "A1R"+(-2,0) **@{-}, "A1R"+(2,0); "B1" **@{-},
"A2"; "A2R"+(-2,0) **@{-}, "A2R"+(2,0); "B2L"+(-2,0) **@{-},
    "B2L"+(2,0); "B2" **@{-},
"A3"; "A3R"+(-2,0) **@{-}, "A3R"+(2,0); "B3L"+(-2,0) **@{-},
    "B3L"="aa", "aa"+(2,0); "aa"+(8,0) **@{-},
    "aa"+(10,0)="aa", "aa"+(2,0); "aa"+(10,0) **@{-},
"A4"; "B4L"+(-2,0) **@{-},
    "B4L"="aa", "aa"+(2,0); "aa"+(8,0) **@{-},
    "aa"+(10,0)="aa", "aa"+(2,0); "aa"+(8,0) **@{-},
    "aa"+(10,0)="aa", "aa"+(2,0); "aa"+(8,0) **@{-},
    "aa"+(10,0)="aa", "aa"+(2,0); "aa"+(10,0) **@{-},
"A5"; "B5L"+(-2,0) **@{-},
    "B5L"="aa", "aa"+(2,0); "aa"+(8,0) **@{-},
    "aa"+(10,0)="aa", "aa"+(2,0); "aa"+(8,0) **@{-},
    "aa"+(10,0)="aa", "aa"+(2,0); "aa"+(8,0) **@{-},
    "aa"+(10,0)="aa", "aa"+(2,0); "aa"+(8,0) **@{-},
    "aa"+(10,0)="aa", "aa"+(2,0); "aa"+(10,0) **@{-},
"D0"; "D6" **@{-}; "C5" **@{-},
"D1"="aa"; "aa"+(8,0) **@{-},
    "aa"+(10,0)="aa", "aa"+(2,0); "aa"+(8,0) **@{-},
    "aa"+(10,0)="aa", "aa"+(2,0); "aa"+(8,0) **@{-},
    "aa"+(10,0)="aa", "aa"+(2,0); "aa"+(8,0) **@{-},
    "aa"+(10,0)="aa", "aa"+(2,0); "aa"+(8,0) **@{-},
    "aa"+(10,0)="aa", "aa"+(2,0); "aa"+(10,0) **@{-},
"D2"="aa"; "aa"+(8,0) **@{-},
    "aa"+(10,0)="aa", "aa"+(2,0); "aa"+(8,0) **@{-},
    "aa"+(10,0)="aa", "aa"+(2,0); "aa"+(8,0) **@{-},
    "aa"+(10,0)="aa", "aa"+(2,0); "aa"+(8,0) **@{-},
    "aa"+(10,0)="aa", "aa"+(2,0); "aa"+(8,0) **@{-},
    "aa"+(10,0)="aa", "aa"+(2,0); "aa"+(10,0) **@{-},
"D3"="aa"; "aa"+(8,0) **@{-},
    "aa"+(10,0)="aa", "aa"+(2,0); "aa"+(8,0) **@{-},
    "aa"+(10,0)="aa", "aa"+(2,0); "aa"+(8,0) **@{-},
    "aa"+(10,0)="aa", "aa"+(2,0); "aa"+(8,0) **@{-},
    "aa"+(10,0)="aa", "aa"+(2,0); "C2" **@{-},
"D3n"="aa"; "aa"+(8,0) **@{=},
    "aa"+(10,0)="aa", "aa"+(2,0); "aa"+(8,0) **@{=},
    "aa"+(10,0)="aa", "aa"+(2,0); "aa"+(8,0) **@{=},
    "aa"+(10,0)="aa", "aa"+(2,0); "C3" **@{=},
"D4"="aa"; "aa"+(8,0) **@{=},
    "aa"+(10,0)="aa", "aa"+(2,0); "aa"+(8,0) **@{=},
    "aa"+(10,0)="aa", "aa"+(2,0); "C3n" **@{=},
"D5"="aa"; "aa"+(8,0) **@{-},
    "aa"+(10,0)="aa", "aa"+(2,0); "C4" **@{-},
"B1"+(-13,1)="aa", "aa"+(0,10) *!D{13}, "aa" *!D{12},
"aa"+(0,-10)="aa" *!D{11},
"aa"+(0,-10)="aa" *!D{10},
"aa"+(0,-10)="aa" *!D{9},
"aa"+(0,-10)="aa" *!D{8},
"aa"+(0,-10)="aa" *!D{7},
"aa"+(10,-12.5)="aa" *!U{6},
"aa"+(0,-10)="aa" *!U{5},
"aa"+(0,-10)="aa" *!U{4},
"aa"+(0,-10)="aa" *!U{3},
"aa"+(0,-10)="aa" *!U{2},
"aa"+(0,-10)="aa" *!U{1},
"D1"+(-.5, 5)="aa", "aa" *!R{1}, 
"aa"+(10,0)="aa" *!R{2},
"aa"+(10,0)="aa" *!R{3},
"aa"+(10,0)="aa" *!R{4},
"aa"+(10,0)="aa" *!R{5},
"aa"+(10,0)="aa" *!R{6},
"aa"+(12,-10)="aa" *!L{7},
"aa"+(10,0)="aa" *!L{8},
"aa"+(10,0)="aa" *!L{9},
"aa"+(10,0)="aa" *!L{10},
"aa"+(10,0)="aa" *!L{11},
"aa"+(10,0)="aa" *!L{12},
"aa"+(10,0)="aa" *!L{13},
\end{xy}\\
\mbox{\normalsize (a)\ \ $\PG(\pi)$} &&
\mbox{\normalsize (b)\ \ $\PG(s_3(\pi))$\rule{0pt}{1.5em}}
\end{array}
$$
\caption{$\PG(\pi)$ and $\PG(s_3(\pi))$
for $\pi=(6,5,4,3,2,1)\odot(11,10,9,8,7)$}
\label{F:st0}
\end{figure}

The following lemma shows that $s_k(\pi)$ is easy to compute from $\pi$
when $\pi$ is strongly braided.
The proof is obvious from the definition of $s_k(\pi)$.

\begin{lemma}\label{L:stBr}
Let $p=2n+1$ and let $\pi=(n+1,n,\ldots,1)\odot(a_2,a_4,\ldots,a_{2n})$
be a strongly braided petal permutation of length $p$.
Then, for $1\le k\le n$, the stabilization $s_k(\pi)$ is also strongly braided
and of length $p+2$ such that
$$
s_k(\pi)=(n+2,n+1,\ldots,1)\odot(a_2+1,\ldots,a_{2(n-k+1)}+1, 2n+3,
\underbrace{a_{2(n-k+2)}+1,\ldots,a_{2n}+1}_{k-1}).
$$
In other words, the even terms of $s_k(\pi)$ is obtained from
$(a_2,a_4,\ldots,a_{2n})$ by increasing each term by 1
and then inserting $p+2=2n+3$ at the $k$th position
from the right.
\end{lemma}

\begin{proposition}\label{T:br}
Let $\pi$ be a strongly braided petal permutation of length $2n+1$.
Let $2\le k\le n$ and $\alpha\in B_k$.
Then $K(\pi;\alpha U_k)$ is isotopic to $K(s_k(\pi);\alpha)$.
\end{proposition}

The above proposition is proved by Lee and Jin~\cite{LJ21}
for the case where $\pi=(n+1,n,\ldots,1)\odot(2n+1,2n,\ldots,n+2)$,
$\alpha$ is the identity, $n$ is odd, and $k=\frac{n+1}2$.
Their proof works well for more general case as in the above proposition.
For readers' convenience, we include an idea of proof.

\begin{proof}[Idea of Proof]
Figure~\ref{F:st1} shows an example of the diagrams
$\PG(\pi;\alpha U_k)$ and $\PG(s_k(\pi);\alpha)$.
The diagram $\PG(s_k(\pi);\alpha)$ can be isotoped to the diagram in Figure~\ref{F:st2}(a)
by moving the highest horizontal edge to the bottom.
Then it can be isotoped to the diagram in Figure~\ref{F:st2}(b)
by moving the double lined vertical edges.
The diagram in Figure~\ref{F:st2}(b)
is isotopic to the diagram $\PG(\pi;\alpha U_k)$ in Figure~\ref{F:st1}(a).
\end{proof}

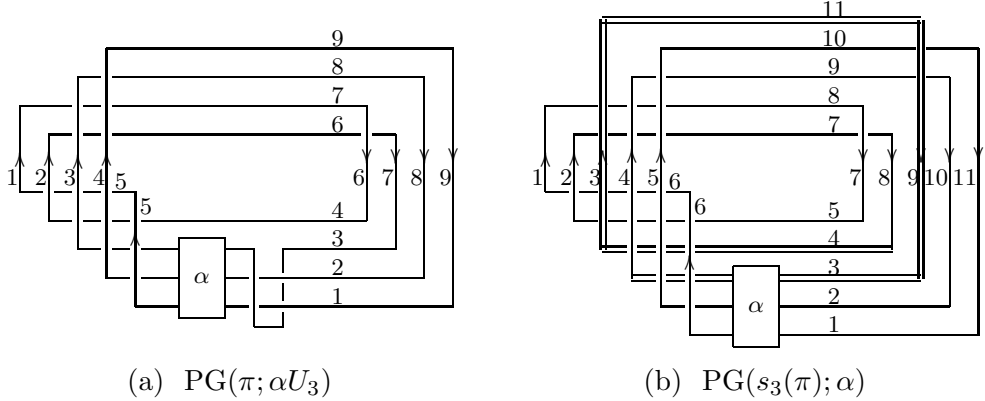
\begin{figure}\footnotesize
$$
\begin{array}{ccc}
\begin{xy}/r.38mm/:
(0,0)="or", (10,0)="hd", (0,-10)="vd", (10,-10)="dd", (60,0)="len",
(0,4)="vlen", (8,0)="hlen",
"or"+(0,-20)="A1"="aa", 
    "aa"+"dd"="A2"="aa",
    "aa"+(10,20)="A3"="aa",
    "aa"+(10,10)="A4"="aa",
"A1"+"len"+(60,0)="B1"="aa", 
    "aa"+"dd"="B2"="aa",
    "aa"+(10,20)="B3"="aa",
    "aa"+(10,10)="B4"="aa",
"B1"+(0,-40)="C1"="aa", 
    "aa"+"dd"="C2"="aa",
    "aa"+"dd"="C3"="aa",
    "aa"+"dd"="C4"="aa",
"A1"+(0,-30)="D1"="aa", 
    "aa"+"dd"="D2"="aa",
    "aa"+"dd"="D3"="aa",
    "aa"+"dd"="D4"="aa",
    "aa"+"dd"="D5"="aa",
    "D1"+(40,0)="D0",
"D5"+(15,0)="LL1"="aa",
    "aa"+(0,10)="LL2"="aa",
    "aa"+(0,10)="LL3"="aa",
    "LL1"+"hlen"+"hlen"="RR1"="aa",
    "aa"+(0,10)="RR2"="aa",
    "aa"+(0,10)="RR3"="aa",
"B1"; "C1" **@{-} ?(.5) *@{>},
"B2"; "C2" **@{-} ?(.25) *@{>},
"B3"; "C3" **@{-} ?(.43) *@{>},
"B4"; "C4" **@{-} ?(.44) *@{>},
"D1"; "A1" **@{-} ?(.45) *@{>},
"D2"; "A2" **@{-} ?(.8) *@{>},
"D3"; "A3" **@{-} ?(.55) *@{>},
"D4"; "A4" **@{-} ?(.55) *@{>},
"D5"; "D0" **@{-} ?(.65) *@{>},
"A1"="aa"; "aa"+(18,0)="aa" **@{-},
    "aa"+(4,0)="aa"; "aa"+(6,0)="aa" **@{-},
    "aa"+(4,0)="aa"; "B1" **@{-},
"A2"="aa"; "aa"+(8,0)="aa" **@{-},
    "aa"+(4,0)="aa"; "aa"+(6,0)="aa" **@{-},
    "aa"+(4,0)="aa"; "B2"-(12,0)="aa" **@{-},
    "aa"+(4,0)="aa"; "B2" **@{-},
"A3"="aa"; "aa"+(8,0)="aa" **@{-},
    "aa"+(4,0)="aa"; "B3" **@{-},
"A4"="aa"; "B4" **@{-},
"D1"="aa"; "aa"+(8,0)="aa" **@{-},
    "aa"+(4,0)="aa"; "aa"+(6,0)="aa" **@{-},
    "aa"+(4,0)="aa"; "aa"+(6,0)="aa" **@{-},
    "aa"+(4,0)="aa"; "D0" **@{-},
"D2"="aa"; "aa"+(8,0)="aa" **@{-},
    "aa"+(4,0)="aa"; "aa"+(6,0)="aa" **@{-},
    "aa"+(4,0)="aa"; "aa"+(6,0)="aa" **@{-},
    "aa"+(4,0)="aa"; "C1" **@{-},
"D3"="aa"; "aa"+(8,0)="aa" **@{-},
    "aa"+(4,0)="aa"; "aa"+(6,0)="aa" **@{-},
    "aa"+(4,0)="aa"; "LL3" **@{-},
    "RR3"="aa"; "aa"+(10,0)="aa" **@{-};
    "aa"+(0,-27)="aa" **@{-}; "aa"+(10,0)="aa" **@{-};
    "aa"+(0,5)="aa" **@{-}, "aa"+(0,4)="aa";
    "aa"+(0,6)="aa" **@{-}, "aa"+(0,4)="aa";
    "aa"+(0,8)="aa" **@{-}; "C2" **@{-},
"D4"="aa"; "aa"+(8,0)="aa" **@{-},
    "aa"+(4,0)="aa"; "LL2" **@{-}, "RR2"="aa";
    "aa"+(8,0)="aa" **@{-}, "aa"+(4,0)="aa";"C3" **@{-},
"D5"; "LL1" **@{-}, "RR1"="aa"; "aa"+(8,0)="aa" **@{-},
    "aa"+(4,0)="aa"; "C4" **@{-},
"B1"+(-10,21)="aa", "aa" *!D{9}, 
"aa"+(0,-10)="aa" *!D{8},
"aa"+(0,-10)="aa" *!D{7},
"aa"+(0,-10)="aa" *!D{6},
"aa"-"len"+(-15,-20) *!D{5},
"aa"+(0,-30)="aa" *!D{4},
"aa"+(0,-10)="aa" *!D{3},
"aa"+(0,-10)="aa" *!D{2},
"aa"+(0,-10)="aa" *!D{1},
"D1"+(-.5,5)="aa", "aa" *!R{1}, 
"aa"+(10,0)="aa" *!R{2},
"aa"+(10,0)="aa" *!R{3},
"aa"+(10,0)="aa" *!R{4},
"aa"+(12,-10) *!L{5},
"aa"+"len"+(30,0)="aa" *!R{6},
"aa"+(10,0)="aa" *!R{7},
"aa"+(10,0)="aa" *!R{8},
"aa"+(10,0)="aa" *!R{9},
"LL3"+"vlen";
"LL1"-"vlen" **@{-};
"RR1"-"vlen" **@{-};
"RR3"+"vlen" **@{-};
"LL3"+"vlen" **@{-},
"LL2"+"hlen" *{\alpha},
\end{xy}
&\quad&
\begin{xy}/r.38mm/:
(0,0)="or", (10,0)="hd", (0,-10)="vd", (10,-10)="dd", (50,0)="len",
(0,4)="vlen", (8,0)="hlen",
"or"+(0,-20)="A1"="aa", 
    "aa"+"dd"="A2"="aa",
    "aa"+(10, 40)="A3"="aa",
    "aa"+(10,-20)="A4"="aa",
    "aa"+(10, 10)="A5"="aa",
"A1"+"len"+(60,0)="B1"="aa", 
    "aa"+"dd"="B2"="aa",
    "aa"+(10, 40)="B3"="aa",
    "aa"+(10,-20)="B4"="aa",
    "aa"+(10, 10)="B5"="aa",
"B1"+(0,-40)="C1"="aa", 
    "aa"+"dd"="C2"="aa",
    "aa"+"dd"="C3"="aa",
    "aa"+"dd"="C4"="aa",
    "aa"+"dd"="C5"="aa",
"A1"+(0,-30)="D1"="aa", 
    "aa"+"dd"="D2"="aa",
    "aa"+"dd"="D3"="aa",
    "aa"+"dd"="D4"="aa",
    "aa"+"dd"="D5"="aa",
    "aa"+"dd"="D6"="aa",
    "D1"+(50,0)="D0",
"D6"+(15,0)="LL1"="aa",
    "aa"+(0,10)="LL2"="aa",
    "aa"+(0,10)="LL3"="aa",
    "LL1"+"hlen"+"hlen"="RR1"="aa",
    "aa"+(0,10)="RR2"="aa",
    "aa"+(0,10)="RR3"="aa",
"B1"; "C1" **@{-} ?(.5) *@{>},
"B2"; "C2" **@{-} ?(.25) *@{>},
"B3"; "C3" **@{=} ?(.55) *@{>},
"B4"; "C4" **@{-} ?(.37) *@{>},
"B5"; "C5" **@{-} ?(.39) *@{>},
"D1"; "A1" **@{-} ?(.45) *@{>},
"D2"; "A2" **@{-} ?(.8) *@{>},
"D3"; "A3" **@{=} ?(.43) *@{>},
"D4"; "A4" **@{-} ?(.64) *@{>},
"D5"; "A5" **@{-} ?(.61) *@{>},
"D6"; "D0" **@{-} ?(.55) *@{>},
"A1"="aa"; "aa"+(18,0)="aa" **@{-},
    "aa"+(4,0)="aa"; "aa"+(6,0)="aa" **@{-},
    "aa"+(4,0)="aa"; "aa"+(6,0)="aa" **@{-},
    "aa"+(4,0)="aa"; "B1" **@{-},
"A2"="aa"; "aa"+(8,0)="aa" **@{-},
    "aa"+(4,0)="aa"; "aa"+(6,0)="aa" **@{-},
    "aa"+(4,0)="aa"; "aa"+(6,0)="aa" **@{-},
    "aa"+(4,0)="aa"; "B2"-(12,0)="aa" **@{-},
    "aa"+(4,0)="aa"; "B2" **@{-},
"A3"; "B3" **@{=},
"A4"="aa"; "aa"+(8,0)="aa" **@{-},
    "aa"+(4,0)="aa"; "B4"-(12,0)="aa" **@{-},
    "aa"+(4,0)="aa"; "B4" **@{-},
"A5"="aa"; "B5"-(22,0)="aa" **@{-},
    "aa"+(4,0)="aa"; "B5" **@{-},
"D1"="aa"; "aa"+(8,0)="aa" **@{-},
    "aa"+(4,0)="aa"; "aa"+(6,0)="aa" **@{-},
    "aa"+(4,0)="aa"; "aa"+(6,0)="aa" **@{-},
    "aa"+(4,0)="aa"; "aa"+(6,0)="aa" **@{-},
    "aa"+(4,0)="aa"; "D0" **@{-},
"D2"="aa"; "aa"+(8,0)="aa" **@{-},
    "aa"+(4,0)="aa"; "aa"+(6,0)="aa" **@{-},
    "aa"+(4,0)="aa"; "aa"+(6,0)="aa" **@{-},
    "aa"+(4,0)="aa"; "aa"+(6,0)="aa" **@{-},
    "aa"+(4,0)="aa"; "C1" **@{-},
"D3"="aa"; "aa"+(8,0)="aa" **@{=},
    "aa"+(4,0)="aa"; "aa"+(6,0)="aa" **@{=},
    "aa"+(4,0)="aa"; "aa"+(6,0)="aa" **@{=},
    "aa"+(4,0)="aa"; "C2" **@{=},
"D4"="aa"; "aa"+(8,0)="aa" **@{=},
    "aa"+(4,0)="aa"; "aa"+(6,0)="aa" **@{=},
    "aa"+(4,0)="aa"; "LL3" **@{=},
    "RR3"="aa"; "C3" **@{=},
"D5"="aa"; "aa"+(8,0)="aa" **@{-},
    "aa"+(4,0)="aa"; "LL2" **@{-}, "RR2"="aa"; "C4" **@{-},
"D6"; "LL1" **@{-}, "RR1"="aa"; "C5" **@{-},
"B1"+(-10,31)="aa", "aa" *!D{11}, 
"aa"+(0,-10)="aa" *!D{10},
"aa"+(0,-10)="aa" *!D{9},
"aa"+(0,-10)="aa" *!D{8},
"aa"+(0,-10)="aa" *!D{7},
"aa"-"len"+(-5,-20) *!D{6},
"aa"+(0,-30)="aa" *!D{5},
"aa"+(0,-10)="aa" *!D{4},
"aa"+(0,-10)="aa" *!D{3},
"aa"+(0,-10)="aa" *!D{2},
"aa"+(0,-10)="aa" *!D{1},
"D1"+(-.5,5)="aa", "aa" *!R{1}, 
"aa"+(10,0)="aa" *!R{2},
"aa"+(10,0)="aa" *!R{3},
"aa"+(10,0)="aa" *!R{4},
"aa"+(10,0)="aa" *!R{5},
"aa"+(12,-10) *!L{6},
"aa"+"len"+(20,0)="aa" *!R{7},
"aa"+(10,0)="aa" *!R{8},
"aa"+(10,0)="aa" *!R{9},
"aa"+(10,0)="aa" *!R{10},
"aa"+(10,0)="aa" *!R{11},
"LL3"+"vlen";
"LL1"-"vlen" **@{-};
"RR1"-"vlen" **@{-};
"RR3"+"vlen" **@{-};
"LL3"+"vlen" **@{-},
"LL2"+"hlen" *{\alpha},
\end{xy}\\
\mbox{\normalsize (a)\ \ $\PG(\pi;\alpha U_3)$} &&
\mbox{\normalsize (b)\ \ $\PG(s_3(\pi); \alpha)$\rule{0pt}{1.5em}}
\end{array}
$$
\caption{$\PG(\pi;\alpha U_3)$ and $\PG(s_3(\pi); \alpha)$,
where $\alpha\in B_3$ and $\pi=(5,4,3,2,1)\odot(7,6,8,9)$}
\label{F:st1}
\end{figure}

\begin{figure}\footnotesize
$$\begin{array}{ccc}
\begin{xy}/r.38mm/:
(0,0)="or", (10,0)="hd", (0,-10)="vd", (10,-10)="dd", (70,0)="len",
(0,4)="vlen", (8,0)="hlen",
"or"+(0,-20)="A1"="aa", 
    "aa"+"dd"="A2"="aa",
    "aa"+(10, 40)="A3"="aa",
    "aa"+(10,-20)="A4"="aa",
    "aa"+(10, 10)="A5"="aa",
"A1"+"len"+(60,0)="B1"="aa", 
    "aa"+"dd"="B2"="aa",
    "aa"+(10, 40)="B3"="aa",
    "aa"+(10,-20)="B4"="aa",
    "aa"+(10, 10)="B5"="aa",
"B1"+(0,-40)="C1"="aa", 
    "aa"+"dd"="C2"="aa",
    "aa"+"dd"="C3"="aa",
    "aa"+"dd"="C4"="aa",
    "aa"+"dd"="C5"="aa",
"A1"+(0,-30)="D1"="aa", 
    "aa"+"dd"="D2"="aa",
    "aa"+"dd"="D3"="aa",
    "aa"+"dd"="D4"="aa",
    "aa"+"dd"="D5"="aa",
    "aa"+"dd"="D6"="aa",
    "D1"+(50,0)="D0",
"D6"+(15,0)="LL1"="aa",
    "aa"+(0,10)="LL2"="aa",
    "aa"+(0,10)="LL3"="aa",
    "LL1"+"hlen"+"hlen"="RR1"="aa",
    "aa"+(0,10)="RR2"="aa",
    "aa"+(0,10)="RR3"="aa",
"D3"+(0,-40)="A3",
"C3"+(0,-30)="B3",
"B1"; "C1" **@{-} ?(.5) *@{>},
"B2"; "C2" **@{-} ?(.25) *@{>},
"B3"; "C3" **@{=}, 
"B4"; "C4" **@{-} ?(.37) *@{>},
"B5"; "C5" **@{-} ?(.39) *@{>},
"D1"; "A1" **@{-} ?(.45) *@{>},
"D2"; "A2" **@{-} ?(.8) *@{>},
"D3"; "A3" **@{=}, 
"D4"; "A4" **@{-} ?(.64) *@{>},
"D5"; "A5" **@{-} ?(.61) *@{>},
"D6"; "D0" **@{-} ?(.55) *@{>},
"A1"="aa"; "aa"+(28,0)="aa" **@{-},
    "aa"+(4,0)="aa"; "aa"+(6,0)="aa" **@{-},
    "aa"+(4,0)="aa"; "B1" **@{-},
"A2"="aa"; "aa"+(18,0)="aa" **@{-},
    "aa"+(4,0)="aa"; "aa"+(6,0)="aa" **@{-},
    "aa"+(4,0)="aa"; "B2"-(12,0)="aa" **@{-},
    "aa"+(4,0)="aa"; "B2" **@{-},
"A3"; "B3" **@{=},
"A4"="aa"; "aa"+(8,0)="aa" **@{-},
    "aa"+(4,0)="aa"; "B4" **@{-},
"A5"; "B5" **@{-},
"D1"="aa"; "aa"+(8,0)="aa" **@{-},
    "aa"+(4,0)="aa"; "aa"+(16,0)="aa" **@{-},
    "aa"+(4,0)="aa"; "aa"+(6,0)="aa" **@{-},
    "aa"+(4,0)="aa"; "D0" **@{-},
"D2"="aa"; "aa"+(18,0)="aa" **@{-},
    "aa"+(4,0)="aa"; "aa"+(6,0)="aa" **@{-},
    "aa"+(4,0)="aa"; "aa"+(6,0)="aa" **@{-},
    "aa"+(4,0)="aa"; "C1" **@{-},
"D3"="aa"; "aa"+(8,0)="aa" **@{=},
    "aa"+(4,0)="aa"; "aa"+(6,0)="aa" **@{=},
    "aa"+(4,0)="aa"; "aa"+(6,0)="aa" **@{=},
    "aa"+(4,0)="aa"; "C2" **@{=},
"D4"="aa"; "aa"+(8,0)="aa" **@{=},
    "aa"+(4,0)="aa"; "aa"+(6,0)="aa" **@{=},
    "aa"+(4,0)="aa"; "LL3" **@{=},
    "RR3"="aa"; "C3" **@{=},
"D5"="aa"; "aa"+(8,0)="aa" **@{-},
    "aa"+(4,0)="aa"; "LL2" **@{-}, "RR2"="aa";
    "C4"-(13,0)="aa" **@{-},
    "aa"+(6,0)="aa"; "C4" **@{-},
"D6"; "LL1" **@{-}, "RR1"="aa";
    "C5"-(23,0)="aa" **@{-},
    "aa"+(6,0)="aa"; "C5" **@{-},
"B1"+(-10,21)="aa" *!D{10}, 
"aa"+(0,-10)="aa" *!D{9},
"aa"+(0,-10)="aa" *!D{8},
"aa"+(0,-10)="aa" *!D{7},
"aa"-"len"+(-5,-20) *!D{6},
"aa"+(0,-30)="aa" *!D{5},
"aa"+(0,-10)="aa" *!D{4},
"aa"+(0,-10)="aa" *!D{3},
"aa"+(0,-10)="aa" *!D{2},
"aa"+(0,-10)="aa" *!D{1},
"D1"+(-.5,5)="aa", "aa" *!R{1}, 
"aa"+(10,0)="aa" *!R{2},
"aa"+(10,-50) *!R{3},
"aa"+(20,0)="aa" *!R{4},
"aa"+(10,0)="aa" *!R{5},
"aa"+(12,-10) *!L{6},
"aa"+"len"+(20,0)="aa" *!R{7},
"aa"+(10,0)="aa" *!R{8},
"aa"+(10,-50)*!R{9},
"aa"+(20,0)="aa" *!R{10},
"aa"+(10,0)="aa" *!R{11},
"LL3"+"vlen";
"LL1"-"vlen" **@{-};
"RR1"-"vlen" **@{-};
"RR3"+"vlen" **@{-};
"LL3"+"vlen" **@{-},
"LL2"+"hlen" *{\alpha},
\end{xy}
&\quad&
\begin{xy}/r.38mm/:
(0,0)="or", (10,0)="hd", (0,-10)="vd", (10,-10)="dd", (70,0)="len",
(0,4)="vlen", (8,0)="hlen",
"or"+(0,-20)="A1"="aa", 
    "aa"+"dd"="A2"="aa",
    "aa"+(10, 40)="A3"="aa",
    "aa"+(10,-20)="A4"="aa",
    "aa"+(10, 10)="A5"="aa",
"A1"+"len"+(60,0)="B1"="aa", 
    "aa"+"dd"="B2"="aa",
    "aa"+(10, 40)="B3"="aa",
    "aa"+(10,-20)="B4"="aa",
    "aa"+(10, 10)="B5"="aa",
"B1"+(0,-40)="C1"="aa", 
    "aa"+"dd"="C2"="aa",
    "aa"+"dd"="C3"="aa",
    "aa"+"dd"="C4"="aa",
    "aa"+"dd"="C5"="aa",
"A1"+(0,-30)="D1"="aa", 
    "aa"+"dd"="D2"="aa",
    "aa"+"dd"="D3"="aa",
    "aa"+"dd"="D4"="aa",
    "aa"+"dd"="D5"="aa",
    "aa"+"dd"="D6"="aa",
    "D1"+(50,0)="D0",
"D6"+(15,0)="LL1"="aa",
    "aa"+(0,10)="LL2"="aa",
    "aa"+(0,10)="LL3"="aa",
    "LL1"+"hlen"+"hlen"="RR1"="aa",
    "aa"+(0,10)="RR2"="aa",
    "aa"+(0,10)="RR3"="aa",
"B1"; "C1" **@{-} ?(.5) *@{>},
"B2"; "C2" **@{-} ?(.25) *@{>},
"B4"; "C4" **@{-} ?(.37) *@{>},
"B5"; "C5" **@{-} ?(.39) *@{>},
"D1"; "A1" **@{-} ?(.45) *@{>},
"D2"; "A2" **@{-} ?(.8) *@{>},
"D4"; "A4" **@{-} ?(.64) *@{>},
"D5"; "A5" **@{-} ?(.61) *@{>},
"D6"; "D0" **@{-} ?(.55) *@{>},
"A1"="aa"; "aa"+(28,0)="aa" **@{-},
    "aa"+(4,0)="aa"; "aa"+(6,0)="aa" **@{-},
    "aa"+(4,0)="aa"; "B1" **@{-},
"A2"="aa"; "aa"+(18,0)="aa" **@{-},
    "aa"+(4,0)="aa"; "aa"+(6,0)="aa" **@{-},
    "aa"+(4,0)="aa"; "B2"-(12,0)="aa" **@{-},
    "aa"+(4,0)="aa"; "B2" **@{-},
"A4"="aa"; "aa"+(8,0)="aa" **@{-},
    "aa"+(4,0)="aa"; "B4" **@{-},
"A5"; "B5" **@{-},
"D1"="aa"; "aa"+(8,0)="aa" **@{-},
    "aa"+(4,0)="aa"; "aa"+(16,0)="aa" **@{-},
    "aa"+(4,0)="aa"; "aa"+(6,0)="aa" **@{-},
    "aa"+(4,0)="aa"; "D0" **@{-},
"D2"="aa"; "aa"+(18,0)="aa" **@{-},
    "aa"+(4,0)="aa"; "aa"+(6,0)="aa" **@{-},
    "aa"+(4,0)="aa"; "aa"+(6,0)="aa" **@{-},
    "aa"+(4,0)="aa"; "C1" **@{-},
"D4"="aa"; "aa"+(8,0)="aa" **@{=},
    "aa"+(4,0)="aa"; "aa"+(6,0)="aa" **@{=},
    "aa"+(4,0)="aa"; "LL3" **@{=},
    "RR3"="aa"; "aa"+(10,0)="aa" **@{=}; "aa"+(0,-30)="aa" **@{=};
    "aa"+(10,0)="aa" **@{=};
    "aa"+(0,7)="aa" **@{=}, "aa"+(0,6)="aa";
    "aa"+(0,4)="aa" **@{=}, "aa"+(0,6)="aa";
    "aa"+(0,17)="aa" **@{=};
    "C2" **@{=},
"D5"="aa"; "aa"+(8,0)="aa" **@{-},
    "aa"+(4,0)="aa"; "LL2" **@{-}, "RR2"="aa";
    "aa"+(7,0)="aa" **@{-},
    "aa"+(6,0)="aa"; "C4" **@{-},
"D6"; "LL1" **@{-}, "RR1"="aa";
    "aa"+(7,0)="aa" **@{-},
    "aa"+(6,0)="aa"; "C5" **@{-},
"B1"+(-10,21)="aa" *!D{10}, 
"aa"+(0,-10)="aa" *!D{9},
"aa"+(0,-10)="aa" *!D{8},
"aa"+(0,-10)="aa" *!D{7},
"aa"-"len"+(-5,-20) *!D{6},
"aa"+(0,-30)="aa" *!D{5},
"aa"+(0,-10)="aa" *!D{4},
"aa"+(0,-10)="aa", 
"aa"+(0,-10)="aa" *!D{2},
"aa"+(0,-10)="aa" *!D{1},
"D1"+(-.5,5)="aa", "aa" *!R{1}, 
"aa"+(10,0)="aa" *!R{2},
"aa"+(10,-50), 
"aa"+(20,0)="aa" *!R{4},
"aa"+(10,0)="aa" *!R{5},
"aa"+(12,-10) *!L{6},
"aa"+"len"+(20,0)="aa" *!R{7},
"aa"+(10,0)="aa" *!R{8},
"aa"+(10,-50), 
"aa"+(20,0)="aa" *!R{10},
"aa"+(10,0)="aa" *!R{11},
"LL3"+"vlen";
"LL1"-"vlen" **@{-};
"RR1"-"vlen" **@{-};
"RR3"+"vlen" **@{-};
"LL3"+"vlen" **@{-},
"LL2"+"hlen" *{\alpha},
\end{xy}\\
\mbox{\normalsize (a)\ \ } &&
\mbox{\normalsize (b)\ \ \rule{0pt}{1.5em}}
\end{array}
$$
\caption{$K(s_3(\pi); \alpha)$ is isotopic to $K(\pi;\alpha U_3)$}
\label{F:st2}
\end{figure}
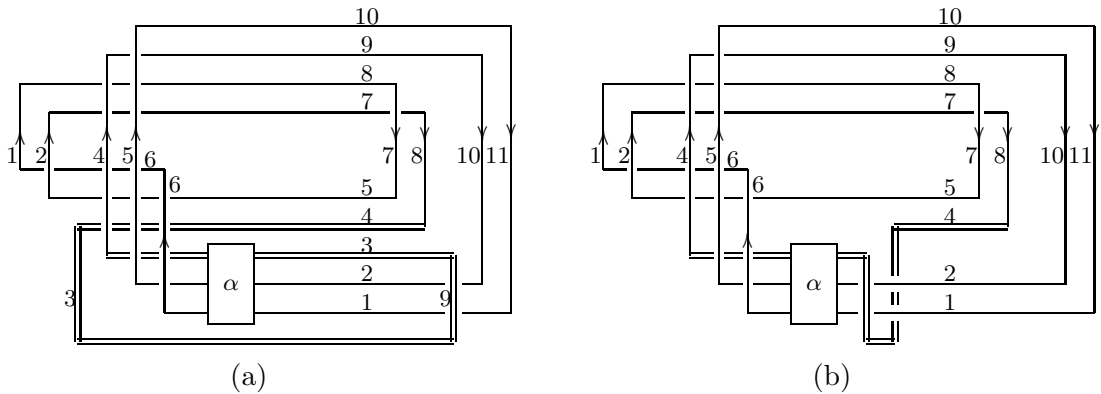

The following theorem is useful
for computing upper bounds for petal numbers.

\begin{theorem}\label{C:br}
Let $p=2n+1$ and $2\le k_1,\ldots,k_q\le n$.
If $\pi$ is a strongly braided petal permutation of length $p$, then
$$
p(K(\pi;U_{k_1}\cdots U_{k_q}))\le p+2q.
$$
\end{theorem}

\begin{proof}
Let $\beta=U_{k_1}\cdots U_{k_q}$.
Since $U_{k_1},\ldots, U_{k_q}$ mutually commute,
we have $\beta=U_{k_q}\cdots U_{k_1}$
and we may assume $k_1\ge\cdots\ge k_q$.
Let $\beta_i=U_{k_q}\cdots U_{k_i}$ for $1\le i\le q$,
and let $\beta_{q+1}$ be the identity braid.
Then each $\beta_i\in B_{k_i}$ and
$$
\beta_1=\beta,\qquad \beta_i=\beta_{i+1}U_{k_i}\quad (1\le i\le q-1),
\qquad \beta_q=\beta_{q+1}U_{k_q}=U_{k_q}.
$$
Let $\pi_1,\ldots,\pi_{q+1}$ be the petal permutations defined inductively
by $\pi_1=\pi$ and $\pi_{i+1}=s_{k_i}(\pi_{i})$ for $1\le i\le q$.
Since $\pi$ is strongly braided of length $p$,
each $\pi_i$ is strongly braided of length $p+2(i-1)$
for $1\le i\le q+1$ (by Lemma~\ref{L:stBr}).

By Proposition~\ref{T:br},
$K(\pi;\beta)=K(\pi_1;\beta_2U_{k_1})$ is isotopic to
$K(s_{k_1}(\pi_1);\beta_2)=K(\pi_2;\beta_3 U_{k_2})$.
By applying Proposition~\ref{T:br} again, it is isotopic to
$K(s_{k_2}(\pi_2);\beta_3)=K(\pi_3;\beta_4 U_{k_3})$.
Continuing this process, we conclude that
$K(\pi;\beta)$ is isotopic to $K(\pi_{q+1};\beta_{q+1})=K(\pi_{q+1})$.
Since $\pi_{q+1}$ has length $p+2q$, we are done.
\end{proof}

\begin{corollary}\label{T:11}
If $K$ is the closure of the $n$-braid $\Delta^2\delta U_{k_1}\cdots U_{k_q}$
for $2\le k_1,\ldots,k_q\le n$,
then
$$p(K)\le 2n+1+2q.$$
\end{corollary}

\begin{proof}
Let $\pi=(n+1,n,\ldots,1)\odot(2n+1,2n,\ldots,n+2)$.
Then $\PG(\pi)$ is a petal grid diagram of $T_{n,n+1}$
which is the closure of the braid $\delta^{n+1}=\Delta^2\delta$,
and $\PG(\pi;U_{k_1}\cdots U_{k_q})$ is a knot diagram of $K$.
Therefore $K=K(\pi;U_{k_1}\cdots U_{k_q})$
and it has petal number at most $2n+1+2q$ (by Theorem~\ref{C:br}).
\end{proof}

\begin{corollary}[Theorem A]
\label{T:1main}
Let $n$ and $s$ be relatively prime integers with $2\le n<s$.
Then
$$p(T_{n,s})\le 2s-2\left\lfloor\frac sn\right\rfloor +1.$$
\end{corollary}

\begin{proof}
The torus knot $T_{n,s}$ is the closure of the $n$-braid $\delta^s$.

Let $s=nm+k$ with $1\le k\le n-1$ and $m\ge 1$.
By Corollary~\ref{C:ns}, $\delta^s$ is conjugate to
$$
\delta (U_2\ldots U_n)^{m} U_{a_1}\ldots U_{a_{k-1}}
=\Delta^2 \delta (U_2\ldots U_n)^{m-1} U_{a_1}\ldots U_{a_{k-1}}
$$
for some $2\le a_1,\ldots,a_{k-1}\le n$.
The number of $U_i$'s in $(U_2\ldots U_n)^{m-1} U_{a_1}\ldots U_{a_{k-1}}$ is
$$
(n-1)(m-1)+(k-1)=nm-n-m+k=s-n-m.
$$
Therefore,
by Corollary~\ref{T:11}, $T_{n,s}$ has petal number at most
$$
2n+1+2(s-n-m)=2s-2m+1=2s-2\left\lfloor\frac sn\right\rfloor+1.
$$
\vskip-\baselineskip
\end{proof}

The following proposition was observed by Kim, No and Yoo
in the proof of Theorem~1 in \cite{KNY22}.
For readers' convenience, we bring their proof.

\begin{proposition}[\cite{KNY22}]
\label{pro:sb}
Let $n$ and $s$ be relatively prime integers with $2\le n<s<2n$.
Then
$$p(T_{n,s})\ge 2s-1.$$
\end{proposition}

\begin{proof}
In~\cite[Theorem 2]{KNY22},
Kim, No and Yoo proved  that for any knot $K$
$$
p(K)\ge 2 \sb(K)-1,
$$
where $\sb(K)$ denotes the super bridge index of $K$.
In~\cite[Theorem B]{Kui87}, Kuiper proved that
$$\sb(T_{n,s})=\min\{2n,s\}$$
for $2\le n<s$.
If $2\le n<s<2n$, then $\sb(T_{n,s})=\min\{2n,s\}=s$, hence
$p(T_{n,s})\ge 2\sb(T_{n,s})-1=2s-1$.
\end{proof}

\begin{corollary}[Theorem B]\label{T:2main}
Let $n$ and $s$ be relatively prime integers with $2\le n<s<2n$.
Then
$$p(T_{n,s})=2s-1.$$
\end{corollary}

\begin{proof}
By Proposition~\ref{pro:sb}, $p(T_{n,s})\ge 2s-1$.
By Corollary~\ref{T:1main}, $p(T_{n,s})\le 2s-1$
because $2\le n<s< 2n$ implies $\lfloor \frac sn\rfloor=1$.
\end{proof}

The following example shows how to compute a petal permutation
for a torus knot $T_{n,s}$.
In particular, if $2\le n<s<2n$, then we can see (by the above corollary)
that the petal permutation obtained in this way realizes the petal number $p(T_{n,s})$.

\begin{example}
Consider the torus knots $T_{5,6}$, $T_{5,7}$, $T_{5,8}$ and $T_{5,9}$.
Notice that $T_{5,6}=K(\pi)$, where
$$
\pi=(6,5,4,3,2,1)\odot(11,10,9,8,7).
$$

By Theorem~\ref{T:Br}, the 5-braid $\delta^2$ is conjugate to $\delta U_3$
(because $a_1=\lceil\frac{5}{2}\rceil=3$),
hence $\delta^7=\delta^5\cdot\delta^2$ is conjugate to $\Delta^2\delta U_3$.
Therefore,
by Proposition~\ref{T:br} and Lemma~\ref{L:stBr},
$T_{5,7}=K(s_3(\pi))$ and
$$
s_3(\pi)=(7,\ldots,1)\odot(12,11,10,\underline{13},9,8).
$$

Similarly, $\delta^3$ is conjugate to $\delta U_2U_4$
(because $a_1=\lceil\frac{5}{3}\rceil=2$, $a_2=\lceil\frac{5\cdot 2}{3}\rceil=4$),
hence $\delta^8=\delta^5\cdot\delta^3$ is conjugate to $\Delta^2\delta U_2U_4$.
Therefore
$T_{5,8}=K((s_2\circ s_4)(\pi))$ and
\begin{align*}
s_4(\pi)&=(7,\ldots,1)\odot(12,11,\underline{13},10,9,8),\\
s_2(s_4(\pi)) &=(8,\ldots,1)\odot(13,12,14,11,10,\underline{15},9).
\end{align*}

Similarly, $\delta^4$ is conjugate to $\delta U_2U_3U_4$
(because $a_1=\lceil\frac{5}{4}\rceil=2$, $a_2=\lceil\frac{5\cdot 2}{4}\rceil=3$,
$a_3=\lceil\frac{5\cdot 3}{4}\rceil=4$),
hence $\delta^9=\delta^5\cdot\delta^4$ is conjugate to $\Delta^2\delta U_2U_3U_4$.
Therefore $T_{5,9}=K((s_2\circ s_3\circ s_4)(\pi))$ and
\begin{align*}
s_4(\pi)&=(7,\ldots,1)\odot(12,11,\underline{13},10,9,8),\\
s_3(s_4(\pi)) &=(8,\ldots,1)\odot(13,12,14,11,\underline{15},10,9),\\
s_2(s_3(s_4(\pi))) &=(9,\ldots,1)\odot(14,13,15,12,16,11,\underline{17},10).
\end{align*}

\end{example}

\end{document}